\documentclass[11pt,reqno]{amsart}
\usepackage{fullpage}
\usepackage[dvipsnames]{xcolor}
\usepackage[T1]{fontenc}                          
\usepackage{amsfonts}
\usepackage[utf8]{inputenc}                       
\usepackage{comment}                              
\usepackage{mparhack}                             
\usepackage{amsmath,amssymb,amsthm,mathrsfs,eucal}
\usepackage{amsfonts}
\usepackage{enumerate}
\usepackage{enumitem}
\usepackage{booktabs}                             
\usepackage{graphicx,subfig}   
\usepackage{cancel}
\usepackage{wrapfig}                            
\usepackage[bookmarks=true, colorlinks=true]{hyperref}
\hypersetup{urlcolor=blue,citecolor=red,linkcolor=black}
\usepackage{cleveref}
\usepackage{soul}
\usepackage[normalem]{ulem}  

\usepackage{appendix}

\usepackage{bm}                                   
\usepackage{dsfont}
\usepackage[sort]{cite}

\newtheorem{defn}{Definition}[section]
\newtheorem{thm}{Theorem}[section]
\newtheorem{prop}{Proposition}[section]
\newtheorem{lem}{Lemma}[section]
\newtheorem{cor}{Corollary}[section]

\newtheorem{rem}{Remark}[section]

\numberwithin{equation}{section}

\DeclareMathOperator*{\argmin}{argmin}

\newcommand{\mP}{{\mathcal{P}}}

\newcommand{\mh}{\mathcal{H}}

\newcommand{\mn}{\mathcal{N}}

\newcommand{\mum}{\mathcal{U}^m}
\newcommand{\mume}{\mathcal{U}^m_\varepsilon}

\newcommand{\diff}{\mathop{}\!\mathrm{d}}

\newcommand{\mprd}{{\mathcal{P}(\Rd)}}

\newcommand{\mptrd}{{\mathcal{P}_2(\Rd)}}

\newcommand{\mpdtard}{{\mathcal{P}_{2}^a(\Rd)}}

\newcommand{\rhot}{\rho_{\tau}}
\newcommand{\rhote}{\rho_{\tau}^{\varepsilon}}
\newcommand{\rhotes}{\rho_{\tau}^{\varepsilon,\sigma}}

\newcommand{\rhotne}{\rho_{\tau}^{\varepsilon,n}}
\newcommand{\rhotnes}{\rho_{\tau}^{\varepsilon,\sigma,n}}

\newcommand{\rhotnn}{\rho_{\tau}^{n+1}}
\newcommand{\rhotnne}{\rho_{\tau}^{\varepsilon,n+1}}
\newcommand{\rhotnnes}{\rho_{\tau}^{\varepsilon,\sigma,n+1}}

\newcommand{\R}{\mathbb{R}}
\newcommand{\Rd}{{\mathbb{R}^{d}}}
\newcommand{\Rn}{{\mathbb{R}^{n}}}
\newcommand{\Rdd}{{\mathbb{R}^{2d}}}

\newcommand{\eps}{\varepsilon}

\newcommand{\rhoe}{\rho^\varepsilon}
\newcommand{\rhoes}{\rho^{\varepsilon,\sigma}}

\newcommand{\supp}{\mathrm{supp}}

\makeatletter
\@namedef{subjclassname@2020}{\textup{2020} Mathematics Subject Classification}
\makeatother

\def\XXint#1#2#3{{\setbox0=\hbox{$#1{#2#3}{\int}$}
  \vcenter{\hbox{$#2#3$}}\kern-.5\wd0}}

\title{Nonlocal particle approximation for linear and fast diffusion equations}
\date{}

\begin{document}

\author{José Antonio Carrillo \and Antonio Esposito \and Jakub Skrzeczkowski \and Jeremy Sheung-Him Wu}
\address{J. A. Carrillo, A. Esposito, J. Skrzeczkowski -- Mathematical Institute, University of Oxford, Woodstock Road, Oxford, OX2 6GG, United Kingdom.}

\email{carrillo@maths.ox.ac.uk}
\email{antonio.esposito@maths.ox.ac.uk}
\email{jakub.skrzeczkowski@maths.ox.ac.uk}

\address{J. S. H. Wu -- Mathematical Sciences Building, University of California, Los Angeles, CA 90095, United States.}

\email{jeremywu@math.ucla.edu}

\begin{abstract} We construct deterministic particle solutions for linear and fast diffusion equations using a nonlocal approximation. We exploit the $2$-Wasserstein gradient flow structure of the equations in order to obtain the nonlocal approximating PDEs by regularising the corresponding internal energy with suitably chosen mollifying kernels, either compactly or globally supported. Weak solutions are obtained by the JKO scheme. From the technical point of view, we improve known commutator estimates, fundamental in the nonlocal-to-local limit, to include globally supported kernels which, in particular cases, allow us to justify the limit without any further perturbation needed. Furthermore, we prove geodesic convexity of the nonlocal energies in order to prove convergence of the particle solutions to the nonlocal equations towards weak solutions of the local equations. We overcome the crucial difficulty of dealing with the singularity of the first variation of the free energies at the origin. As a byproduct, we provide convergence rates expressed as a scaling relationship between the number of particles and the localisation parameter. The analysis we perform leverages the fact that globally supported kernels yield a better convergence rate compared to compactly supported kernels. Our result is relevant in statistics, more precisely in sampling Gibbs and heavy-tailed distributions.
\end{abstract}







\maketitle

\section{Introduction}
In this manuscript, we continue the development of nonlocal approximations of diffusion equations studied in~\cite{BE22,CEW_nl_to_local_24,Patacchini_blob19}, to name a few references. We focus on linear and fast diffusion PDEs, where loss of regularity near the origin (at the level of the internal energy) needs a deeper analysis. The class of equations we have in mind reads
\begin{equation}\label{eq:general_diffusion_nonlinear+advection}
\partial_t\rho=\Delta \rho^m + \mbox{div}(\rho \, \nabla (V+W\ast \rho)),
\end{equation}
where $d/(d+2)<m\le1$, $V:\R^d \to \R$ is a smooth external potential satisfying $D^2 V \geq \lambda \, \mathbb{I},\, \lambda \in \R$, and $W:\R^d \to \R$ is a smooth interaction potential. The restriction $m>d/(d+2)$ is natural from the viewpoint of solutions as curves of probability densities with bounded second moment.
Nonlocal approximations of the diffusion part allow to interpret \eqref{eq:general_diffusion_nonlinear+advection} as a continuity equation which leads directly to deterministic particle methods approximating \eqref{eq:general_diffusion_nonlinear+advection}.

\smallskip
This is important in statistics in sampling Gibbs distributions of the form $\bar \rho\propto e^{-V(x)}$ when $W=0$, or to nonlinear Gibbs measures of the form $\bar \rho\propto e^{-(V+W\ast \bar\rho)(x)}$ with $\bar \rho$ the unique steady state of \eqref{eq:general_diffusion_nonlinear+advection} under suitable assumptions on $V$ and $W$, see e.g.~\cite{CMcCV03}. Furthermore, the case $0<m<1$ leads to heavy-tailed distributions~\cite{HP85,DD02,CV03}, whose sampling is a recent research topic~\cite{MR4723893}. In the case $m=1$, a classical approach for Gibbs measures consists in simulating the so-called over-damped Langevin dynamics, \cite{MR615176},
$$
{\diff} X(t) = -\nabla V(X(t))\diff t + \sqrt{2}\,  \diff B(t)\,,
$$
where $X(t)$ is the position of a particle at time $t$ and $B(t)$ is the Brownian motion. This allows to approximate $e^{-V(x)}$ since the probability density of $X(t)$ solves
\begin{equation}\label{eq:Fokker-Planck}
\partial_t \rho = \Delta \rho + \mbox{div}(\rho \, \nabla V),
\end{equation}
whose unique stationary state corresponds to the desired probabilty distribution $\bar \rho\propto e^{-V(x)}$. The Langevin dynamics have received increasing attention from the mathematical community, cf. for instance~\cite{MR3857306, MR4003567, MR3960927}, but its inherent difficulty is reconciling computationally expensive stochastic movements. Therefore, it is reasonable to find purely deterministic particle dynamics approximating \eqref{eq:Fokker-Planck} which is the content of the current paper. Our result is an alternative to other deterministic method to sample Gibbs distributions, for instance: the so-called Stein variational gradient descent, recently introduced in~\cite{liu2016stein} and based on a nonlocal PDE, see~\cite{carrillo2023convergence, MR4582478, korba2020non} and the references therein, and gradient descent algorithms of the Kullback-Leibler divergence as in~\cite{huix_korba_durmus_moulines_24}. We point out we could eventually add to the right-hand side of~\eqref{eq:general_diffusion_nonlinear+advection} the nonlinear diffusion term, $\kappa\,\Delta \rho^{\tilde{m}}$, with $\tilde{m} > 1$ and $\kappa\geq 0$, and analyse the sampling of nonlinear unique steady states of the corresponding problem under suitable assumptions on $\tilde{m}$, $V$, and $W$. A final motivation to develop deterministic particle methods for \eqref{eq:general_diffusion_nonlinear+advection} comes from mathematical modelling of tissue growth or cell adhesion in development biology as in \cite{CarrilloMurakawaCellAdhesion, zoology, Holzinger_deriv_cross_diff, Hecht2023porous}, for instance.

\smallskip
For ease of presentation, we shall not consider the external and interaction potentials, i.e. $V=W=0$ in \eqref{eq:general_diffusion_nonlinear+advection}, as these are easily included in deterministic particle methods, see \cite{Patacchini_blob19} for instance. Hence, we will only focus on the nonlinearity. To the best of our knowledge, the first ideas on how to approximate diffusion PDEs were given in the works of Lions and MasGallic, \cite{masgallic}, as well as Oelschl{\"a}ger, \cite{Ol}. Their main idea was to introduce convolutions in the drift term which approximates the nonlinear diffusion terms as interaction potentials concentrating at zero. Considering, for example, the quadratic porous medium equation, $\partial_t \rho = \Delta \rho^2$, one can introduce its nonlocal approximation
\begin{equation}\label{eq:nonlocal_quadratic_PME}
\partial_t \rho^{\eps} = \mbox{div}(\rho^{\eps} \, \nabla \rho^{\eps} \ast V_{\eps} \ast V_{\eps}),
\end{equation}
where $V_{\eps}$ is the usual mollifying kernel approaching the Dirac Delta at zero when $\eps\to 0$. Notice that the equation above is~\eqref{eq:general_diffusion_nonlinear+advection} without the diffusion term, $V=0$ and $W=V_{\eps} \ast V_{\eps}$. The advantage of considering equation \eqref{eq:nonlocal_quadratic_PME} is that it can be solved by the following particle method: if $\rho^{\eps}_0 = \sum_{i=1}^N \delta_{X_i(0)}$, then $\rho^{\eps}_t = \sum_{i=1}^N \delta_{X_i(t)}$ solves in the sense of distributions~\eqref{eq:nonlocal_quadratic_PME} with positions $X_i(t)$ given by solutions to the ODEs
$$
\frac{\diff}{\diff t}X_i(t) = - \frac{1}{N} \sum_{j \neq i} \nabla V_{\eps}\ast V_{\eps}(X_i(t) - X_j(t)),\quad i=1,\dots, N.
$$
In the context of Wasserstein gradient flows, the nonlocal regularisation~\eqref{eq:nonlocal_quadratic_PME} is equivalent to the regularisation of the free energy 
$$
\mathcal{F}^2[\rho] = \int_{\R^d} \rho(x)^2 \, \diff x, \qquad \mathcal{F}^2_{\eps}[\rho] = \int_{\R^d} (V_{\eps}\ast\rho(x))^2 \, \diff x = \int_{\R^d} (V_{\eps} \ast V_{\eps})\ast\rho(x) \, \diff \rho(x),
$$
where we used that $V_\varepsilon$ is even. For general nonlinear diffusion problems, there are two possible regularisation of the free energy:
\begin{equation}\label{eq:functional_F}
\mathcal{F}^m_{\eps}[\rho] = \begin{cases}
\displaystyle \frac{1}{m-1}\, \int_{\R^d} |V_{\eps}\ast\rho|^{m-1} \diff \rho &\mbox{ if } m> 1,\\[3mm]
\displaystyle \int_{\R^d } \log(V_\eps\ast\rho) \diff \rho &\mbox{ if } m= 1
\end{cases}
\end{equation}
and
\begin{equation}\label{eq:functional_U}
\mathcal{U}^m_{\eps}[\rho] = \begin{cases}
\displaystyle 
\frac{1}{m-1}\,  \int_{\R^d} |V_{\eps}\ast\rho|^{m} \diff x &\mbox{ if } m> 1,\\[3mm]
\displaystyle  \int_{\R^d } (V_\eps\ast\rho)\log(V_\eps\ast\rho) \diff x &\mbox{ if } m= 1.
\end{cases}
\end{equation}
Depending on the choice of the regularisation one obtains different nonlocal equations, providing different nonlocal approximation of nonlinear diffusion equations. This procedure is also known as \textit{the blob method}, first introduced in \cite{Patacchini_blob19} for diffusion equations with the addition of local and nonlocal drifts, and used in the previous work~\cite{Craig_blob2016} to approximate nonlocal equations with singular kernels by smooth kernels.
The first regularisation \eqref{eq:functional_F} has been studied in~\cite{Patacchini_blob19}, which gave complete convergence analysis for $m=2$ and conditional results for $m>2$. One of the main difficulties relies on the fact that the dissipation of the free energy $\mathcal{F}^m_{\eps}$ does not give sufficient estimates to pass to the limit $\eps \to 0$. From this point of view, the dissipation of $\mathcal{U}^m_{\eps}$ is easier to analyse and recently, \cite{BE22,CEW_nl_to_local_24} exploited it to provide a particle scheme for $m > 1$, with $m=1$ left open. The result is provided up to requiring the initial datum has finite second order moment and log-entropy. The recent work~\cite{blob_weighted_craig} deals with a weighted quadratic porous medium equation, where the weight is a target probability measure to be approximated from specific samples drawn from it. The blob method provides a deterministic particle approximation for the weighted porous medium equation, and, as a byproduct, a strategy to quantise a target $\bar\rho$ in the long-time behaviour. In one space dimension,~\cite{Daneri_Radici_Runa_JDE} introduces a deterministic particle approximation for aggregation-diffusion equations, including the porous medium equation for the subquadratic ($1<m<2$) and superquadratic ($m>2$) cases. This approach is not variational and it is limited to one space dimension. We also mention the recent work~\cite{difrancesco2024_pme_morse} where the one-dimensional quadratic porous medium equation is obtained from interacting particles subject to the repulsive Morse potential. Furthermore, the recent work~\cite{art:rate_pme_amassad_zhou} provides the rate of convergence for the nonlocal-to-local limit for the quadratic porous media equation. Several authors have extended these ideas to systems \cite{BE22,Hecht2023porous}, higher-order PDEs \cite{elbar2023limit}, the kinetic Landau equation \cite{CHWW20,CDDW20}, and optimal control problems \cite{craig2024blob}. It is worth to observe that numerical particle methods have been proposed in the literature. In particular we mention the two simultaneous results for linear diffusion $(m=1)$ \cite{Degond_Mustieles90,Russo90}. In one dimension and for nonlinear diffusions, we refer to~\cite{GosseToscani06,CM09,CRW16}: their approach is based on the PDE satisfied by the transporting maps. The survey~\cite{CMWreview} provides further details on most of the available numerical methods for these families of equations. Among the works proposing a stochastic approach we refer to~\cite{Ol,oelschlaeger2001,morale2005interacting,FigPhi2008,Holzinger_deriv_cross_diff} and the references therein.

\smallskip
To the best of our knowledge the only manuscript that addresses a deterministic particle approximation of the linear diffusion term is the recent work~\cite{craig2023nonlocal}, which, in fact, covers a more general class of equations, including fast diffusion. The authors propose to approximate the general class of energies 
$$
\mathcal{F}^f[\rho] = \int_{\Rd} f(\rho) \diff x \approx \int_{\Rd} f_{\sigma}(V_{\eps}\ast\rho) \diff x =: \mathcal{F}^f_{\sigma,\eps}[\rho],
$$
where $f$ is convex and lower semicontinuous while $f_{\sigma}$ is the Moreau-Yosida approximation of $f$ with parameter $\sigma$. It is well-known that $f_\sigma \to f$ when $\sigma \to 0$ and $f_{\sigma}$ is strictly convex,~\cite[Chapters 12-14]{MR2798533}. Under a specific scaling between $\sigma$ and $\eps$, the authors obtain convergence of approximations and identify the limit via compactness and duality arguments. 
They base their proof on the formal equivalence between the 2-Wasserstein gradient flow of $\mathcal{F}^f[\rho]$ and the $H^{-1}$-gradient flow of $\mathcal{F}^f[\rho]$ interpretations for nonlinear diffusion equations as
$$
\partial_t\rho= \mbox{div} (\rho \nabla f'(\rho))=\Delta \left[ f^*(f'(\rho))\right],
$$
where $f^*$ is the Legendre-Fenchel transform of $f$. While~\cite{craig2023nonlocal} provides the first proof of convergence of a nonlocal approximation for the heat equation, it does not provide a rate of convergence nor is it applicable to more natural regularisations for computational purposes such as \eqref{eq:functional_F} and \eqref{eq:functional_U} --- a difficulty mentioned in~\cite{craig2023nonlocal}. We also mention the recent paper \cite{crestetto2023deterministic} concerning regularisation \eqref{eq:functional_U} (with a particular choice of $V$ being Gaussian) for the kinetic Fokker-Planck equation and studying its numerical performance. 

\smallskip
Concerning~\eqref{eq:general_diffusion_nonlinear+advection}, let us focus on the case $m=1$ first. As aforementioned, the main challenge is to approximate the linear diffusion term $\Delta \rho$, thus we consider the heat equation
\begin{equation}\label{eq:heat_eq}
\partial_t\rho=\Delta \rho, \tag{HE}
\end{equation}
which is the $2$-Wasserstein gradient flow of the (extended) energy functional
\begin{equation}\label{eq:log_entropy}
    \mh[\rho]=
    \begin{cases}
\displaystyle\int_\Rd \rho(x)\log\rho(x) \diff x \qquad &\rho\log\rho\in L^1(\Rd),\\[2mm]
    +\infty &\mbox{otherwise}.
    \end{cases}
\end{equation}
We consider the regularised version
\begin{equation}\label{eq:log_entropy_reg}
    \mh^\varepsilon[\rho]:= \int_\Rd (V_\varepsilon*\rho)(x)\log\left(V_\varepsilon*\rho(x)\right)\,\diff x,
\end{equation}
leading to the nonlocal PDE
\begin{equation}\label{eq:nonlocal_he}
    \partial_t\rho=\nabla\cdot(\rho\nabla V_\varepsilon*\log(V_\varepsilon*\rho)). \tag{NLHE}
\end{equation}
For kernels $V_{\eps}$ which are compactly supported we introduce another regularisation 
\begin{equation}\label{eq:log_entropy_reg_gaussian_intro}
\mh^\varepsilon_\sigma[\rho]:= \int_\Rd ((1-\sigma)V_\varepsilon*\rho+\sigma \mn)(x)\log\left((1-\sigma)V_\varepsilon*\rho(x)+\sigma \mn\right) \diff x
\end{equation}
which is necessary because, if both $\rho$ and $V_{\eps}$ are compactly supported, then $V_\varepsilon*\rho(x) = 0$ on a set of positive measure. The function $\mn$ is globally supported and we consider, for simplicity, the specific choices in~\eqref{eq:notation_kernel_n} (below) for this work. The corresponding PDE reads
\begin{equation}\label{eq:nonlocal_he_gauss}
    \partial_t\rho=(1-\sigma)\nabla\cdot(\rho\nabla V_\varepsilon*\log((1-\sigma)V_\varepsilon*\rho+\sigma \mn)). \tag{NLHE$_{\sigma}$}
\end{equation}
Our main results assert that there exist solutions $\rho^{\varepsilon, \sigma}$ and $\rho^{\varepsilon}$ to \eqref{eq:nonlocal_he_gauss} and \eqref{eq:nonlocal_he}, respectively. They are constructed via the JKO scheme (Theorems \ref{thm:exist_nlhe_gauss} and \ref{thm:exist_nlhe_sigma=0}). Moreover, using a new commutator estimate (Lemma \ref{lem:commutator_lemma}), we prove that when $\eps \to 0$ and $\eps, \sigma\to 0$, 
$$
\rho^{\varepsilon}(t), \rho^{\varepsilon, \sigma}(t) \to \rho(t) \mbox{ narrowly for all } t \in [0,T],
$$
where $\rho$ is the solution of the heat equation \eqref{eq:heat_eq} (see Theorems~\ref{thm:joint_limit_eps_sigma}, \ref{thm:sig_to_zero}, and~\ref{thm:eps_to_zero} for the precise statements). Finally, for the nonlocal equations \eqref{eq:nonlocal_he} and \eqref{eq:nonlocal_he_gauss} we define particle schemes approximating their solutions. To obtain quantitative rates of convergence when $N \to \infty$, $\eps \to 0$, we prove that the functionals \eqref{eq:log_entropy_reg} and \eqref{eq:log_entropy_reg_gaussian_intro} are $\lambda_{\eps,\sigma}$-geodesically convex when restricted to compactly supported densities $\rho$. Interpreting \eqref{eq:nonlocal_he} and \eqref{eq:nonlocal_he_gauss} as continuity equations, we prove that solutions remain compactly supported and this allows us to obtain quantitative rates of convergence for these schemes (Theorem \ref{thm:particle_approx}).

\smallskip
The previous results are extended to the fast diffusion equation 
\begin{equation}
\partial_t \rho = \Delta \rho^m = -\frac{m}{1-m}\,\mbox{div}(\rho \nabla \rho^{m-1})\tag{FDE}
\end{equation}
for $m \in \left(\frac{d}{d+2},1\right)$, which is the $2$-Wasserstein gradient flow of the free energy
\begin{equation}
    \mathcal{U}^m[\mu]:=
        -\frac{1}{1-m}\int_\Rd\rho(x)^m\diff x, \qquad m\in\left(\frac{d}{d+2},1\right),
\end{equation}
for $\mu$ a probability measure such that $\mu=\rho\mathcal{L}^d+\mu_s$ with $\mu_s\perp\mathcal{L}^d$, being $\mathcal{L}^d$ the $d$-dimensional Lebesgue measure. The nonlocal approximation of~\eqref{eq:fast_diff} is
\begin{equation}
\partial_t \rho^{\eps} = -\frac{m}{1-m}\, \mbox{div}(\rho^{\eps} \nabla V_{\eps}\ast (V_{\eps}\ast\rhoe)^{m-1})\tag{NLFDE},
\end{equation}
whose corresponding (regularised) energy functional is
\begin{equation}
    \mume[\mu]:=\mum[V_\varepsilon*\mu]=-\frac{1}{1-m}\int_\Rd|V_\varepsilon*\mu|^m\diff x, \qquad m\in\left(\frac{d}{d+2},1\right).
\end{equation}
The kernel we consider resembles a Barenblatt, cf.~\eqref{eq:kernel_fast_diffusion}. In short, we construct weak solutions to~\eqref{eq:fast_diff_nonlocal} via the JKO scheme in~\Cref{thm:existence_weak_sol_jko_nonlocal_fast} and prove they converge to weak solutions to~\eqref{eq:fast_diff} as $\varepsilon\to0^+$ (\Cref{thm:eps_to_0_fast}). We prove $\lambda_\varepsilon$-geodesic convexity for $\mume$ (\Cref{prop:convexity-energy_fast_nonlocal}) which leads to a particle approximation for~\eqref{eq:fast_diff} as consequence, cf.~\Cref{thm:particle_approx_fast}.

\smallskip
Our manuscript provides several novelties in the analysis of nonlocal-to-local limits related to particle methods. First, we provide quantitative rates of convergence for the particle approximation of \eqref{eq:nonlocal_he} and \eqref{eq:nonlocal_he_gauss} in Section \ref{sec:particles}. The main difficulty is that the functionals $\mh^\varepsilon[\rho]$ and $\mh^\varepsilon_\sigma[\rho]$ in \eqref{eq:log_entropy_reg} and \eqref{eq:log_entropy_reg_gaussian_intro} are not known to be $\lambda$-geodesically convex, not even with an $\eps$-dependent $\lambda$, for general densities. This is a well-known issue for nonlocal functionals. Indeed, a necessary condition in dimension $d\geq 2$ for the quadratic functional $\mathcal{F}[\rho] = \int_{\R^d} |\rho \ast V_{\eps}|^2 \diff x = \int_{\R^d} \rho  \ast V_{\eps} \ast V_{\eps} \, \diff \rho$ to be geodesically convex is that $V_{\eps}\ast V_{\eps}$ is convex which cannot be true for $V_{\eps}$ of class $C^1$ \cite[Exercise 5.28]{V1}. In~\cite{CEW_nl_to_local_24} it is proven that $\mathcal{F}[\rho]$ is $\lambda_\eps$-geodesically convex with $\lambda_\eps\to-\infty$ when $\eps\to 0$. We are able to prove that $\mh^\varepsilon[\rho]$ and $\mh^\varepsilon_\sigma[\rho]$ are $\lambda_\eps$-geodesically convex with $\lambda_\eps\to-\infty$ when $\eps\to 0$ only for compactly supported measures under certain assumptions on the kernel $V_\eps$ --- this allows to control the log-singularity near the origin and provide linear growth for the first variations of our functionals. Proving $\lambda$-convexity for general measures remains an interesting open question. Restricting to compactly supported measures does not induce issues since, for $\eps>0$, we show finite speed of propagation of solutions to \eqref{eq:nonlocal_he} and \eqref{eq:nonlocal_he_gauss} with compactly supported initial data. Combining the last two properties allows us to obtain the results on rate of convergence of the particle methods we propose under suitable assumption on the kernels for compactly supported initial data. Compared to the particle approximation given in \cite{craig2023nonlocal}, we do not need any particular scaling between the regularisation parameters $\eps$, $\sigma$, except with respect to the number of particles. Moreover our regularisation does not need a numerical approximation of a regularised Moreau-Yosida approximation of the entropy. The same observations holds true for~\eqref{eq:fast_diff},~\eqref{eq:fast_diff_nonlocal}, and $\mume$.

\smallskip
Another important improvement concerns the use of globally supported kernels $V_\eps$, for which it was previously unknown how to pass to the $\eps\downarrow0$ limit in the weak formulation of \eqref{eq:nonlocal_he}. More precisely, the difficulty comes from the fact that the convolution operator does not commute with nonlinearities. The main tool, the so-called commutator estimates, dates back to the work of DiPerna and Lions,~\cite{DiPernaLions89}, and allows to exchange of the convolution operator and nonlinearity up to a small error. For globally supported kernels, the commutator error equals 
$
(\rho^{\eps}\ast (V_{\eps}|\cdot|))/(\sqrt{\rho^{\eps}\ast V_{\eps}})
$
which is a fraction of two seemingly incomparable quantities. In~\Cref{lem:commutator_lemma}, we develop a simple trick to estimate this fraction in $L^{\infty}_t L^2_x$ using the Cauchy-Schwartz inequality and second moment of $V_{\eps}$. We also cover compactly supported kernels $V_\eps$ as they are easier to implement in practical applications in sampling, however they lead to slower convergence rates in theory compared to globally supported kernels with our approach.

\smallskip
The structure of the paper is as follows.~\Cref{sec:preliminary} contains notations and definitions of weak solutions. Then, in~\Cref{sec:nonlocal_eq} we construct weak solutions to \eqref{eq:nonlocal_he} and \eqref{eq:nonlocal_he_gauss} via the JKO scheme.~\Cref{sec:nl_lo} concerns the proofs of convergence of these solutions to the heat equation~\eqref{eq:heat_eq}. We define the particle methods approximating~\eqref{eq:nonlocal_he} and \eqref{eq:nonlocal_he_gauss} in~\Cref{sec:particles}, and we study their convergence properties towards~\eqref{eq:heat_eq}. We provide a nonlocal particle approximation of~\eqref{eq:fast_diff} in~\Cref{sec:fast_diffusion}. Our manuscript is concluded with~\Cref{sec:appendix} containing complementary technical results.

\section{Preliminaries}
\label{sec:preliminary}
We denote by $\mptrd$ the space of probability measures on $\R^d$ with finite second moment, i.e. $\mptrd$ consists of all probability measures $\rho$ such that $m_2(\rho) := \int_{\R^d} |x|^2 \, \mathrm{d}\rho(x)<+\infty$. The subset $\mpdtard\subset \mptrd$ consists of probability measures with absolutely continuous densities with respect to the Lebesgue measure on $\R^d$. For $\rho \in \mpdtard$, we will often denote the probability \textit{measure} with its probability \textit{density}, i.e. $\diff\rho(x) = \rho(x)\diff x$. We equip $\mptrd$ with the $2$-Wasserstein distance, which is, for $\mu_1,\mu_2\in \mptrd$
\begin{equation}\label{wass}
d_W^2(\mu_1,\mu_2):=\min_{\gamma\in\Gamma(\mu_1,\mu_2)}\left\{\int_{\Rdd}|x-y|^2\,\mathrm{d}\gamma(x,y)\right\},
\end{equation}
where $\Gamma(\mu_1,\mu_2)$ is the class of all transport plans between $\mu_1$ and $\mu_2$, that is the class of measures $\gamma\in\mP(\Rdd)$ such that, denoting by $\pi_i$ the projection operator on the $i$-th component of the product space, the following marginality condition is satisfied 
$$
(\pi_i)_{\#}\gamma=\mu_i \quad \mbox{for}\ i=1,2.
$$
We write marginals as the push-forward of $\gamma$ through $\pi_i$. For a measure $\rho\in\mP(\Rd)$ and a Borel map $T:\Rd\to\Rn$, $n\in\mathbb{N}$, the push-forward of $\rho$ through $T$ is defined by
$$
 \int_{\Rn}f(y)\,\diff T_{\#}\rho(y)=\int_{\Rd}f(T(x))\,\diff \rho(x) \qquad \mbox{for all Borel functions $f$ on}\ \Rn.
$$

Setting $\Gamma_0(\mu_1,\mu_2)$ as the class of optimal plans, i.e. minimizers of \eqref{wass}, the $2$-Wasserstein distance can be written as
$$
d_W^2(\mu_1,\mu_2)=\int_{\Rdd}|x-y|^2\,\diff \gamma(x,y), \qquad \gamma\in\Gamma_0(\mu_1,\mu_2).
$$

We shall also use the $1$-Wasserstein distance, denoted by $d_1$ and defined by
\begin{align}\label{eq:1wass}
    d_1(\mu_1,\mu_2):=\min_{\gamma\in\Gamma(\mu_1,\mu_2)}\left\{\int_{\Rdd}|x-y|\,\diff \gamma(x,y)\right\}.
\end{align}
We refer the reader to, e.g.~\cite{AGS,V2,S}, for further details on optimal transport theory.

We denote the Lebesgue bracket by $\langle x\rangle:= \sqrt{1+|x|^2}$, for any $x\in\R^d$, and consider the following smooth and globally supported functions
\begin{equation}\label{eq:notation_kernel_n}
\mn_p(x) := \frac{1}{Z}\exp(-\langle x\rangle^p), \quad Z = \int_{\R^d}\exp(-\langle x\rangle)^p\, \diff x, \qquad p=1,2. \tag{$\bm{\mn}$}
\end{equation}
The parameter $\sigma\in[0,1)$ distinguishes the energies~\eqref{eq:log_entropy} (corresponding to $\sigma = 0$) and~\eqref{eq:log_entropy_reg} (corresponding to $\sigma>0$) leading to~\eqref{eq:nonlocal_he} and~\eqref{eq:nonlocal_he_gauss}, respectively. For $\varepsilon>0$, we define the mollifying sequence by $V_\varepsilon(x) = \varepsilon^{-d}V_1(x/\varepsilon)$ where $V_1$ satisfies:
\begin{enumerate}[label=\textbf{(V\arabic*)}]
	\item $V_1\in C_b(\Rd)\cap C^1(\Rd)$, $V_1\ge 0$, $\|V_1\|_{L^1}=1$, $V_1(x)=V_1(-x)$ \label{ass:v1};
    \item depending on $\sigma\in[0,1)$, we assume \label{ass:v2}
    \begin{enumerate}[label=$\bm{\mathrm{(V2)_c}}$]
        \item in the case $\sigma>0$ that supp$V_1\subset B_1$ or \label{ass:v2-c}
    \end{enumerate}
    \begin{enumerate}[label=$\bm{\mathrm{(V2)_g}}$]
        \item in the case $\sigma = 0$ that $V_1 \equiv \mn_p$, for $p=1$ or $p=2$.\label{ass:v2-g}
    \end{enumerate}
\end{enumerate}

The subscripts in~\ref{ass:v2} refer to whether the kernel is \textit{compactly} or \textit{globally} supported. 
From now on, we refer to $\mn$ as either $\mn_1$ or $\mn_2$. We gain, by distinguishing between~\ref{ass:v2-c} and~\ref{ass:v2-g}, the ability to incorporate both compactly supported and globally supported kernels $V_1$ into our theory. Note that, for $p=2$, we could choose $\exp{\left(-|x|^2\right)}$ with no Lebesgue bracket, which is needed, instead, for $p=1$ to have better regularity. For consistency of notation we will stick to $\mn_2$ with the previous observation in mind.

\begin{rem}
\label{rem:general_global_kernel}
    All we need from~\ref{ass:v2-g} is that $V_1$ is radially decreasing, $|\cdot|^2 V_1 \in L^1(\Rd)$, and that there exists a constant $C>0$ such that $|\nabla V_1(x)|\le C\langle x\rangle^{-(d+3)}$ and $V_1(x) \ge Ce^{-|x|^2}$. For ease of exposition, in the rest we only discuss the cases $V_1 = \mn_p$.
\end{rem}

\begin{rem}\label{rem:gaussian}
In the particular case $V_1 = \mathcal{N}_1$ (or any kernel as in~\Cref{rem:general_global_kernel}) we do not need to consider the regularization with respect to $\sigma$, thus we can assume $\sigma=0$. Indeed, for fixed $\eps>0$, the sequence of solutions to the JKO scheme (which is well-defined for $\sigma = 0$) is tight, uniformly in time. It follows that there exists $R>0$ such that
$$
\int_{B_R} \diff \rho^{\eps}_t(x) \geq \frac{1}{2}.
$$
Hence, with $V_1 = \mn_1$, a simple estimation with~\Cref{lem:peetre} yields
$$
\rho^{\eps} \ast V_{\eps}(t,x) =\int_{\R^d} V_\eps(x-y) \diff \rho^{\eps}_t(y) \geq \int_{B_R} V_{\eps}(x-y) \diff \rho^{\eps}_t(y) \geq C\exp\left(-\sqrt{1 + \frac{R^2}{\varepsilon^2}}\right) V_\varepsilon
(x)>0.
$$
The lower bound above plays essentially the same role as regularization with $\sigma>0$. In this case, we also note $\mn_1$ enjoys further properties allowing to compensate for its global support. More precisely, it holds $\int_\Rd |x|^2\mn_1(x)\diff x<\infty,\, \mbox{and } |\nabla\mn_1(x)|\le C\langle x \rangle^{-(d+3)}$. The same can be said of $\mn_2$.
\end{rem}
\noindent We shall see (cf.~\Cref{thm:exist_nlhe_gauss}) that the 2-Wasserstein gradient flow of $\mh_\sigma^\varepsilon$ is a solution to \eqref{eq:nonlocal_he_gauss}.
\begin{defn}[Weak measure solution to~\eqref{eq:nonlocal_he_gauss}]
\label{def:weak-meas-sol-sigma}
For $\varepsilon>0$ and $\sigma \in [0,1)$, we say that an absolutely continuous curve $\rho:[0,T]\to \mptrd$ is a weak measure solution to~\eqref{eq:nonlocal_he_gauss} if, for every $\varphi \in C_c^1(\R^d)$ and any $t\in[0,T]$, it holds
\begin{equation}
    \label{eq:weak-form-sigma}
    \begin{split}
    \int_{\R^d}\varphi(x) \diff\rho_t(x) &- \int_{\R^d}\varphi(x) \diff\rho_0(x)\\ 
    &= -(1-\sigma)\int_0^t\int_{\R^d}\nabla \varphi(x)\cdot \nabla V_\varepsilon*[\log((1-\sigma)V_\varepsilon*\rho_s +\sigma \mn)](x) \diff \rho_s(x) \diff s.
    \end{split}
\end{equation}
\end{defn}
\noindent Note that the formulation~\eqref{eq:weak-form-sigma} is equivalent to the more common distributional form
\begin{equation}    \label{eq:weak-form-sigma_towork_with}
\begin{split}
    \int_{\R^d}\varphi(T,x) \diff \rho_T(x) &- \int_{\R^d}\varphi(0,x) \diff \rho_0(x) = \int_0^T \int_{\R^d} \partial_t \varphi(t,x) \diff \rho_t \diff x
    \\ &-(1-\sigma)\int_0^T\!\!\!\int_{\R^d}\nabla \varphi(t,x)\!\cdot\! \nabla V_\varepsilon*[\log((1-\sigma)V_\varepsilon*\rho_t +\sigma \mn)](x) \diff \rho_t(x)  \diff t
\end{split}
\end{equation}
required for all $\varphi \in C_c^1([0,T]\times\R^d)$ which can be obtained by multiplying 
\eqref{eq:weak-form-sigma} by $\partial_t \psi(t)$, integrating by parts in time and noticing that any function $\varphi(t,x)$ can be approximated by a linear combination of functions of the form $\psi(t)\, \varphi(x)$ in the supremum norm.

\begin{rem}
We observe that the RHS of~\eqref{eq:weak-form-sigma_towork_with} is well-defined since, for $p=1$ or $p=2$, 
\[
\left|\int_0^T\int_\Rd \nabla \varphi(x)\!\cdot\! \nabla V_\varepsilon * [\log((1\!-\!\sigma)V_\varepsilon*\rho_t\!+\! \sigma\mn)](x)\mathrm{d}\rho_t(x)\diff t\right| \!\le\! C\!\!\int_0^T\int_\Rd\! \langle x\rangle^p \mathrm{d}\rho_t(x) <\infty.
\]
This is a consequence of~\Cref{lem:log_gauss} and Peetre's inequality (\Cref{lem:peetre}) which imply
\begin{align*}
    |\nabla V_\varepsilon &* [\log ((1-\sigma)V_\varepsilon * \rho_t + \sigma \mn(\cdot))](x)|\\
    &\le \int_\Rd |\nabla V_\varepsilon(y)| |[\log ((1-\sigma)V_\varepsilon * \rho_t + \sigma \mn(\cdot))](x-y)|\, \mathrm{d}y 
    \\
    &\le C\int |\nabla V_\varepsilon(y)|\langle x-y\rangle^p \, \mathrm{d}y \le C\langle x\rangle^p \int |\nabla V_\varepsilon(y)| \langle y\rangle^p \, \mathrm{d}y \le C\langle x\rangle^p.
\end{align*}
In the case $\sigma=0$, first we notice $V_\varepsilon*\rho_s>0$, for any $s\in[0,T]$, since the solution is tight and we can argue as in~\Cref{rem:gaussian}. Furthermore, a similar estimate as above can be obtained using the growth proven in~\Cref{lem:log_conv_glob_rescaling}. We stress that~\Cref{def:weak-meas-sol-sigma} is important to state for $\sigma\in(0,1)$ for the double limit. If $\sigma=0$ the improved regularity we discuss below allows for a stronger concept of solution, which is the one we need for the nonlocal-to-local limit in~\Cref{sec:nl_lo}.  
\end{rem}
Due to the convolution by $V_\varepsilon$ and the lift by $\mn$, the integrand on the right-hand side of~\eqref{eq:weak-form-sigma_towork_with} can also be interpreted as
    \[
    \nabla \varphi(x)\cdot \left\{V_\varepsilon * \left[
    \frac{\nabla ((1-\sigma)V_\varepsilon * \rho_s + \sigma \mn)}{(1-\sigma)V_\varepsilon * \rho_s + \sigma \mn}
    \right]\right\}(x) = 2\nabla \varphi(x)\cdot \left\{V_\varepsilon * \left[\frac{\nabla \sqrt{(1-\sigma)V_\varepsilon * \rho_s + \sigma \mn}}{\sqrt{(1-\sigma)V_\varepsilon * \rho_s + \sigma \mn}}\right]\right\}(x).
    \]
Indeed, we will see later with the $L^2_tH_{x}^1$-bound for $\sqrt{(1-\sigma)V_\varepsilon * \rho + \sigma \mn}$ (cf.~\eqref{eq:deriv-dis}), that this form is highly amenable to the $\sigma \to 0$ limit and it can be obtained also in the $\sigma=0$ case. We, therefore, consider the following notion of solution to~\eqref{eq:nonlocal_he}.
\begin{defn}[Weak measure solution to~\eqref{eq:nonlocal_he}]
    \label{def:weak-meas-sol}
    For $\varepsilon>0$, we say that an absolutely continuous curve $\rho:[0,T]\to \mptrd$ is a weak measure solution to~\eqref{eq:nonlocal_he} if the following holds:
    \begin{enumerate}[label = (NLHE-\arabic*)]
        \item (Finite initial entropy) \label{a:fie}$\mh[\rho_0]<+\infty$.
        \item (Regularity) \label{a:l2} $\nabla \sqrt{V_\varepsilon*\rho} \in L_{t,x}^2$.
        \item (Weak formulation) For every $\varphi \in C_c^1(\R^d)$ and any $t\in[0,T]$, it holds
        \begin{equation}
            \label{eq:weak-form}
            \int_{\R^d}\varphi(x) \, \mathrm{d}\rho_t(x) - \int_{\R^d}\varphi(x) \, \mathrm{d}\rho_0(x) = -2 \int_0^t\int_{\R^d} { \mathds{1}_{V_{\varepsilon} \ast \rho_{s} > 0}}\, \frac{V_\varepsilon * (\nabla \varphi \rho_s)}{\sqrt{V_\varepsilon * \rho_s}} \cdot \nabla \sqrt{V_\varepsilon * \rho_s}\diff x\diff s.
        \end{equation}

    \end{enumerate}
\end{defn}
Similarly to~\eqref{eq:weak-form-sigma_towork_with}, the formulation~\eqref{eq:weak-form} is equivalent to the following:
for all $\varphi \in C^1_c([0,T]\times\R^d)$ we have
\begin{equation}    \label{eq:weak-form-just_eps_towork_with}
\begin{split}
    \int_{\R^d}\varphi(T,x) \diff \rho_T(x)\! - \!\int_{\R^d}\varphi(0,x) \diff \rho_0(x) &\!=\! \int_0^T\!\! \int_{\R^d} \partial_t \varphi(t,x) \diff \rho_t \diff x
    \\ &\quad-2\int_0^T\!\!\int_{\R^d}{ \mathds{1}_{V_{\varepsilon} \ast \rho_{t} > 0}}\, \frac{V_\varepsilon * (\nabla \varphi \rho_t)}{\sqrt{V_\varepsilon * \rho_t}} \cdot \nabla \sqrt{V_\varepsilon * \rho_t} \diff x  \diff t.
\end{split}
\end{equation}
\begin{rem}
\label{rem:sqrt-well-defined}
The right-hand side of~\eqref{eq:weak-form} is well-defined owing to the regularity assumption $\nabla \sqrt{V_\varepsilon * \rho}\in L_{t,x}^2$ and the fact that
\[
\left|\frac{V_\varepsilon*(\nabla \varphi \rho_s)}{\sqrt{V_\varepsilon * \rho_s}}\right| \le \|\nabla\varphi\|_{L^\infty}\sqrt{V_\varepsilon * \rho_s} \in L_t^\infty L_x^2.
\]
The characteristic function makes the expression well-defined pointwisely. It is rigorously proved in Theorem \ref{thm:sig_to_zero} that the $\sigma \to 0$ limit yields the formulation with the characteristic function. 
\end{rem}

\begin{defn}[Weak solution to~\eqref{eq:heat_eq}]\label{def:weak_sol_he}
A weak solution to~\eqref{eq:heat_eq} on the time interval $[0,T]$ with initial datum $\rho\in\mpdtard$ such that $\mh[\rho_0]<\infty$ is an absolutely continuous curve $\rho\in C([0,T];\mptrd)$ satisfying the following properties:
\begin{enumerate}
    \item for almost every $t\in[0,T]$ the measure $\rho_t$ is absolutely continuous with respect to the Lebesgue measure, still denoted by $\rho_t$, such that $\rho\in L^\infty([0,T];L^1(\Rd))$ and $\nabla \sqrt{\rho}\in L^2([0,T]\times\Rd)$;
    \item for any $\varphi\in C^1_c(\Rd)$ and all $t\in[0,T]$ it holds
    \[
        \int_\Rd\varphi(x)\rho_t(x)\,\diff x=\int_\Rd\varphi(x)\rho_0(x)\,\diff x - \int_0^t\int_\Rd\nabla\varphi(x)\cdot \nabla\rho_t(x)\diff x\diff t,
    \] 
    or the equivalent formulation with test functions $\varphi\in C_c^1([0,T]\times\Rd)$.
\end{enumerate}
\end{defn}
For the reader's convenience we postpone the notions of solution for~\eqref{eq:fast_diff_nonlocal} to~\Cref{sec:fast_diffusion}.

\section{Nonlocal approximating equation for the heat equation}
\label{sec:nonlocal_eq}
In this section we focus on the nonlocal PDEs and construct weak measure solutions to~\eqref{eq:nonlocal_he_gauss} and~\eqref{eq:nonlocal_he}. As for the latter, we observe most of the estimates are already provided in~\cite{CEW_nl_to_local_24}, though the consistency of JKO scheme (or existence of weak measure solutions) is not proven. For this reason, in this section we let $\varepsilon>0$ be fixed and extend the existence theory to the case $0<\sigma<1$, i.e.~\eqref{eq:nonlocal_he_gauss}, including $\sigma=0$ as well,~\eqref{eq:nonlocal_he}.
\subsection{Perturbation by vacuum exclusion}
\label{sec:nonlocal_he_gauss}
The main issue with~\eqref{eq:nonlocal_he} is the singularity of the logarithm near the origin. We deal with the loss of regularity near vacuum by adding a nice perturbation or choosing cleverly the kernel. For this reason, in the first case we consider the perturbed problem~\eqref{eq:nonlocal_he_gauss}. Let us recall the JKO scheme~\cite{JKO98} for constructing solutions to~\eqref{eq:nonlocal_he_gauss} on $[0,T]$ subject to initial condition $\rho_0\in \mptrd$ is given by the following recursive procedure. 
\begin{itemize}
    \item Fix a (sufficiently small) time step $\tau\in(0,1)$ and set $\rho_{\tau}^{\varepsilon, \sigma,0}:= \rho_0$.
    \item Define $N:= \left[\frac{T}{\tau}\right]$ and, for $n=0,1,2,\dots, N-1$ we choose
    \begin{equation}
        \label{eq:jko}
		\rhotnnes\in\argmin_{\rho\in\mptrd}\left\{\frac{d_W^2(\rhotnes,\rho)}{2\tau}+\mh_\sigma^\varepsilon[\rho]\right\}.
    \end{equation}
\end{itemize}
Owing to~\cite[Lemma A.2]{CEW_nl_to_local_24}, the above sequence is well-defined for $\tau$ sufficiently small and independent of $\varepsilon$ and $\sigma$. Let $0<\varepsilon_0<\infty$ and, for $0<\varepsilon\le\varepsilon_0$, consider the piecewise constant interpolation
\[
\rho_\tau^{\varepsilon, \sigma}(0) = \rho_0, \quad \rho_\tau^{\varepsilon, \sigma}(t) = \rhotnes, \quad t\in((n-1)\tau,n\tau], \quad n=1,\dots,N,
\]
begin $\rhotnes$ defined in~\eqref{eq:jko}. From~\cite[Proposition 3.1]{CEW_nl_to_local_24}, we obtain the following compactness and uniform bound result. Throughout this section, we fix $\mn$ to be either $\mn_1$ or $\mn_2$.
\begin{lem}
    \label{lem:en-ineq-mom-bound}
    Let $\varepsilon>0$, $\sigma \in [0,1)$. There exists an absolutely continuous curve $\rho^{\varepsilon, \sigma}:[0,T]\to \mptrd$ such that, up to a subsequence, we have
\[
\rho_\tau^{\varepsilon, \sigma}(t) \overset{\tau \to 0}{\rightharpoonup} \rho^{\varepsilon, \sigma}(t), \quad \text{uniformly in }t\in[0,T].
\]
There is a constant $C>0$ depending only on $T, \,m_2(\rho_0), \, m_2(V_1)$, and $ \mh_\sigma^\varepsilon[\rho_0]$ where $\rho_0 = \rho_\tau^{\varepsilon, \sigma}(0) = \rho^{\varepsilon, \sigma}(0)$ such that the second moment bound holds
\[
\int_{\R^d}|x|^2 \, \mathrm{d}\rho_\tau^{\varepsilon, \sigma}(t,x) \le C, \quad \forall t\in[0,T].
\]
As well, for every $t\in[0,T]$, the regularised entropies are ordered in the following way
\begin{equation}
\label{eq:sig_ent}
\mh_\sigma^\varepsilon[\rho_\tau^{\varepsilon,\sigma}(t)]\le \mh_\sigma^\varepsilon[\rho_0]\le (1-\sigma)\mh^\varepsilon[\rho_0]+\sigma C_\mn\le (1-\sigma)\mh[\rho_0] +\sigma C_\mn,
\end{equation}
for $C_\mn := \int_{\R^d} \mn(x) \log\mn(x)\,\mathrm{d}x$. All these bounds are also true with $\rho^{\varepsilon, \sigma}(t)$ replacing $\rho_\tau^{\varepsilon, \sigma}(t)$.
\end{lem}
\begin{proof}
    The existence of the limit curve $\rho^{\varepsilon,\sigma}$ and the second moment bound is classical and we refer to~\cite[Proposition 3.1]{CEW_nl_to_local_24} for all the details. We focus the rest of this proof on the entropy estimates in~\eqref{eq:sig_ent}.

    By construction in the JKO scheme, for any $t\in (n\tau,\,(n+1)\tau]$, we get
    \[
\mh_\sigma^\varepsilon[\rho_\tau^{\varepsilon,\sigma}(t)] \le\mh_\sigma^\varepsilon[\rhotnnes] + \frac{d_W^2(\rhotnes,\,\rhotnnes)}{2\tau} \le \mh_\sigma^\varepsilon[\rhotnes] + \frac{d_W^2(\rhotnes,\,\rhotnes)}{2\tau} = \mh_\sigma^\varepsilon[\rho_\tau^{\varepsilon,\sigma}(t-\tau)].
    \]
By iterating this, we obtain the first inequality in~\eqref{eq:sig_ent}. The remaining inequalities in~\eqref{eq:sig_ent} exploit the convexity of the map $r \mapsto r\log r$. The second inequality is easily deduced by
\begin{align*}
    \mh_\sigma^\varepsilon[\rho_0] &= \int ((1-\sigma) V_\varepsilon * \rho_0 + \sigma\mn)\log ((1-\sigma) V_\varepsilon * \rho_0 + \sigma\mn)\diff x  \\
    &\le \int (1-\sigma)(V_\varepsilon*\rho_0)\log (V_\varepsilon*\rho_0) + \sigma \mn \log \mn \diff x= (1-\sigma)\mh^\varepsilon[\rho_0] +C_\mn.
\end{align*}
For the third and final inequality of~\eqref{eq:sig_ent}, we appeal to Jensen's inequality
\begin{align*}
\mh^\varepsilon[\rho_0] &= \int \left(\int \rho_0(x-y) \, \diff V_\varepsilon(y)\right) \log \left(\int \rho_0(x-y) \, \diff V_\varepsilon(y)\right) \, \diff x \\
&\le \int \left( \int \rho_0(x-y) \log (\rho_0(x-y)) \, \diff V_\varepsilon(y)\right)\ \diff x \\
&= \int V_\varepsilon * [\rho_0\log \rho_0] \diff x= \int \rho_0\log \rho_0 \diff x= \mh[\rho_0].\qedhere
\end{align*}
\end{proof}
We now verify that the limit $\rho^{\varepsilon,\sigma}$ in~\Cref{lem:en-ineq-mom-bound} is a weak solution to~\eqref{eq:nonlocal_he_gauss}. To keep notations simple, we suppress the dependence on $\sigma\ge 0$ until the end of this section. To keep track of the different assumptions~\ref{ass:v2-c} with $\sigma>0$ and~\ref{ass:v2-g} when $\sigma=0$, we begin with the case $\sigma>0$.
\begin{thm}
\label{thm:exist_nlhe_gauss}
Fix $\varepsilon>0$, $\sigma \in (0,1)$ together with $V_1$ satisfying~\ref{ass:v1} and~\ref{ass:v2-c}. Suppose $\rho_0\in \mptrd$ satisfies $\mh_\sigma^\varepsilon[\rho_0]<+\infty$. Then, there exists a weak measure solution $\rho^\varepsilon$ to~\eqref{eq:nonlocal_he_gauss} such that $\rho^\varepsilon(0) = \rho_0.$
\end{thm}
\begin{proof}
Recall $\mn$ is defined in~\eqref{eq:notation_kernel_n}. Our candidate weak measure solution is $\rhoe$ from~\Cref{lem:en-ineq-mom-bound} although we will spend most of this proof working at the level of the discrete minimisers $\rhotne$ from~\eqref{eq:jko}. Fix $n=1,2,\dots, N-1; \eta\in (0,1);$ and $\zeta \in C_c^\infty(\R^d; \, \R^d)$. Given consecutive elements of the sequence $\rhotne$ and $\rhotnne$, we will suppress the dependence on $\tau, \varepsilon$ and consider the following perturbation
    \[
    \rho^n \equiv \rhotne, \quad \rho^{n+1} \equiv \rhotnne, \quad \rho_\eta^{n+1} = P_\#^\eta \rho^{n+1}, \quad P^\eta(x) = x + \eta\zeta(x).
    \]
    Being $\rho^{n+1}$ a minimiser of~\eqref{eq:jko}, we have
    \begin{equation}
    \label{eq:optimality}
		\frac{1}{2\tau}\left[\frac{d_W^2(\rho^n,\rho_\eta^{n+1})- d_W^2(\rho^n, \rho^{n+1})}{\eta}\right]+\frac{\mh_\sigma^\varepsilon[\rho_\eta^{n+1}]-\mh_\sigma^\varepsilon[\rho^{n+1}]}{\eta}\ge0.
	\end{equation}
By sending $\eta\to 0$ in~\eqref{eq:optimality}, we will recover~\eqref{eq:weak-form-sigma_towork_with}.\\

\noindent \underline{The energy functional terms in~\eqref{eq:optimality}:} We claim that
\begin{equation}
    \label{eq:energy-diff}
    \frac{\mh_\sigma^\varepsilon[\rho_\eta^{n+1}]-\mh_\sigma^\varepsilon[\rho^{n+1}]}{\eta} \to (1-\sigma)\int_{\R^d}\zeta(x)\cdot \nabla V_\varepsilon* [\log((1-\sigma)V_\varepsilon * \rho^{n+1} + \sigma\mn)](x)\, \mathrm{d}\rho^{n+1}(x), \quad \text{as }\eta \to 0.
\end{equation}
The left-hand side of~\eqref{eq:energy-diff} can be written in the following way by Taylor expansion
\begin{align}
    &\quad \frac{\mh_\sigma^\varepsilon[\rho_\eta^{n+1}]-\mh_\sigma^\varepsilon[\rho^{n+1}]}{\eta} \notag\\
    &= \frac{1}{\eta}\int_{\R^d} ((1-\sigma)V_\varepsilon*\rho_\eta^{n+1}(x) + \sigma\mn(x))\log [(1-\sigma)V_\varepsilon*\rho_\eta^{n+1}(x) + \sigma\mn(x)]    \notag\\
    &\quad - ((1-\sigma)V_\varepsilon*\rho^{n+1}(x) + \sigma\mn(x))\log \left[(1-\sigma)V_\varepsilon*\rho^{n+1}(x) + \sigma\mn(x)\right]\, \diff x \notag \\
    &= \frac{1-\sigma}{\eta}\int_{\R^d}\{V_\varepsilon * \rho_\eta^{n+1}(x) - V_\varepsilon * \rho^{n+1}(x)\}\times \notag\\
    &\qquad\qquad\times \underbrace{\int_0^1\log[(1-\sigma)\left\{sV_\varepsilon*\rho_\eta^{n+1}(x) + (1-s)V_\varepsilon*\rho^{n+1}(x)\right\} + \sigma \mn(x)] + 1\, \mathrm{d}s}_{=:M_\eta^\varepsilon(x)}\, \diff x \notag \\
    &= \frac{1-\sigma}{\eta} \int_{\R^d}(V_\varepsilon * M_\eta^\varepsilon)(x)\, \mathrm{d}[\rho_\eta^{n+1} - \rho^{n+1}](x) \notag \\
    &= (1-\sigma)\int_{\R^d} \frac{V_\varepsilon*M_\eta^\varepsilon(P^\eta(x)) - V_\varepsilon*M_\eta^\varepsilon(x)}{\eta}\, \mathrm{d}\rho^{n+1}(x) \notag \\
    &= (1-\sigma)\int_{\R^d}
\int_{\R^d}\left(\frac{V_\varepsilon(P^\eta(x) - y) - V_\varepsilon(x-y)}{\eta}\right)M_\eta^\varepsilon(y)\, \mathrm{d}y
    \, \mathrm{d}\rho^{n+1}(x). \label{eq:interaction-terms}
\end{align}
Owing to the Dominated Convergence Theorem, we have
\begin{equation}
    \label{eq:conv_m}
    M_\eta^\varepsilon(y) \to \log ((1-\sigma)V_\varepsilon * \rho^{n+1}(y) + \sigma \mn(y)) + 1, \quad \text{a.e. }y\in\R^d \quad \text{as }\eta \to 0.
\end{equation}
Indeed, by definition of $M_\eta^\varepsilon(y)$ and for almost every $s\in[0,1]$ and $y\in\R^d$, we have
\begin{align*}
\log \left[(1-\sigma)\left\{sV_\varepsilon*\rho_\eta^{n+1}(y) + (1-s)V_\varepsilon * \rho^{n+1}(y)\right\} +  \sigma \mn(y)\right] &+ 1   \\
    \overset{\eta \to 0}{\to} \log &[(1-\sigma)V_\varepsilon * \rho^{n+1}(y) + \sigma \mn(y)] + 1,
\end{align*}
since it is easy to verify that $\rho_\eta^{n+1} \overset{\eta \to 0}{\rightharpoonup}\rho^{n+1}$. Moreover, owing to~\Cref{lem:log_gauss}, the integrand in $M_\eta^\varepsilon(y)$ is uniformly-in-$\eta$ majorised by
\[
\left|\log \left[(1-\sigma)\left\{sV_\varepsilon*\rho_\eta^{n+1}(y) + (1-s)V_\varepsilon * \rho^{n+1}(y)\right\} +  \sigma \mn(y)\right] + 1\right| \le C(\sigma, \|V_\varepsilon\|_{L^\infty})\langle y\rangle^p \in L^1(\mathrm{d}s).
\]
In particular, we also have
\begin{equation}
    \label{eq:major_m}
    |M_\eta^\varepsilon(y)|\le C\langle y\rangle^p.
\end{equation}
As for the difference quotient in~\eqref{eq:interaction-terms}, we can write it as
\[
\frac{V_\varepsilon(P^\eta(x) - y) - V_\varepsilon(x-y)}{\eta}   = \frac{1}{\eta}\int_0^1\frac{d}{ds}V_\varepsilon(x+s\eta\zeta(x) - y)\, \mathrm{d}s   = \zeta(x)\cdot \int_0^1\nabla V_\varepsilon(x + s\eta \zeta(x) - y)\, \mathrm{d}s.
\]
By the Dominated Convergence Theorem, we have
\begin{equation}
    \label{eq:conv_diff_quot}
    \frac{V_\varepsilon(P^\eta(x) - y) - V_\varepsilon(x-y)}{\eta} \to \zeta(x) \cdot \nabla V_\varepsilon(x-y), \quad \text{a.e. }x,y\in\R^d, \quad \text{as }\eta\to 0.
\end{equation}
We now exhibit a uniform-in-$\eta$ majorant for the integrand in~\eqref{eq:interaction-terms}
\begin{align*}
    \left|
\frac{V_\varepsilon(P^\eta(x) - y) - V_\varepsilon(x-y)}{\eta} M_\eta^\varepsilon(y)
    \right|\le C\langle y\rangle^p|\zeta(x)|\int_0^1|\nabla V_\varepsilon(x+s\eta \zeta(x) - y)|\,\mathrm{d}s.
\end{align*}
Recall by assumption~\ref{ass:v2-c} that $V_1$ is supported on the unit ball. Hence, for some constant $C>0$ depending on $\varepsilon$ and $V_1$, we have
\[
|\nabla V_\varepsilon(z)| \le C\langle z\rangle^{-(d+3)}, \quad \forall z\in\R^d.
\]
Therefore, we can further estimate
\begin{align*}
   \left|
\frac{V_\varepsilon(P^\eta(x) - y) - V_\varepsilon(x-y)}{\eta} M_\eta^\varepsilon(y)
    \right|&\le C\langle y\rangle^p|\zeta(x)|\int_0^1\langle x+s\eta\zeta(x) - y \rangle^{-(d+3)}\,\mathrm{d}s   \\
    &\le C\langle y\rangle^{-(d+3-p)}|\zeta(x)|\int_0^1\langle x + s\eta\zeta(x)\rangle^{d+3}\,\mathrm{d}s     \\
    &\le C\langle y\rangle^{-(d+3-p)}|\zeta(x)| \left(\langle x\rangle^{d+3} + \langle \zeta(x)\rangle^{d+3}\right) \in L^1(\mathrm{d}y \, \mathrm{d}\rho^{n+1}(x)).
\end{align*}
We used~\Cref{lem:peetre} in the second inequality where the constant $C>0$ has possibly increased line-by-line. The last expression is a majorant in $L^1(\mathrm{d}y \, \mathrm{d}\rho^{n+1}(x))$ uniformly in $\eta$ owing to the decay in $y$ (recall $p=1$ or $p=2$) as well as the fact that $\zeta(x)$ is smooth and compactly support. Moreover, since~\eqref{eq:conv_m} and~\eqref{eq:conv_diff_quot} tell us what the pointwise limit of their product is, we can pass to the limit $\eta\to 0$ from~\eqref{eq:interaction-terms} using the Dominated Convergence Theorem to obtain
\begin{align*}
    &\quad \frac{\mh_\sigma^\varepsilon[\rho_\eta^{n+1}]-\mh_\sigma^\varepsilon[\rho^{n+1}]}{\eta}  = \int_{\R^d}
\int_{\R^d}\left(\frac{V_\varepsilon(P^\eta(x) - y) - V_\varepsilon(x-y)}{\eta}\right)M_\eta^\varepsilon(y)\, \mathrm{d}y
   \, \mathrm{d}\rho^{n+1}(x)   \\
    &\overset{\eta\to 0}{\to} (1-\sigma)\int_{\R^d}\zeta(x)\cdot \int_{\R^d}\nabla V_\varepsilon(x-y)[\log ((1-\sigma)V_\varepsilon * \rho^{n+1}(y) + \sigma \mn(y)) + 1]\, \mathrm{d}y\,\mathrm{d}\rho^{n+1}(x)  \\
    &= (1-\sigma)\int_{\R^d}\zeta(x)\cdot \nabla V_\varepsilon* [\log((1-\sigma)V_\varepsilon * \rho^{n+1} + \sigma\mn) + 1](x)\, \mathrm{d}\rho^{n+1}(x).
\end{align*}
This is precisely~\eqref{eq:energy-diff} since $\nabla V_\varepsilon * 1 = 0$.\\

\noindent \underline{Recovering the weak formulation of~\eqref{eq:nonlocal_he_gauss}:} Sending $\eta\to 0$ in the 2-Wasserstein terms in~\eqref{eq:optimality} is standard and we refer, e.g., to~\cite[Theorem 2.1]{CEW_nl_to_local_24} for details. We set $\zeta = \nabla \varphi$ for $\varphi \in C_c^\infty(\R^d)$, pass to the limit $\eta\to 0$ in~\eqref{eq:optimality}, and use~\eqref{eq:energy-diff} to obtain
\begin{align*}
\int_{\R^d}\varphi(x)\, \mathrm{d}\rho^{n+1}(x) &- \int_{\R^d}\varphi(x)\, \mathrm{d}\rho^n(x) + \mathcal{O}(\tau^2)\\
&= -\tau(1-\sigma)\int_{\R^d}\nabla \varphi(x)\cdot \nabla V_\varepsilon*[\log ((1-\sigma)V_\varepsilon*\rho^{n+1} + \sigma \mn)](x) \, \mathrm{d}\rho^{n+1}(x).
\end{align*}
At the level of the piecewise constant interpolants, for any fixed $0\le s < t \le T$ this implies (after summing up over $n$)
\begin{align}
\label{eq:weak-form-tau}
\begin{split}
	\int_\Rd\varphi(x)&\,\mathrm{d}\rhote(t,x)-\int_\Rd\varphi(x)\,\mathrm{d}\rhote(s,x)+\mathcal{O}(\tau^2)=\\
	&-(1-\sigma)\int_s^t\int_{\Rd}\nabla\varphi(x)\cdot \nabla V_\varepsilon* [\log((1-\sigma)V_\varepsilon*\rhote(r,\cdot)+ \sigma\mn(\cdot))](x)\,\mathrm{d}\rhote(r,x)\,\mathrm{d}r.
\end{split}
\end{align}
We would now like to pass to the limit $\tau\to 0$. Since $\rhote$ narrowly converges uniformly in time to $\rho^\varepsilon$ from~\Cref{lem:en-ineq-mom-bound}, the left-hand side of~\eqref{eq:weak-form-tau} is easily handled. As for the right-hand side, due to the uniform second moment bound in~\Cref{lem:en-ineq-mom-bound} and the regularity of $\varphi$ and $V_\varepsilon$, we have for almost every $r\in[s,t]$
\[
\left|\int \nabla \varphi(x)\!\cdot\! \nabla V_\varepsilon * [\log((1\!-\!\sigma)V_\varepsilon*\rhote(r,\cdot)\!+\! \sigma\mn(\cdot))](x)\mathrm{d}\rhote(r,x)\right| \!\le\! C\!\!\int\! \langle x\rangle^2 \diff \rhote(r,x) \!\le\! C \in L^1([s,t]; \mathrm{d}r).
\]
Here, we are using~\Cref{lem:log_gauss} and~\Cref{lem:peetre} to estimate
\begin{align*}
    &\quad |\nabla V_\varepsilon * [\log ((1-\sigma)V_\varepsilon * \rho_\tau^\varepsilon(r,\cdot) + \sigma \mn(\cdot))](x)|    \\
    &\le \int |\nabla V_\varepsilon(y)| |[\log ((1-\sigma)V_\varepsilon * \rho_\tau^\varepsilon(r,\cdot) + \sigma \mn(\cdot))](x-y)|\, \mathrm{d}y 
    \\
    &\le C\int |\nabla V_\varepsilon(y)|\langle x-y\rangle^p \, \mathrm{d}y \le C\langle x\rangle^p \int |\nabla V_\varepsilon(y)| \langle y\rangle^p \, \mathrm{d}y \le C\langle x\rangle^p.
\end{align*}
Thus, we have exhibited a uniform-in-$\tau$ majorant for the spatial integral. It remains to prove
\begin{align}
	\label{eq:AEC}
	\begin{split}
	& \int_{\Rd}\!\!\nabla\varphi(x)\!\cdot\! \nabla V_\varepsilon* [\log((1-\sigma)V_\varepsilon*\rhote(r,\cdot) + \sigma \mn(\cdot)](x)\,\mathrm{d}\rhote(r,x) \\
	&\to \int_{\Rd}\!\!\nabla\varphi(x)\!\cdot \!\nabla V_\varepsilon* [\log((1-\sigma)V_\varepsilon*\rho^\varepsilon_r(\cdot) +\sigma \mn(\cdot)](x)\,\mathrm{d}\rho^\varepsilon_r(x), \quad \tau\downarrow 0, \, \text{a. e. }r\in[0,T].
	\end{split}
\end{align}
For fixed $r\in[0,T]$, we drop the explicit dependence on this variable and consider the difference of the two expressions in~\eqref{eq:AEC}.
\begin{align}
&\int_{\Rd}\nabla\varphi(x)\cdot \nabla V_\varepsilon* [\log((1-\sigma)V_\varepsilon * \rhote + \sigma \mn)](x)\,\mathrm{d}\rhote(x)    \notag\\
&\qquad \qquad \qquad \qquad  - \int_{\Rd}\nabla\varphi(x)\cdot \nabla V_\varepsilon* [\log((1-\sigma)V_\varepsilon * \rho^\varepsilon + \sigma \mn)](x)\,\mathrm{d}\rho^\varepsilon(x) \notag\\
&=\int_{\supp \varphi}\nabla\varphi(x)\cdot \nabla V_\varepsilon* [\log ((1-\sigma)V_\varepsilon*\rhote + \sigma \mn) - \log (V_\varepsilon * \rho^\varepsilon + \sigma \mn)](x)\,\mathrm{d}\rhote(x) 	\label{eq:diff-F'} 	\\
&\qquad \qquad \qquad \qquad  + \int_{\supp \varphi}\nabla\varphi(x)\cdot \nabla V_\varepsilon* [\log((1-\sigma)V_\varepsilon * \rho^\varepsilon + \sigma \mn)](x)\,\diff[\rho_\tau^\varepsilon -\rho^\varepsilon](x). 	\label{eq:diff-rho}
\end{align}
Fix $x\in $ supp$\varphi$. Let us note that $d_1(V_\varepsilon*\rho_\tau^\varepsilon,V_\varepsilon*\rho^\varepsilon)\le d_1(\rhote,\rhoe)\to0$, as $\tau\to0$, due to narrow convergence of $\rhote$ and uniform integrability of first order moments of $\rhote$, cf.~\cite[Proposition 7.1.5]{AGS}. In particular, since $V_{\eps}$ is Lipschitz continuous, we have uniform convergence of $V_\varepsilon*\rho_\tau^\varepsilon \to V_\varepsilon*\rho^\varepsilon$. Then, by further using the continuity of $\log$ away from the origin
\begin{align*}
&\quad \left|\nabla V_\varepsilon*  [\log ((1-\sigma)V_\varepsilon*\rhote + \sigma \mn) - \log ((1-\sigma)V_\varepsilon * \rho^\varepsilon + \sigma \mn)](x)\right|     \\
&\le \int_{\{|y|\le \varepsilon\}} |\nabla V_\varepsilon(y)| |\log ((1\!-\!\sigma)V_\varepsilon*\rhote(x\!-\!y) \!+\! \sigma \mn(x\!-\!y))\! -\! \log ((1\!-\!\sigma)V_\varepsilon * \rho^\varepsilon(x\!-\!y) \!+\! \sigma \mn(x\!-\!y))| \mathrm{d}y  \\
&\le \int_{\{|y|\le \varepsilon\}} |\nabla V_\varepsilon(y)| \, \frac{|1-\sigma|}{\sigma\, \mathcal{N}(x-y)} \, |V_{\eps}\ast (\rho_\tau^\varepsilon - \rho^\varepsilon)|(x-y) \diff y \to 0 \mbox{ uniformly for all } x \in \supp \varphi,
\end{align*}
so that the term in \eqref{eq:diff-F'} converges to 0 by duality between measures and continuous functions. Turning to the integral in~\eqref{eq:diff-rho}, the convolution by $\nabla V_\varepsilon$ ensures that
\[
\nabla \varphi(x)\cdot\nabla V_\varepsilon*[\log((1-\sigma)V_\varepsilon*\rho^\varepsilon + \sigma \mn)](x)
\]
is continuous and bounded. Again since $\rhote$ narrowly converges to $\rho^\varepsilon$, as $\tau\to 0$, the term \eqref{eq:diff-rho} converges to 0. This shows~\eqref{eq:AEC}, so that $\rho^{\eps}$ satisfies~\Cref{def:weak-meas-sol-sigma} by passing to the limit $\tau\to 0$ in~\eqref{eq:weak-form-tau}.
\end{proof}
We now prove the analogous result in the case when $\sigma=0$ provided that the mollifier $V_1$ is globally supported with exponential tails (cf.~\ref{ass:v2-g}). The proof follows very similar lines and we only highlight technical differences.
\begin{thm}
\label{thm:exist_nlhe_sigma=0}
    Fix $\varepsilon>0$ and $V_1$ satisfying~\ref{ass:v1} and~\ref{ass:v2-g}. Suppose $\rho_0\in \mptrd$ satisfies $\mh^\varepsilon[\rho_0]<+\infty$. Then, there exists a weak measure solution $\rho^\varepsilon$ to~\eqref{eq:nonlocal_he_gauss} for $\sigma=0$ such that $\rho^\varepsilon(0) = \rho_0$.
\end{thm}
\begin{proof}[Sketch of the proof]
We emphasise the fact that this proof follows very closely the proof of~\Cref{thm:exist_nlhe_gauss} by replacing the assumption of $\sigma>0$ and compactly supported $V_1$.

Our candidate weak measure solution is $\rho^\varepsilon$ from~\Cref{lem:en-ineq-mom-bound} for the case $\sigma = 0$ and globally supported $V_1$. Existence of minimisers to~\eqref{eq:jko} is proven in~\cite[Proposition 3.1]{CEW_nl_to_local_24}. As in the proof of~\Cref{thm:exist_nlhe_gauss}, we consider the perturbation
\[
\rho^n \equiv \rhotne, \quad \rhotnn=\rhotnne, \quad \rho_\eta^{n+1} = P_\#^\eta\rho^{n+1}, \quad P^\eta(x) = x + \eta\, \zeta(x).
\]
Being $\rho^{n+1}$ a minimiser of~\eqref{eq:jko}, we get~\eqref{eq:optimality} for $\sigma=0$ i.e.
\[
\frac{1}{2\tau}\left[ \frac{d_W^2(\rho^n, \rho_\eta^{n+1} - d_W^2(\rho^n,\rho^{n+1})}{\eta} \right] + \frac{\mh^\varepsilon[\rho_\eta^{n+1}] - \mh^\varepsilon[\rho^{n+1}]}{\eta}\ge 0.
\]
The passage to the limit $\eta\to 0$ in the Wasserstein difference goes through exactly as in the proof of~\Cref{thm:exist_nlhe_gauss}. As for the energy functional terms, keeping in mind $V_\varepsilon*\rhotnn>0$ as explained in~\Cref{rem:gaussian}, we follow the calculation in~\eqref{eq:interaction-terms} except here we set
\[
M_\eta^\varepsilon(y) = \int_0^1\log [sV_\varepsilon * \rho_\eta^{n+1}(y) + (1-s)V_\varepsilon*\rho^{n+1}(y)]\, \mathrm{d}s.
\]
We obtain the $\sigma=0$ analogue of~\eqref{eq:conv_m} which is to say that
\[
M_\eta^\varepsilon(y) \to \log (V_\varepsilon*\rho^{n+1}(y)), \quad \text{a.e. }y\in\R^d \quad \text{as }\eta \to 0.
\]
This is justified by the fact that the same estimate as~\eqref{eq:major_m} holds in this case
\[
|M_\eta^\varepsilon(y)|\le C\langle y \rangle^p,\quad \text{for some }C>0,
\]
owing to~\Cref{lem:log_conv_glob_rescaling} and~\ref{ass:v2-g}. Thus, we can pass to the limit $\eta \to 0$ as before and set $\zeta = \nabla \varphi$ to obtain
\begin{align*}
    \int_{\R^d}\varphi(x) \, \mathrm{d}\rho_\tau^\varepsilon(t,x) &- \int_{\R^d}\varphi (x) \, \mathrm{d}\rho_\tau^\varepsilon(s,x) + \mathcal{O}(\tau^2)\\
    &=-\int_s^t\int_{\R^d}\nabla \varphi(x)\cdot \nabla V_\varepsilon* \log[V_\varepsilon*\rho_\tau^\varepsilon(r,\cdot)](x) \, \mathrm{d}\rho_\tau^\varepsilon(r,x) \, \mathrm{d}r.
\end{align*}
It remains to pass to the limit $\tau\to 0$. The two terms on the left-hand side are easily handled. As for the right-hand side, \Cref{lem:log_conv_glob_rescaling} and~\ref{ass:v2-g} assert that the logarithmic term grows at most quadratically, and the second moment bound in~\Cref{lem:en-ineq-mom-bound} yield
\[
\left|
\int \nabla \varphi(x)\cdot \nabla V_\varepsilon*\log[V_\varepsilon*\rho_\tau^\varepsilon(r,\cdot)](x)\mathrm{d}\rho_\tau^\varepsilon(r,x)
\right| \le C \int\langle x \rangle^p\mathrm{d}\rho_\tau^\varepsilon(r,x) \le C \in L^1([s,t];\mathrm{d}r).
\]
Thus, the spatial integral on the right-hand side has a uniform-in-$\tau$ majorant. It remains to prove the $\sigma=0$ analogue of~\eqref{eq:AEC} i.e.
\begin{equation*}
\begin{split}
\int_{\R^d}\nabla \varphi(x)\cdot \nabla V_\varepsilon*\log[V_\varepsilon*\rho_\tau^\varepsilon(r,\cdot)](x) &\mathrm{d}\rho_\tau^\varepsilon(r,x) \to \int_{\R^d}\nabla\varphi(x)\cdot \nabla V_\varepsilon*\log[V_\varepsilon*\rho^\varepsilon_r(\cdot)](x)\mathrm{d}\rho^\varepsilon_r(x),\\& \text{as } \tau\downarrow 0, \text{ a.e. }r\in[0,T].
\end{split}
\end{equation*}
This amounts to proving that the $\sigma=0$ analogues of~\eqref{eq:diff-F'} and~\eqref{eq:diff-rho} vanish in the limit $\tau\to 0$. The term corresponding to~\eqref{eq:diff-rho} is handled in the same way as the proof of~\Cref{thm:exist_nlhe_gauss} because
\[
\nabla \varphi(x)\cdot \nabla V_\varepsilon*\log[V_\varepsilon*\rho^\varepsilon](x)
\]
is continuous and bounded. As for~\eqref{eq:diff-F'}, we wish to show that
\begin{equation}
\label{eq:diff-F'-sigma=0}
\left|\int_{\text{supp}\varphi} \nabla \varphi(x)\cdot \nabla V_\varepsilon * [\log (V_\varepsilon*\rho_\tau^\varepsilon) - \log (V_\varepsilon*\rho^\varepsilon)]\, \mathrm{d}\rho_\tau^\varepsilon(x)\right| \overset{\tau\downarrow 0}{\to}0.
\end{equation}
Fix $\delta>0, x\in \text{supp}\varphi$, and take $R>0$ sufficiently large such that
\begin{equation}
    \label{eq:int_recip}
\int_{|y|\ge R}\langle y \rangle^{-(d+1)}\, \mathrm{d}y\le \delta.
\end{equation}
Arguing as for \eqref{eq:diff-F'}, we have uniform convergence of $V_\varepsilon*\rho_\tau^\varepsilon \to V_\varepsilon*\rho^\varepsilon$. Hence, for $|y|\le R$, take $\tau>0$ sufficiently small such that, one has
\begin{equation}
    \label{eq:log_diff}
    |\log(V_\varepsilon*\rho_\tau^\varepsilon(x-y)) - \log(V_\varepsilon*\rho^\varepsilon(x-y))| \le \delta.
\end{equation}
Now, the convolution within~\eqref{eq:diff-F'-sigma=0} can be estimated by partitioning the set into
\begin{align*}
    &\quad \left|\nabla V_\varepsilon*[\log(V_\varepsilon*\rho_\tau^\varepsilon) - \log(V_\varepsilon*\rho^\varepsilon)]\right|     \\
    &\le \left(\int_{\{|y| \le R\}}+\int_{\{|y| \ge R\}}\right)|\nabla V_\varepsilon(y)|\,|\log[V_\varepsilon*\rho_\tau^\varepsilon(x-y)] - \log [V_\varepsilon*\rho^\varepsilon(x-y)]|\, \mathrm{d}y    \\
    &\le \|\nabla V_\varepsilon\|_{L^1}\, \delta + C_\varepsilon\int_{\{|y|\ge R\}}\langle y\rangle^{-(d+3)}\langle x-y\rangle^p\, \mathrm{d}y.
\end{align*}
The first term in the inequality is treated exactly as in the proof of~\Cref{thm:exist_nlhe_gauss} using~\eqref{eq:log_diff}. The second term uses~\ref{ass:v2-g} to estimate $\nabla V_\varepsilon$ and~\Cref{lem:log_conv_glob_rescaling} to bluntly estimate the logarithmic terms. Owing to~\Cref{lem:peetre} and~\eqref{eq:int_recip}, we can continue the estimate to obtain
\begin{align*}
    &\quad |\nabla V_\varepsilon*[\log(V_\varepsilon*\rho_\tau^\varepsilon) - \log(V_\varepsilon*\rho^\varepsilon)](x)| \\
    &\le \|\nabla V_\varepsilon\|_{L^1}\, \delta + C_\varepsilon \left( \int_{\{|y|\ge R\}}\langle y\rangle^{-(d+3-p)}\,\mathrm{d}y \right)\langle x\rangle^p \le C_\varepsilon \, \delta \, \langle x\rangle^p,
\end{align*}
where we have absorbed $\|\nabla V_\varepsilon\|_{L^1}$ into the constant $C_\varepsilon$. Plugging this into the $\sigma=0$ analogue of~\eqref{eq:diff-F'} yields
\[
\left| \int_{\text{supp}\varphi} \nabla \varphi(x)\cdot \nabla V_\varepsilon*[\log(V_\varepsilon*\rho_\tau^\varepsilon) - \log(V_\varepsilon*\rho^\varepsilon)](x) \, \mathrm{d}\rho_\tau^\varepsilon(x) \right| \le C_\varepsilon\, \delta \int |\nabla \varphi(x)|\, \langle x\rangle^p \, \mathrm{d}\rho_\tau^\varepsilon(x) \le C_{\varepsilon,\varphi}\, \delta.
\]
This establishes~\eqref{eq:diff-F'-sigma=0}.
\end{proof}

\begin{rem}
In the case $\sigma\in(0,1)$ and the mollifier $V_1$ is globally supported (i.e. it satisfies~\ref{ass:v2-g}), the statement of~\Cref{thm:exist_nlhe_gauss} remains true. In this case, the estimation of~\eqref{eq:diff-F'} follows what was described in the proof just before this remark. There, one uses~\Cref{lem:log_gauss} instead of~\Cref{lem:log_conv_glob_rescaling}.
\end{rem}
In the subsequent section, we prove uniform-in-$\sigma,\varepsilon$ estimates which will be crucial to pass to the nonlocal-to-local limits in~\Cref{sec:nl_lo}.


\subsection{Uniform-in-$\sigma$,$\varepsilon$ estimates and compactness}
\label{sec:unif_sigma}
We remind the reader that each $\rho^{\varepsilon,\sigma}$ is a weak measure solution to~\eqref{eq:nonlocal_he_gauss} owing to~\Cref{thm:exist_nlhe_gauss}. The goal of this section is to prove that the sequence $\{\rhoes\}_{\varepsilon,\sigma}$ is compact with respect to the right topology (which we shall identify in the sequel) in order to pass to the limits~\eqref{eq:nonlocal_he_gauss}$\overset{\sigma \to 0}{\to}$\eqref{eq:nonlocal_he}$\overset{\varepsilon \to 0}{\to}$~\eqref{eq:heat_eq} and \eqref{eq:nonlocal_he_gauss}$\overset{\sigma,\varepsilon \to 0}{\longrightarrow}$~\eqref{eq:heat_eq}. Here and throughout the rest of the paper, we make the standing assumption that the initial condition $\rhoes(t=0) = \rho_0$ satisfies~\ref{a:fie} i.e.
$\mh[\rho_0]<+\infty$. To summarise, let us denote by $\rhoes :[0,T]\to \mptrd$ an absolutely continuous curve which is a weak-measure solution to~\eqref{eq:nonlocal_he_gauss} such that $\mh[\rhoes(0)]<+\infty$ (for fixed $\varepsilon>0$). This curve satisfies the bounds described in~\eqref{eq:sig_ent}.
Using the refined version of Ascoli-Arzel\`a in~\cite[Proposition 3.3.1]{AGS} we can prove compactness of $\{\rho^{\varepsilon,\sigma}\}_{\varepsilon,\sigma}$ in $C([0,T];\mptrd)$. 
\begin{prop}
\label{prop:limit-rho}
There exists an absolutely continuous curve $\tilde{\rho}^\varepsilon:[0,T]\to\mptrd$ such that the sequence $\{\rhoes\}_{\sigma}$ admits a subsequence $\{\rho^{\varepsilon,\sigma_k}\}_{k}$ such that $\rho^{\varepsilon, \sigma_k}(t)$ narrow converges to $\tilde{\rho}^\varepsilon(t)$ for any $t\in[0,T]$ as $k\to+\infty$. Up to passing to a subsequence in $\varepsilon$, $\tilde\rhoe(t)\rightharpoonup\tilde\rho(t)$, for any $t\in[0,T]$, as $\varepsilon\to 0$. Furthermore, the sequence $\{\rhoes\}_{\varepsilon,\sigma}$ admits a subsequence $\{\rho^{k}\}_k:=\{\rho^{\varepsilon_k,\sigma_k}\}_k$ such that 
$\rho^k(t)\rightharpoonup\tilde\rho(t)$ as $k\to\infty$, uniformly in time.
\end{prop}
\begin{proof}
The limits in $\sigma$ and $\varepsilon$ can be obtained following the proof in~\cite[Proposition 4.1]{BE22}, owing to~\Cref{lem:en-ineq-mom-bound} and~\eqref{eq:sig_ent} since we have the uniform bounds
    \[
\sup_{\varepsilon, \sigma}\mh_\sigma^\varepsilon[\rhoes(t)]\le (1-\sigma)\,\mh[\rho_0]+\sigma C_\mn \quad \mbox{and}\quad  \sup_{\varepsilon, \sigma}m_2(\rho^{\varepsilon,\sigma}(t))\le C,\quad \forall t\in[0,T].
\]
As for the diagonal limit, we can proceed using the 1-Wasserstein distance because of the uniform bound for the second order moments --- in this case narrow convergence is equivalent to the 1-Wassertein ones, see e.g.~\cite[Proposition~7.1.5]{AGS}. In particular, for every $k\in\mathbb{N}$ there exist $\varepsilon_k$ and $\sigma_k$ such that, for every $t\in[0,T]$, we have
\[
d_1(\rho^{\varepsilon_k,\sigma_k}_t,\tilde\rho_t)\le d_1(\rho_t^{\varepsilon_k,\sigma_k},\tilde\rho_t^{\varepsilon_k})+d_1(\tilde\rho_t^{\varepsilon_k},\tilde\rho_t)\le\frac{1}{k}.
\]
We observe for completeness that the limit is uniform in time in view of the equicontinuity
\[
d_1^2(\rhoes_t,\rhoes_s)\le d_W^2(\rho^{\varepsilon,\sigma}_t, \rho^{\varepsilon,\sigma}_s) \leq C|t-s|.
\]
\end{proof}
Since $\rhoes(t)\in\mptrd$, for any $t\in[0,T]$, we employ a further regularisation to obtain the limit as a density --- as it should be since it will be the solution of~\eqref{eq:heat_eq}. We consider $v_\tau^{\varepsilon,\sigma}:=V_\varepsilon*\rhotes$ and prove higher regularity and suitable compactness for the latter in order to pass to the limit in the weak form of the equations. More precisely, we focus on the $\sigma$-perturbation $(1-\sigma)v_\tau^{\varepsilon,\sigma}+\sigma\mathcal{N}$, since for $\sigma=0$ and $\varepsilon>0$ higher regularity was proven in~\cite[Lemma 4.1]{CEW_nl_to_local_24}. This task is performed by using the so-called flow-interchange technique, cf.~\cite{MMCS}. We use as auxiliary flow the heat semigroup, denoted by $S_\mh$, which is a $0$-flow.
\begin{defn}[$\lambda$-flow]\label{def:lambda-flow}
	A semigroup $S_{\mathcal{E}}:[0,+\infty]\times\mptrd\to\mptrd$ is a $\lambda$-flow for a functional $\mathcal{E}:\mptrd\to\R\cup\{+\infty\}$ with respect to the distance $d_W$ if, for an arbitrary $\rho\in\mptrd$, the curve $t\mapsto S_{\mathcal{E}}^t\rho$ is absolutely continuous on $[0,+\infty[$ and it satisfies the so-called Evolution Variational Inequality (\textbf{E.V.I.})
	\begin{equation}
		\frac{1}{2}\frac{d^+}{dt}d_W^2(S_{\mathcal{E}}^t\rho,\bar{\rho})+\frac{\lambda}{2}d_W^2(S_{\mathcal{E}}^t\rho,\bar{\rho})\le \mathcal{E}(\bar{\rho})-\mathcal{E}(S_{\mathcal{E}}^t\rho)
	\end{equation}
	for all $t\ge 0$, with respect to every reference measure $\bar{\rho}\in\mptrd$ such that $\mathcal{E}(\bar{\rho})<\infty$.
\end{defn}
In the following result, for any $\nu\in\mptrd$ such that $\mh(\nu)<+\infty$, we denote by $S_{\mh}^t\nu$ the solution at time $t$ of the heat equation coupled with an initial value $\nu$ at $t=0$. Moreover, for every $\rho\in\mptrd$, we define the dissipation of $\mh_\sigma^\varepsilon$ along $S_{\mh}$ by
$$
D_{\mh}\mh^\varepsilon_\sigma(\rho):=\limsup_{s\downarrow0}\left\{\frac{\mh_\sigma^\varepsilon[\rho]-\mh_\sigma^\varepsilon[S_{\mh}^s\rho]}{s}\right\}.
$$
\begin{lem}\label{lem:h1_estimate}
Let $\varepsilon>0$, $\sigma \in [0,1)$. Let $\rho_0\in\mpdtard$ be such that $\mh[\rho_0]<\infty$. There exists a constant $C=C(\mn, \rho_0, V_1, T) >0$ such that, with $v_\tau^{\varepsilon,\sigma}=V_\varepsilon*\rhotes$, we have
\begin{equation}\label{eq:h_1_bound}
\sup_{\varepsilon,\,\sigma, \,  \tau>0}\left\|
\sqrt{(1-\sigma)v_\tau^{\varepsilon,\sigma}+\sigma\mn}
\right\|_{L^2(0,T;\, H^1(\R^d))} \le C.
\end{equation}
\end{lem}
\begin{proof}
We note that
\begin{align*}
	\left\| \sqrt{(1-\sigma)v_\tau^{\varepsilon,\sigma}+\sigma\mn}\right\|_{L^2(0,T;\, L^2(\R^d))}^2 = \int_0^T\int_{\R^d}(1-\sigma) V_\varepsilon*\rho_\tau^{\varepsilon,\sigma}(t,x) +\sigma\mn(x) \diff x \diff t = T,
\end{align*}
since  $\|\mn\|_{L^1}=\|V_\varepsilon\|_{L^1} = \int_{\R^d}\mathrm{d}\rho_\tau^{\varepsilon,\sigma}(t)(x) = 1$. Next, we need to derive the uniform bound for $\nabla \sqrt{(1-\sigma)v_\tau^{\varepsilon,\sigma}+\sigma\mn}$. Let $s>0$ and consider $S_\mh^s\rhotnnes$ as a competitor against $\rhotnnes$ in the minimisation problem~\eqref{eq:jko}. By definition of the scheme, it follows
\[
\frac{1}{2\tau}d_W^2(\rhotnnes,\rhotnes)+\mh^\varepsilon_\sigma[\rhotnnes]\le\frac{1}{2\tau}d_W^2(S_{\mh}^s\rhotnnes,\rhotnes)+\mh_\sigma^\varepsilon[S_{\mh}^s\rhotnnes],
\]
which, dividing by $s>0$ and passing to $\limsup_{s\downarrow 0}$, gives
\begin{equation}\label{eq:flow-interchange}
	\tau D_{\mh}\mh_\sigma^\varepsilon(\rhotnnes)\le\left.\frac{1}{2}\frac{d^+}{dt}\right|_{t=0}\Big(d_W^2(S_{\mh}^t\rhotnnes,\rhotnes)\Big)\overset{\bm{(E.V.I.)}}{\le}\mh[\rhotnes]-\mh[\rhotnnes],
\end{equation}
since $S_\mh$ is a $0$-flow. The computation below (cf.~\eqref{eq:deriv-dis}) will show that $\mh[\rhotnes]<\infty$ for all $n$. The $H^1$ bound will be clear after making the left-hand side of~\eqref{eq:flow-interchange} explicit. Note that
\begin{equation}\label{eq:integral-form-dis}
	\begin{split}
		D_{\mh}\mh_\sigma^\varepsilon(\rhotnnes)&=\limsup_{s\downarrow0}\left\{\frac{\mh_\sigma^\varepsilon[\rhotnnes]-\mh_\sigma^\varepsilon[S_{\mh}^s\rhotnnes]}{s}\right\}\\&=\limsup_{s\downarrow0}\int_0^1\left(-\frac{d}{dz}\Big|_{z=st}\mh_\sigma^\varepsilon[S_{\mh}^{z}\rhotnnes]\right)\diff t.
	\end{split}
\end{equation}
Thus, we now compute the time derivative inside the above integral, taking advantage of the $C^\infty$ regularity of the heat semigroup and that $S_{\mh}^t\rhotnnes>0$:
\begin{equation}\label{eq:deriv-dis}
	\begin{split}
		\quad \frac{d}{dt}\mh_\sigma^\varepsilon[S_{\mh}^t\rhotnnes] &= \!-\!\int_\Rd (1-\sigma)\nabla (V_\varepsilon * S_{\mh}^t\rhotnnes)\cdot \nabla \log ((1-\sigma)V_\varepsilon*S_{\mh}^t\rhotnnes + \sigma\mn)\,\diff x    \\
    &= \!-\!\int_\Rd\!\! \nabla ((1\!-\!\sigma)V_\varepsilon * S_{\mh}^t\rhotnnes + \sigma \mn) \cdot \nabla \log ((1\!-\!\sigma)V_\varepsilon*S_{\mh}^t\rhotnnes + \sigma\mn)\diff x      \\
    &\qquad \qquad +\sigma \int_\Rd \nabla \mn \cdot \nabla \log ((1-\sigma)V_\varepsilon*S_{\mh}^t\rhotnnes + \sigma\mn)\,\diff x     \\
    &= \!-4\int_\Rd \left|
\nabla \sqrt{(1-\sigma)V_\varepsilon*S_{\mh}^t\rhotnnes + \sigma \mn}
    \right|^2\,\diff x \\
    &\qquad \qquad + \sigma \int_\Rd \nabla \mn \frac{\nabla[(1-\sigma)V_\varepsilon*S_{\mh}^t\rhotnnes + \sigma\mn]}{(1-\sigma)V_\varepsilon*S_{\mh}^t\rhotnnes + \sigma\mn}\,\diff x\\
   &=  \!-4\int_\Rd \left|
\nabla \sqrt{(1-\sigma)V_\varepsilon*S_{\mh}^t\rhotnnes + \sigma \mn}
    \right|^2\,\diff x \\
    &\qquad \qquad + 2\sigma \int_\Rd \nabla \mn \frac{\nabla\sqrt{(1-\sigma)V_\varepsilon*S_{\mh}^t\rhotnnes + \sigma\mn}}{\sqrt{(1-\sigma)V_\varepsilon*S_{\mh}^t\rhotnnes + \sigma\mn}}\,\diff x\\
    &\le\!-3\int_\Rd \left|
\nabla \sqrt{(1-\sigma)V_\varepsilon*S_{\mh}^t\rhotnnes + \sigma \mn}
    \right|^2\,\diff x+4\sigma\int_\Rd\left|\nabla\sqrt{\mn}\right|^2\,\diff x,
    \end{split}
\end{equation}
where in the last inequality we used Young's inequality and the regularity of $\mn$. By substituting~\eqref{eq:deriv-dis} into~\eqref{eq:integral-form-dis}, from~\eqref{eq:flow-interchange} we obtain
\[
3\tau\liminf_{s\downarrow0}\int_0^1\int_{\Rd}\left|\nabla \sqrt{(1-\sigma)V_\varepsilon*S_{\mh}^{st}\rhotnnes+\sigma\mn}\right|^2\,\diff x\diff t\le\mh[\rhotnes]-\mh[\rhotnnes] + C(\mn).
\]
The weak $L^2$ lower semi-continuity of the $H^1$ semi-norm gives
\[
\tau\int_{\Rd}\left|\nabla \sqrt{(1-\sigma)V_\varepsilon*\rhotnnes+\sigma\mn}\right|^2\,\diff x\le\frac{1}{3}\left(\mh[\rhotnes]-\mh[\rhotnnes]\right) +C(\mn).
\]
By summing up over $n$ from $0$ to $N-1$, taking into account the entropy is bounded from below by second order moments (see for example~\cite[Remark 4.2]{CEW_nl_to_local_24}) which are uniformly bounded (see~\Cref{lem:en-ineq-mom-bound}), we get
\begin{equation}\label{eq:bound-nablaveps}
\begin{split}
	\int_0^T\int_{\Rd}\left|\nabla\sqrt{(1-\sigma)V_\varepsilon*\rhotes(t)+\sigma\mn}\right|^2\diff x\diff t&\le \frac{1}{3}\left(\mh[\rho_0]-\mh[\rhotnes]\right)+C\\
    &\le \frac{1}{3}\left(\mh[\rho_0] + C(\mn, \rho_0, \, V_1, \, T)\right).
 \end{split}
\end{equation}
Since the initial entropy is assumed to be bounded, this concludes the proof.
\end{proof}
The previous $L^2_tH^1_x$ bound is important to obtain strong compactness in $L^1([0,T]\times\Rd)$, by applying the compactness result by Rossi and Savaré, cf.~\cite[Theorem 2]{RS}, recalled in the appendix as~\Cref{thm:aulirs-meas-appendix} for the reader's convenience. Let us denote 
\[
f_\tau^{\varepsilon,\sigma}:=(1-\sigma)v_\tau^{\varepsilon,\sigma}+\sigma\mathcal{N}, \quad\text{recalling}\quad v_\tau^{\varepsilon,\sigma} = V_\varepsilon * \rho_\tau^{\varepsilon,\sigma}.
\]

\begin{prop}\label{prop:strong-convergence-v}
Consider the family $\{f_\tau^{\varepsilon,\sigma}\}_{\sigma\in(0,1),\varepsilon\in(0,\varepsilon_0), \tau>0}$. The following results hold true: 
\begin{enumerate}[label=\roman*)]
    \item There exists a subsequence $\tau_k\downarrow 0$ such that for any $\sigma \in (0,1),\varepsilon>0$, we have
\[
f_{\tau_k}^{\varepsilon,\sigma} \to f^{\varepsilon,\sigma} = (1-\sigma )V_\varepsilon*\tilde{\rho}^{\sigma,\varepsilon} + \sigma\mn, \quad \text{in }L^1([0,T]\times \R^d).
\]
\item There exists a subsequence $\sigma_k\downarrow0$ such that, for any $\varepsilon>0$, 
\[
f^{\varepsilon,\sigma_k} \to V_\varepsilon*\tilde{\rho}^{\varepsilon}, \quad \text{in }L^1([0,T]\times \R^d).
\]
\item There exists a subsequence $\varepsilon_k\downarrow 0$ and a curve $v\in C([0,T];\mptrd)\cap L^1([0,T]\times \R^d)$ such that 
\[
v^{\varepsilon} := V_\varepsilon * \rho^{\varepsilon}  \to v, \quad \text{in }L^1([0,T]\times\R^d).
\]
\item There exist subsequences $\varepsilon_k\downarrow 0$ and $\sigma_k\downarrow 0$ such that
\[
f^{\varepsilon_k,\sigma_k}\overset{k\to \infty}{\to} v, \quad \mbox{in } L^1([0,T]\times\Rd),
\]
for a curve $v\in C([0,T];\mptrd)\cap L^1([0,T]\times \R^d)$.
\end{enumerate}
\end{prop}
\begin{proof}
    The proof follows the strategy applied in~\cite[Proposition 4.3]{CEW_nl_to_local_24}. It is only worth to notice that $d_1((1-\sigma) \mu_1 +\sigma\nu,(1-\sigma) \mu_2+\sigma \nu)=(1-\sigma) d_1(\mu_1,\mu_2)$, for $\mu,\nu\in\mathcal{P}_1(\Rd)$, thus the weak integral equicontinuity follows. As for the convergence in $\sigma$ we can apply the Lebesgue Dominated Convergence Theorem for $\varepsilon$ fixed and obtain $f^{\varepsilon,\sigma_k}\to V_{\varepsilon}*\tilde{\rho}^\varepsilon$ as $\sigma_k\to0$, taking into account~\Cref{prop:limit-rho}. Compactness in $\varepsilon$ is exactly as in~\cite[Proposition 4.3]{CEW_nl_to_local_24}, upon noticing 
    \[\sup_{\varepsilon>0}\left\|
\sqrt{v^{\varepsilon}}
\right\|_{L^2(0,T;\, H^1(\R^d))} \le C,
\]
for a constant independent of $\varepsilon$. This can be seen as a consequence of the strong convergence just proven, the Lebesgue Dominated Convergence Theorem, and weak $L^2$ lower semicontinuity --- a precise argument can be found in~\cite[Proposition 4.3]{CEW_nl_to_local_24}. The last property follows by the triangle inequality.
\end{proof}

\section{Nonlocal-to-local limit for the heat equation}\label{sec:nl_lo}

In this section we prove convergence of weak solutions $\rho^{\eps,\sigma}$ to~\eqref{eq:nonlocal_he_gauss} by establishing two limits: $\sigma, \eps \to 0$ jointly (Section \ref{sect:joint_limit}, Theorem \ref{thm:joint_limit_eps_sigma}) and $\sigma \to 0$ first, then $\eps \to 0$ (Section \ref{sect:separate_limits}, Theorems \ref{thm:sig_to_zero} and \ref{thm:eps_to_zero}). 

\subsection{Commutator estimates}
We shall consider kernels being either compactly or globally supported --- in the latter case assuming additionally that their second moment is finite. This is relevant when the kernel $V=\mn$ taking into account~\Cref{rem:gaussian}. Owing to the regularity of the vector field far from the vacuum, we can rewrite the term~\eqref{eq:weak-form-sigma_towork_with} (cf. the discussion after~\eqref{eq:weak-form-sigma_towork_with}) as
\begin{equation}\label{eq:weak-form-conv}
\begin{split}
&2(1-\sigma)\int_0^t\int_\Rd\nabla\varphi\cdot V_\varepsilon*\left[\frac{\nabla\sqrt{(1-\sigma)V_\varepsilon*\rho_r^{\varepsilon,\sigma}+\sigma \mn}}{\sqrt{(1-\sigma)V_\varepsilon*\rho_r^{\varepsilon,\sigma}+\sigma \mn}}\right]\diff \rho_r^{\varepsilon,\sigma}(x) \diff r\\
&\qquad=2(1-\sigma)\int_0^t\int_\Rd V_\varepsilon*(\rho_r^{\varepsilon,\sigma}\nabla\varphi)\frac{\nabla\sqrt{(1-\sigma)V_\varepsilon*\rho_r^{\varepsilon,\sigma}+\sigma \mn}}{\sqrt{(1-\sigma)V_\varepsilon*\rho_r^{\varepsilon,\sigma}+\sigma \mn}} \diff x \diff r.
\end{split}
\end{equation}
Weak solutions to~\eqref{eq:nonlocal_he}~and~\eqref{eq:heat_eq} will be obtained upon analysing the convergence of this term, involving the so-called commutator estimate. For the reader's convenience, we collect the a priori estimates from the construction of solutions to~\eqref{eq:nonlocal_he_gauss} in the previous section.

\begin{lem}[a priori estimates]\label{lem:apriori_est_sigma_eps}
Let $\{\rho^{\varepsilon, \sigma}\}_{\varepsilon,\sigma}$ be a sequence of measure solutions to~\eqref{eq:nonlocal_he_gauss}, and~\eqref{eq:nonlocal_he} for $\sigma=0$. The following properties hold uniformly in $\sigma$ and $\varepsilon$:
\begin{enumerate}[label=(\Alph*)]
    \item\label{est1} $\{\rho^{\varepsilon,\sigma}\}_{\varepsilon,\sigma}\subset C([0,T];\mptrd)$ such that $m_2(\rho^{\varepsilon,\sigma})\in L^\infty([0,T])$ and 
    \begin{equation}\label{eq:cont_in_time_Wasserstein_eps_sigma}
d_W^2(\rho^{\varepsilon, \sigma}_t, \rho^{\varepsilon, \sigma}_s) \leq C|t-s|,
    \end{equation}
    for any $t,s\in[0,T]$ and a constant $C$ independent of $\sigma$ and $\eps$.
    \item\label{est2} $\{(1-\sigma)\, \rho^{\varepsilon, \sigma}\ast V_{\varepsilon} + \sigma\, \mathcal{N}\}_{\varepsilon,\sigma}$ is bounded in $L^{\infty}_t L^1_x$ and  $L^{\infty}_t (L\text{Log}L)_x$.
    \item\label{est2.5} $\{|\cdot|^2(V_{\varepsilon}\ast\rho^{\varepsilon, \sigma})\}_{\varepsilon,\sigma}$ is bounded in $L^{\infty}_t L^1_x$.
    \item\label{est3} $\left\{\sqrt{(1-\sigma)\, \rho^{\varepsilon, \sigma}\ast V_{\varepsilon} + \sigma\, \mathcal{N}}\right\}_{\sigma,\varepsilon}$ is bounded in $L^2_t H^1_x$. 
\end{enumerate}
\end{lem}
\begin{proof}
Properties~\ref{est1}~and~\ref{est2}~are obtained from the construction of solutions via the JKO scheme, defined in~\eqref{eq:jko}, whereas~\ref{est3}~is proven in~\Cref{lem:h1_estimate} for $f_\tau^{\varepsilon,\sigma}$. Strong convergence in $\tau$ and the $L^2_tH^1_x$ estimate imply the estimate holds true for $f^{\varepsilon,\sigma}$, i.e.~\ref{est3}. We only need to prove \ref{est2.5}. To this aim we estimate pointwise
    $$
|x|^2(\rho^{\varepsilon, \sigma}\ast V_{\eps})(x) \leq 2\, (\rho^{\varepsilon, \sigma}|\cdot|^2)\ast V_{\eps}(x)+ 2\, \rho^{\varepsilon, \sigma} \ast (V_{\eps}|\cdot|^2)(x)  
    $$
    so that 
    \begin{align*}
    \| |\cdot|^2\rho^{\varepsilon, \sigma}\ast V_{\eps}\|_{L^{\infty}_t L^1_x} &\leq 
    2\,\| (\rho^{\varepsilon, \sigma}|\cdot|^2)\ast V_{\eps} \|_{L^{\infty}_t L^1_x} + 2\, \| \rho^{\varepsilon, \sigma} \ast (V_{\eps}|\cdot|^2) \|_{L^{\infty}_t L^1_x}\\
    &\leq 2\|m_2(\rho^{\varepsilon, \sigma})\|_{L^{\infty}_t} \, \|V\|_{L^1_x} + 2\eps^2\,\| V|\cdot|^2\|_{L^{1}_x}.
    \end{align*}
    We conclude by applying~\ref{est1}~and the assumption $V|\cdot|^2 \in L^1_x$.
\end{proof}
In order to prove both limits the following commutator estimate is crucial. Although this approach is by now well-understood, we present a technical novelty since we can treat globally supported kernels in case of log singularity or power laws with $1<m<2$, which is new to the best of our knowledge. This improves upon the results of~\cite{CEW_nl_to_local_24}.
\begin{lem}[Commutator lemma]\label{lem:commutator_lemma}
Let $\rho\in C([0,T];\mptrd)$ and suppose $V_1$ satisfies~\ref{ass:v1},~\ref{ass:v2}. There exists a constant $C$ depending on $V_1$ such that, for any $\varphi\in C^2_c$,
\begin{equation}\label{eq:commutator_general_rho}
\left\|\mathds{1}_{V_{\eps}\ast\rho>0} \, \frac{V_\varepsilon * (\nabla \varphi\,  \rho) - \nabla \varphi\, (V_{\eps}\ast\rho)}{\sqrt{V_\varepsilon * \rho}}\right\|_{L^2_{t,x}} \leq \varepsilon\, C\, \|D^2\varphi\|_{L^2_t L^\infty_x} .
\end{equation}
\end{lem}
\begin{proof}
As a preliminary observation, note that by nonnegativity of $\rho$ and $V_{\eps}$ we have the pointwise identity
\begin{multline*}
|V_\varepsilon * (\nabla \varphi\,  \rho)(x) - \nabla \varphi\, V_{\eps}\ast\rho(x)| \leq \int_{\R^d} V_{\eps}(x-y) \,  |\nabla \varphi(y) - \nabla \varphi(x)| \diff \rho(y) \leq \\ \leq \|D^2 \varphi\|_{L^\infty} \int_{\R^d} V_{\eps}(x-y)\, |x-y|\, \diff \rho(y) = \|D^2 \varphi\|_{L^\infty_x} \, (V_{\eps}|\cdot|)\ast \rho(x). 
\end{multline*}
It is then sufficient to estimate $\mathds{1}_{V_{\eps}\ast\rho>0}\, \frac{(V_{\eps}|\cdot|)\ast\rho}{\sqrt{V_{\eps}\ast\rho}}$ in $L^{\infty}_t L^2_x$. We distinguish between two cases.

\noindent \underline{\textit{Globally supported kernels $V$.}} 
We introduce a function $f_{\eps} = f_{\eps}(t,x) \geq 0$ to be specified later. Then, we estimate pointwise
\begin{equation}\label{eq:estimate_pointwise_after_commutator}
\frac{(V_{\eps}|\cdot|)\ast \rho(x)}{\sqrt{V_{\eps}\ast\rho}(x)} = \int_{\R^d} V_{\eps}(x-y)\, \frac{|x-y|\,\sqrt{f_{\eps}(t,x)}}{\sqrt{V_{\eps}\ast\rho}(x)\, \sqrt{f_{\eps}(t,x)}} \, \diff \rho(y).
\end{equation}
Using Young's inequality
$$
\frac{|x-y|\,\sqrt{f_{\eps}(t,x)}}{\sqrt{V_{\eps}\ast\rho}(x)\, \sqrt{f_{\eps}(t,x)}} \leq \frac{1}{2}\,\frac{|x-y|^2}{f_{\eps}(t,x)} + \frac{1}{2}\,\frac{f_{\eps}(t,x)}{(V_{\eps}\ast\rho)(x)}\,,
$$
we deduce from \eqref{eq:estimate_pointwise_after_commutator}
$$
\frac{(V_{\eps}|\cdot|)\ast \rho}{\sqrt{V_{\eps}\ast\rho}} \leq \frac{1}{2} \, \frac{(V_{\eps}|\cdot|^2)\ast\rho}{f_{\eps}} + \frac{1}{2}\,f_{\eps} = \sqrt{(V_{\eps}|\cdot|^2)\ast\rho}\,,
$$
where we chose $f_{\eps} = \sqrt{(V_{\eps}|\cdot|^2)\ast\rho}$. To conclude the proof, we estimate, for any $t\in[0,T]$
$$
\left\| \frac{(V_{\eps}|\cdot|)\ast \rho_t}{\sqrt{V_{\eps}\ast\rho_t}}  \right\|_{L^2_x}^2 \leq  \| (V_{\eps}|\cdot|^2)\ast\rho_t \|_{L^1_x} = \| V_{\eps}|\cdot|^2 \|_{L^1_x}= \eps^2\, \|V_1|\cdot|^2\|_{L^1_x}.
$$

\noindent \underline{\textit{The case of compactly supported kernels $V$.}} Without loss of generality we assume that $V$ is supported in the unit ball so that $V_{\eps}$ is supported in the ball of radius $\eps$. Hence,
$$
\frac{(V_{\eps}|\cdot|)\ast \rho}{\sqrt{V_{\eps}\ast\rho}} \leq \eps \, \sqrt{V_{\eps}\ast\rho}.
$$
Arguing as above, for any $t\in[0,T]$,
$$
\left\| \frac{(V_{\eps}|\cdot|)\ast \rho_t}{\sqrt{V_{\eps}\ast\rho_t}}  \right\|_{L^2_x}^2 \leq \eps^2 \, \| V_{\eps}\ast\rho_t \|_{L^1_x} \leq \eps^2\, \|V\|_{L^1_x},
$$
which concludes the proof of~\eqref{eq:commutator_general_rho}.
\end{proof}
The identification of the limiting curve $\tilde\rho$ from~\Cref{prop:limit-rho} as a weak solution of the heat equation is a consequence of the following lemma. Indeed, we see that $v^{\varepsilon,\sigma}$ and $\rhoes$ have the same limit in a distributional sense, both for $\sigma>0$ and $\sigma=0$. 
\begin{lem}\label{lem:limit-dist}
For any $t\in[0,T]$ and any $\varphi\in C_c^1(\Rd)$ it holds
$$
\lim_{\varepsilon\to0^+}\int_\Rd \varphi(x)v_t^\varepsilon(x)\,\diff x=\int_{\Rd}\varphi(x)\diff \tilde{\rho}_t(x) =
\lim_{\varepsilon,\sigma\to0^+}\int_\Rd \varphi(x)v_t^{\varepsilon,\sigma}(x)\,\diff x.
$$
\end{lem}
\begin{proof}
For any $t\in[0,T]$ and any $\varphi\in C_c^1(\Rd)$, by using the definition of $v_t^\varepsilon$ we obtain:
\begin{align*}
\left|\int_\Rd\varphi(x)v_t^\varepsilon(x)\,\diff x-\int_\Rd\varphi(x)\,\mathrm{d}\rhoe_t(x)\right|&=\left|\int_\Rd\varphi(x)(V_\varepsilon*\rhoe_t)(x)\,\diff x-\int_\Rd\varphi(x)\,\mathrm{d}\rhoe_t(x)\right|\\
&=\left|\int_\Rd(\varphi*V_\varepsilon)(x)\,\mathrm{d}\rhoe_t(x)-\int_\Rd\varphi(x)\,\mathrm{d}\rhoe_t(x)\right|\\
&=\left|\int_\Rd[(\varphi*V_\varepsilon)(x)-\varphi(x)]\,\mathrm{d}\rhoe_t(x)\right|\\
&\le\int_\Rd\int_\Rd|\varphi(x-y)-\varphi(x)|V_\varepsilon(y)\,\mathrm{d}y\,\mathrm{d}\rhoe_t(x)\\
&\le\|\nabla\varphi\|_\infty\int_\Rd |y|V_\varepsilon(y)\,\mathrm{d}y\\
&=\varepsilon\|\nabla\varphi\|_\infty\int_\Rd |x|V_1(x)\,\diff x,
\end{align*}
which converges to $0$ as $\varepsilon\to0^+$ since $\int_\Rd|x|V_1(x)\,\diff x<+\infty$. An analogous proof is valid for $v^{\varepsilon,\sigma}$.
\end{proof}

\subsection{Joint limit $\sigma, \eps \to 0$}\label{sect:joint_limit} Here, we prove the convergence in the joint limit $\sigma, \eps \to 0$. Note that we do not require any particular scaling between the parameters.

\begin{thm}\label{thm:joint_limit_eps_sigma}
Let $\rho_0\in\mpdtard$ such that $\mh[\rho_0]<\infty$, $V_1$ satisfy~\ref{ass:v1}~--~\ref{ass:v2}, and $\{\rho^{\eps,\sigma}\}_{\varepsilon,\sigma}$ be a sequence of solutions to \eqref{eq:nonlocal_he_gauss} from~\Cref{thm:exist_nlhe_gauss}. The sequence $\rhoes(t)$ narrowly converges to the unique weak solution of the heat equation,~\eqref{eq:heat_eq}, as $\varepsilon,\sigma\to0$, uniformly in time.
\end{thm}
\begin{proof}
To prove the result, we consider a smooth, compactly supported function $\varphi$ and consider the weak form of~\eqref{eq:nonlocal_he_gauss}. We need to pass to the limit in three expressions. Due to the narrow convergence, $\int_{\R^d} \varphi(T,x) \diff \rho^{\varepsilon, \sigma}_T(x) \to \int_{\R^d} \varphi(T,x) \diff \rho_T(x)$. Similarly, by the dominated convergence theorem
$$
\int_0^T \int_{\R^d} \partial_t \varphi \diff \rho^{\varepsilon, \sigma}_t(x)  \diff t \to \int_0^T \int_{\R^d} \partial_t \varphi  \diff \rho_t(x) \diff t.
$$
We now focus on the term including the velocity field given by~\eqref{eq:weak-form-conv} and notice
\begin{align*}
\mathcal{C}:&= -2(1-\sigma)\int_0^T \int_{\R^d} V_{\eps}\ast\left(\nabla \varphi \, \rho^{\varepsilon, \sigma}\right) \, \frac{\nabla \sqrt{(1-\sigma)\,V_{\eps}\ast \rho^{\varepsilon, \sigma} + \sigma\, \mathcal{N})}}{\sqrt{(1-\sigma)\,V_{\eps}\ast \rho^{\varepsilon, \sigma} + \sigma\, \mathcal{N}}} \diff x \diff t\\
&=-2(1-\sigma)\int_0^T \int_{\R^d} \mathds{1}_{V_{\eps}\ast\rhoes>0}V_{\eps}\ast\left(\nabla \varphi \, \rho^{\varepsilon, \sigma}\right) \, \frac{\nabla \sqrt{(1-\sigma)\,V_{\eps}\ast \rho^{\varepsilon, \sigma} + \sigma\, \mathcal{N})}}{\sqrt{(1-\sigma)\,V_{\eps}\ast \rho^{\varepsilon, \sigma} + \sigma\, \mathcal{N}}} \diff x \diff t,
\end{align*}
since $|V_{\eps}\ast\left(\nabla \varphi \, \rho^{\varepsilon, \sigma}\right)(x)|\le \|\nabla\varphi\|_\infty V_\varepsilon\ast\rhoes(x)$. Moreover, we observe 
$$
\left|\frac{(\nabla \varphi  \, \rho^{\varepsilon, \sigma})\ast V_{\eps} - \nabla \varphi\, (\rho^{\varepsilon, \sigma} \ast V_{\eps})}{\sqrt{(1-\sigma)\,V_{\eps}\ast \rho^{\varepsilon, \sigma} + \sigma\, \mathcal{N}}}\right| \leq \frac{\sqrt{(1-\sigma)\,V_{\eps}\ast \rho^{\varepsilon, \sigma} }}{\sqrt{(1-\sigma)\,V_{\eps}\ast \rho^{\varepsilon, \sigma} + \sigma\, \mathcal{N}}} 
\left|\frac{(\nabla \varphi \, \rho^{\varepsilon, \sigma})\ast V_{\eps} - \nabla \varphi \, (\rho^{\varepsilon, \sigma} \ast V_{\eps})}{\sqrt{(1-\sigma)\,V_{\eps}\ast \rho^{\varepsilon, \sigma}} }\right|,
$$
where the first fraction is bounded by 1 and the second can be estimated as in~\Cref{lem:commutator_lemma}. It follows that 
$$
\left\|\mathds{1}_{V_{\eps}\ast\rhoes>0}\frac{(\nabla \varphi \, \rho^{\varepsilon, \sigma})\ast V_{\eps} - \nabla \varphi \, (\rho^{\varepsilon, \sigma} \ast V_{\eps})}{\sqrt{(1-\sigma)\,V_{\eps}\ast \rho^{\varepsilon, \sigma} + \sigma\, \mathcal{N}}}\right\|_{L^2_{t,x}} \leq \frac{C\, \eps}{\sqrt{1-\sigma}} \to 0
$$
as $\eps \to 0$, whence, thanks to the bound \ref{est3}, the expression $\mathcal{C}$ is vanishingly close to
\begin{align*}
\mathcal{C}' := & -2(1-\sigma)\int_0^T \int_{\R^d}\mathds{1}_{V_{\eps}\ast\rho>0}\nabla \varphi  \, (\rho^{\varepsilon, \sigma}\ast V_{\eps}) \, \frac{\nabla \sqrt{(1-\sigma)\,V_{\eps}\ast \rho^{\varepsilon, \sigma} + \sigma\, \mathcal{N})}}{\sqrt{(1-\sigma)\,V_{\eps}\ast \rho^{\varepsilon, \sigma} + \sigma\, \mathcal{N}}} \diff x \diff t.
\end{align*}
By adding and subtracting $\sigma\mn$, we get 
\begin{align*}
    \mathcal{C}' =& -2 \int_0^T \int_{\R^d} \nabla \varphi  \,\sqrt{(1-\sigma)\,V_{\eps}\ast \rho^{\varepsilon, \sigma} + \sigma\, \mathcal{N}} 
 \nabla \sqrt{(1-\sigma)\,V_{\eps}\ast \rho^{\varepsilon, \sigma} + \sigma\, \mathcal{N}} \diff x \diff t \\
 & + 2 \int_0^T \int_{\R^d} \nabla \varphi \, \frac{\sigma\mathcal{N}}{\sqrt{(1-\sigma)\,V_{\eps}\ast \rho^{\varepsilon, \sigma} + \sigma\, \mathcal{N}}} \, \nabla \sqrt{(1-\sigma)\,V_{\eps}\ast \rho^{\varepsilon, \sigma} + \sigma\, \mathcal{N}} \diff x \diff t.
\end{align*}
In the first term, we integrate by parts and we easily pass to the limit to get $\int_0^T \int_{\R^d} \Delta \varphi \,\rho \diff x \diff t$. By~\Cref{prop:limit-rho},~\Cref{lem:h1_estimate},~\Cref{lem:limit-dist}, and~\Cref{prop:strong-convergence-v} we know there exists a subsequence $\rho^{\varepsilon_k,\sigma_k}(t)$ narrowly converging to $\tilde\rho\in L^1([0,T]\times\Rd)$ and $f^{\varepsilon_k,\sigma_k}$ such that
\begin{align*}
f^{\varepsilon_k,\sigma_k}\to\tilde\rho \quad &\mbox{in }L^1([0,T]\times\Rd),\\
\nabla \sqrt{f^{\varepsilon_k,\sigma_k}}\rightharpoonup g \quad&\mbox{in } L^2([0,T]\times\Rd).
\end{align*}
Since we can prove $\sqrt{f^{\varepsilon_k,\sigma_k}}\to \sqrt{\tilde\rho}$ in $L^2([0,T]\times\Rd)$ by a standard argument, it follows $g=\nabla \sqrt{\tilde\rho}$. In the second term, we use bound \ref{est3} and then simply estimate 
$$
\left|\frac{\sigma\,\mathcal{N}}{\sqrt{(1-\sigma)\,V_{\eps}\ast \rho^{\varepsilon, \sigma} + \sigma\, \mathcal{N}}}\right| \leq \sqrt{\sigma\, \mathcal{N}}
$$
so that the term is bounded by $C\sqrt{\sigma}$. Summing up all the information, upon a further regularisation of the test function to apply~\Cref{lem:commutator_lemma}, we obtain in the joint limit $\varepsilon,\sigma\to0$ a weak solution to~\eqref{eq:heat_eq}, in the sense of Definition~\ref{def:weak_sol_he}. We observe that the weak form is justified by the chain rule in Sobolev spaces since $\nabla \rho=\nabla\left(\sqrt{\rho}\right)^{2}=2\sqrt{\rho}\, \nabla\sqrt{\rho} \in L^1([0,T]\times\Rd)$. Since weak solutions are unique, we can infer the whole sequence $\rhoes$ narrowly converges to $\title{\rho}$.
\end{proof}

\subsection{First $\sigma \to 0$, then $\eps \to 0$.}\label{sect:separate_limits} This limit is slightly harder as one needs to identify the limit in the intermediate step, in particular to prove that we arrive at the formulation~\eqref{eq:weak-form}~or~\eqref{eq:weak-form-just_eps_towork_with} with a characteristic function. For ease of presentation we drop the tilde symbol for the limiting objects.

\begin{thm}\label{thm:sig_to_zero}
Fix $\rho_0\in\mpdtard$ and $\varepsilon>0$ such that $\mh[\rho_0]<\infty$, $V_1$ satisfy~\ref{ass:v1}~-~\ref{ass:v2}, and $\{\rho^{\eps,\sigma}\}_\sigma$ be a sequence of solutions to \eqref{eq:nonlocal_he_gauss} from~\Cref{thm:exist_nlhe_gauss}. Then, as $\sigma \to 0$, there exists a subsequence such that $\rho^{\varepsilon, \sigma}_t \rightharpoonup \rho^{\eps}_t$ for all $t\in [0,T]$, where $\rho^{\eps}$ solves \eqref{eq:nonlocal_he} in the sense of Definition \ref{def:weak-meas-sol}. Moreover, we have the following uniform bounds (with respect to $\eps$):
\begin{equation}\label{eq:uniform_bounds_eps_after_sigma_limit}
\{\rho^{\eps}\ast V_{\eps}\}_\varepsilon,\, \{|\cdot|^2\rho^{\eps}\ast V_{\eps}\}_\varepsilon \mbox{ in } L^{\infty}_t L^1_x, \quad \{\nabla \sqrt{\rho^{\eps}\ast V_{\eps}}\}_\varepsilon \mbox{ in } L^2_{t,x}, \quad
d_W^2(\rho^{\varepsilon}_t, \rho^{\varepsilon}_s) \leq C|t-s|.
\end{equation}
\end{thm}
\begin{proof}
To prove the assertion, arguing as in the proof of Theorem \ref{thm:joint_limit_eps_sigma}, we consider the weak form of~\eqref{eq:nonlocal_he_gauss} and we need to explain how to pass to the limit in the term~\eqref{eq:weak-form-conv}, i.e.
$$
2(1-\sigma)\int_0^T \int_{\R^d} \frac{(\nabla \varphi\, \rho^{\varepsilon, \sigma}) \ast V_{\eps}}{\sqrt{(1-\sigma)\rho^{\varepsilon, \sigma}\ast V_{\eps} + \sigma\, \mathcal{N}}} \,{\nabla \sqrt{(1-\sigma)\rho^{\varepsilon, \sigma}\ast V_{\eps} + \sigma\, \mathcal{N}}} \diff x \diff t.
$$
Thanks to the strong convergence of $(1-\sigma)\rho^{\varepsilon, \sigma}\ast V_{\eps}+\sigma\mn$ we have 
$$
{\nabla \sqrt{(1-\sigma)\rho^{\varepsilon, \sigma}\ast V_{\eps} + \sigma\, \mathcal{N}}} \rightharpoonup \nabla \sqrt{\rho^{\eps}\ast V_{\eps}} \mbox{ weakly in } L^2_{t,x},
$$
so we need to prove
\begin{equation}\label{eq:strong_L2_conv_sufficient_crazy_seq}
\frac{(\nabla \varphi\, \rho^{\varepsilon, \sigma}) \ast V_{\eps}}{\sqrt{(1-\sigma)\rho^{\varepsilon, \sigma}\ast V_{\eps} + \sigma\, \mathcal{N}}} \to \mathds{1}_{V_{\eps}\ast\rho^{\eps}>0} \, \frac{(\nabla \varphi\, \rho^{\eps}) \ast V_{\eps}}{\sqrt{\rho^{\eps}\ast V_{\eps}}} \mbox{ strongly in } L^2_{t,x}.
\end{equation}
To this end, we apply the Vitali convergence theorem which requires uniform integrability, tightness, and a.e. convergence. Note that due to the pointwise estimate and nonnegativity, we have
\begin{equation}\label{eq:pointwise_bound_crazy_sequence_eps_sig}
\frac{(\nabla \varphi\, \rho^{\varepsilon, \sigma}) \ast V_{\eps}}{\sqrt{(1-\sigma)\rho^{\varepsilon, \sigma}\ast V_{\eps} + \sigma\, \mathcal{N}}} \leq C \|\nabla \varphi\|_{L^{\infty}} \, \sqrt{\rho^{\varepsilon, \sigma} \ast V_{\eps}}. 
\end{equation}
We know that $|\cdot|^2\rho^{\varepsilon,\sigma}\ast V_\varepsilon$ is uniformly bounded in $L_t^\infty L_x^1$, hence for any $R>0$, we have
\begin{equation}\label{eq:pointwise_bound_crazy_sequence_eps_sig_unif_int}
\int_0^T \int_{\Rd\setminus B_R} \left|\frac{(\nabla \varphi\, \rho^{\varepsilon, \sigma}) \ast V_{\eps}}{\sqrt{(1-\sigma)\rho^{\varepsilon, \sigma}\ast V_{\eps} + \sigma\, \mathcal{N}}}\right|^2 \diff x \diff t \leq C\, \int_0^T \int_{\Rd\setminus B_R}  \rho^{\varepsilon, \sigma} \ast V_{\eps} \diff x \diff t \le \frac{CT}{R^2}.
\end{equation}
Similarly,~\eqref{eq:pointwise_bound_crazy_sequence_eps_sig} implies tightness of $\left\{\frac{(\nabla \varphi\, \rho^{\varepsilon, \sigma}) \ast V_{\eps}}{\sqrt{(1-\sigma)\rho^{\varepsilon, \sigma}\ast V_{\eps} + \sigma\, \mathcal{N}}}\right\}_\sigma$. Therefore, it remains to prove a.e. convergence of the sequence. Extracting another subsequence, we know that $\rho^{\varepsilon, \sigma}\ast V_{\eps} \to \rho^{\eps}\ast V_{\eps}$ on $[0,T]\times \R^d$. Let $(t,x)$ be a point such that $\rho^{\varepsilon, \sigma}_t\ast V_{\eps}(x) \to \rho^{\eps}_t\ast V_{\eps}(x) = 0$. Then, by estimate \eqref{eq:pointwise_bound_crazy_sequence_eps_sig}, we have
$$
\frac{(\nabla \varphi\, \rho^{\varepsilon, \sigma}) \ast V_{\eps}}{\sqrt{(1-\sigma)\rho^{\varepsilon, \sigma}\ast V_{\eps} + \sigma\, \mathcal{N}}} \to 0.
$$
On the other hand, if $\rho^{\varepsilon, \sigma}_t\ast V_{\eps}(x) \to \rho^{\eps}_t\ast V_{\eps}(x) > 0$, we have $\sqrt{(1-\sigma)\rho^{\varepsilon, \sigma}\ast V_{\eps} + \sigma\, \mathcal{N}} \to \sqrt{\rho^{\eps}\ast V_{\eps}}$ pointwise. Furthermore, by the narrow convergence of $\{\rho^{\varepsilon, \sigma}\}_\sigma$
$$
(\nabla \varphi \, \rho^{\varepsilon, \sigma})\ast V_{\eps} = \int_{\R^d} \nabla V_{\eps}(x-y)\nabla\varphi(t,y)  \diff \rho^{\varepsilon, \sigma}_t(y) \to \int_{\R^d}  V_{\eps}(x-y)\nabla \varphi(t,y)  \diff \rho^{\eps}_t(y) = (\nabla \varphi \, \rho^{\eps})\ast V_{\eps}.
$$
In particular, we proved 
$$
\frac{(\nabla \varphi\, \rho^{\varepsilon, \sigma}) \ast V_{\eps}}{\sqrt{(1-\sigma)\rho^{\varepsilon, \sigma}\ast V_{\eps} + \sigma\, \mathcal{N}}} \to \mathds{1}_{\rho^{\eps}\ast V_{\eps}>0}\, \frac{(\nabla \varphi \, \rho^{\eps})\ast V_{\eps}}{\sqrt{\rho^{\eps}\ast V_{\eps}}} \mbox{ on } (0,T)\times\R^d
$$
and the proof of~\eqref{eq:strong_L2_conv_sufficient_crazy_seq} is concluded.
\end{proof}
\begin{thm}\label{thm:eps_to_zero}
Let $\rho_0\in\mpdtard$ such that $\mh[\rho_0]<\infty$, $V_1$ satisfy~\ref{ass:v1}~--~\ref{ass:v2}, and $\{\rho^{\eps}\}_\varepsilon$ be the sequence of solutions to~\eqref{eq:nonlocal_he} from~\Cref{thm:sig_to_zero}. Then, for any $t\in[0,T]$, $\rho^{\eps}_t \rightharpoonup \rho_t$, where $\rho$ is the unique solution of the heat equation,~\eqref{eq:heat_eq}.
\end{thm}

\begin{proof}
Arguing as above, we consider~\eqref{eq:weak-form}~or~\eqref{eq:weak-form-just_eps_towork_with}~and explain the limit as $\varepsilon\to 0$ in the term 
\begin{equation}\label{eq:term_to_change}
\int_0^T\int_{\R^d} \mathds{1}_{V_{\varepsilon} \ast \rho^{\eps} > 0}\, \frac{V_\varepsilon * (\nabla \varphi\,  \rho^\eps)}{\sqrt{V_\varepsilon * \rho^\eps}} \cdot \nabla \sqrt{V_\varepsilon * \rho^{\eps}} \diff x \diff t.
\end{equation}
By Lemma~\ref{lem:commutator_lemma}, we have
$$\mathds{1}_{V_{\eps}\ast\rho^{\eps}>0} \, \frac{V_\varepsilon * (\nabla \varphi\,  \rho^\eps) - \nabla \varphi \, (V_{\eps}\ast\rho^{\eps})}{\sqrt{V_\varepsilon * \rho^\eps}} \to 0 \mbox{ in } L^2_{t,x}.
$$
By the uniform estimate on $\{\nabla \sqrt{\rho^\varepsilon
 \ast V_\varepsilon}\}_\varepsilon$, this shows that the term in~\eqref{eq:term_to_change} is vanishingly close to
\begin{multline*}
\int_0^T \int_{\R^d} \mathds{1}_{V_{\varepsilon} \ast \rho^{\eps} > 0}\, \nabla \varphi  \frac{V_\varepsilon * \rho^\eps}{\sqrt{V_\varepsilon * \rho^\eps}} \cdot \nabla \sqrt{V_\varepsilon * \rho^{\eps}} \diff x \diff t = \int_0^T\int_\Rd\nabla \varphi \,  \sqrt{V_\varepsilon * \rho^\eps} \cdot \nabla \sqrt{V_\varepsilon * \rho^{\eps}} \diff x \diff t,
\end{multline*}
where we omitted the characteristic function as $\nabla V_{\eps}\ast \rho^{\eps} = 0$ a.e. on the set $\{V_{\eps}\ast\rho^{\eps} = 0\}$. Following the same argument of the proof of~\Cref{thm:joint_limit_eps_sigma} we can show the desired convergence.
\end{proof}

\section{Particle approximation for the heat equation}
\label{sec:particles}
In this section, we provide a deterministic particle approximation for the heat equation using both the nonlocal regularisations~\eqref{eq:nonlocal_he}~and~\eqref{eq:nonlocal_he_gauss}. The argument hinges on the $\lambda_{\varepsilon,\sigma}$-convexity of the regularised function $\mh_\sigma^\varepsilon$ along geodesics connecting compactly supported probability measures. The reason for this restriction to compactly supported measures will be clear during the proof of the result, cf.~\Cref{prop:convexity-energy} below. In order to cope with this further condition we need to prove that solutions of~\eqref{eq:nonlocal_he_gauss} and~\eqref{eq:nonlocal_he} constructed in~\Cref{sec:nonlocal_eq} via the JKO scheme are compactly supported if they are initially compactly supported. A possible strategy to prove such a result is to represent the \textit{JKO solutions} as push-forward of the initial datum via the flow-map solving the characteristic equations. Let us note~\eqref{eq:nonlocal_he_gauss} is a continuity equation whose vector field $\omega^{\varepsilon,\sigma}:[0,T]\times\Rd\to\R$ defined below is locally Lipschitz in space and continuous in time, given $\rhoes\in C([0,T];\mptrd)$,
that is:
\begin{align*}
\omega_t^{\varepsilon,\sigma} = \nabla V_\varepsilon * \log[(1-\sigma)V_\varepsilon* \rho_t^{\varepsilon,\sigma} + \sigma \mn]=V_\varepsilon*\frac{\nabla [(1-\sigma)V_\varepsilon*\rho_t^{\varepsilon,\sigma}+\sigma \mn]}{(1-\sigma)V_\varepsilon*\rho_t^{\varepsilon,\sigma}+\sigma \mn},
\end{align*}
and for $\sigma=0$
\begin{align*}
\omega_t^\varepsilon &=\nabla V_\varepsilon*\log(V_\varepsilon*\rho_t^\varepsilon).
\end{align*}
Although one can prove existence and uniqueness of a flow map, the local-Lipschitzianity in space would not ensure uniqueness of solutions for the continuity equation because the equation is nonlinear, hence the push-forward solution is not necessarily the same as that constructed in~\Cref{sec:nonlocal_eq}. For this reason, we follow the probabilistic representation of solutions to~\eqref{eq:nonlocal_he_gauss}~and~\eqref{eq:nonlocal_he} in the spirit of~\cite[Section 8.2]{AGS}. This approach does not require to investigate well-posedness of the flow map, but only that we have a narrowly continuous solution for the continuity equation, guaranteed by~\Cref{sec:nonlocal_eq}. For the reader's convenience, we briefly recall some preliminary notions. Given a vector field $v:[0,T]\times\Rd\to\Rd$ and the continuity equation $\partial_t\mu_t+\nabla\cdot(v_t \mu_t)=0$, the representation formula for solutions $\mu_t^{\bm{\eta}}$ of the continuity equation is given by
\begin{equation}\label{eq:rep_formula}
\int_\Rd\varphi(x)\diff\mu_t^{\bm{\eta}}(x):=\int_{\Rd\times \Xi_T} \varphi(\gamma(t))\diff\bm{\eta}(x,\gamma), \qquad \forall\ \varphi\in C_b(\Rd), \ t\in[0,T],
\end{equation}
where $\Xi_T:=C([0,T];\Rd)$ is equipped with the supremum norm and $\bm{\eta}\in\mathcal{P}(\Rd\times\Xi_T)$. The measure $\mu_t^{\bm{\eta}}$ can also be written as
\begin{equation}\label{eq:push-forward-eta}
    \mu_t^{\bm{\eta}}=(\mathrm{e}_t)_\#\bm{\eta},
\end{equation}
by introducing the evaluation maps
\[
\mathrm{e}_t:(x,\gamma)\in\Rd\times\Xi_T\mapsto \gamma(t)\in\Rd,
\]
for $t\in[0,T]$. We will use the following result from~\cite{AGS} to represent our solutions, $\rhoes$ and $\rhoe$, as push-forward. Let us specify we only report the part needed for our purposes. Note that we use $\gamma \in AC^p(0,T;\Rd)\iff \gamma:(0,T)\to\Rd$ and $|\gamma'|\in L^p(0,T)$.
\begin{prop}[\cite{AGS}, Theorem 8.2.1]\label{prop:probabilitic_representation}
 Let $\mu_t:[0,T]\to\mprd$ be a narrowly continuous solution of the continuity equation for a suitable vector field $v(t,x)=v_t(x)$ such that $\int_0^T\int_\Rd|v_t(x)|^p\diff\mu_t(x)\diff t<\infty$, for some $p>1$. Then, there exists $\bm\eta\in\mathcal{P}(\Rd\times \Xi_T)$ such that
\begin{enumerate}[label=$(\roman*)$]
    \item $\bm\eta$ is concentrated on the set of pairs $(x,\gamma)$ such that $\gamma\in AC^p((0,T);\Rd)$ is a solution of the ODE $\dot\gamma(t)=v_t(\gamma(t))$ for $\mathcal{L}^1$-a.e. $t\in(0,T)$, with $\gamma(0)=x$;
    \item $\mu_t=\mu_t^{\bm{\eta}}$ for any $t\in[0,T]$, with $\mu_t^{\bm{\eta}}$ as in~\eqref{eq:push-forward-eta}.
\end{enumerate}
\end{prop}
Before applying~\Cref{prop:probabilitic_representation}, in the next results we prove finite time expansion of the support of the solution $\gamma$ from~\Cref{prop:probabilitic_representation}, $(i)$. We remark that our subsequent theoretical results do not encompass the choice $V_1 = \mn_2$ although it is often used in numerical schemes or sampling methods --- this is due to the quadratic growth of the velocity field, cf.~\Cref{lem:log_conv_glob_rescaling}.
\begin{lem}[Compactly supported $V_1$ and $\mn_1$]
\label{lem:vel_cpct_V}
Fix $\varepsilon,\sigma>0$ with $V_1\in C_c^2(\R^d)$ satisfying~\ref{ass:v1} and $\mn_1$ as in \eqref{eq:notation_kernel_n}. Let $\rho^{\varepsilon,\sigma} : t\in[0,T]\mapsto \rho_t^{\varepsilon,\sigma} \in \mptrd$. Then, for $\omega: = \omega^{\varepsilon,\sigma}$, both $\omega$ and $\nabla \omega$ grow at most linearly with the estimate
\[
|\omega(x)|\le C_{\varepsilon,\sigma}^1\langle x\rangle, \quad |\nabla\omega(x)|\le C_{\varepsilon,\sigma}^2 \langle x\rangle,
\]
where the constants scale like
\[
C_{\varepsilon,\sigma}^j \simeq \varepsilon^{-j}(|\log \sigma| + d|\log \varepsilon| + 1), \quad j = 1,2,\quad 0 < \varepsilon,\sigma\ll 1.
\]
\end{lem}
\begin{proof}
According to~\Cref{lem:log_gauss} with $p = 1$ and~\Cref{lem:peetre}, we get
\begin{align*}
    |\omega(x)| &=\left|\int_\Rd \nabla V_\varepsilon(x-y)\log [(1-\sigma)V_\varepsilon*\rhoes(y)+\sigma \mn(y)]\diff y \right|\le C\int_\Rd|\nabla V_\varepsilon(x-y)|\langle y \rangle  \diff y\\
    &\le C \langle x\rangle \int_{\R^d} |\nabla V_\varepsilon(x-y)|\langle x-y\rangle  \diff y.
\end{align*}
Here, the constant $C$ comes from~\Cref{lem:log_gauss} which we keep in mind scales like
\[
C \simeq |\log \sigma| + \log \|V_\varepsilon * \rho^{\varepsilon,\sigma}\|_{L^\infty} \le |\log \sigma| +d| \log \varepsilon| + \log\|V_1\|_{L^\infty}, \quad 0 < \varepsilon,\sigma\ll 1.
\]
After a few changes of integration variables, we get
\[
|\omega(x)| \le \frac{C}{\varepsilon}\langle x\rangle \int |\nabla V_1(z)|\langle \varepsilon z\rangle \diff z = \frac{C}{\varepsilon}\langle x\rangle.
\]
The estimate for $\nabla \omega$ proceeds exactly the same by writing it as
\[
\nabla \omega(x) = (D^2 V_\varepsilon)*\log[(1-\sigma)V_\varepsilon*\rho^{\varepsilon,\sigma} + \sigma \mn_1].
\]
\end{proof}
We can also treat the case $\sigma = 0$ by considering $V_1 = \mn_1$ as below. 
\begin{lem}[Globally supported kernel]\label{lem:bound_velocity}
Let $\varepsilon>0$ and $V_1\equiv\mn_1$ as in~\eqref{eq:notation_kernel_n}. Assume $\rho^{\varepsilon} : t\in[0,T]\mapsto \rho_t^{\varepsilon} \in \mptrd$. Then, there is some constant $C = C(V_1)>0$ such that
\[\left\|\omega^{\varepsilon}\right\|_{L^\infty_{t,x}}\le\frac{C}{\varepsilon}, \qquad  \left\|\nabla \omega^{\varepsilon}\right\|_{L^\infty_{t,x}}\le\frac{C}{\varepsilon^2}.
\]
\end{lem}
\begin{proof}
Arguing as in~\Cref{rem:gaussian} we observe we can differentiate the logarithm in the vector field since $V_\varepsilon*\rhoe>0$ and obtain
\[
\|\omega^{\varepsilon}\|_{L^{\infty}_{t,x}}=\left\|V_\varepsilon*\frac{\nabla V_\varepsilon*\rho_t^{\varepsilon}}{V_\varepsilon*\rho_t^{\varepsilon}}\right\|\le\|V_\varepsilon\|_{L^1_x}\left\|\frac{\nabla V_\varepsilon*\rho^{\varepsilon}}{V_\varepsilon*\rho^{\varepsilon}}\right\|_{L^\infty_{t,x}}.
\]
For $V_1=\mn_1$ we notice $|\nabla V_1(x)|\le C V_1(x)$ and $|\nabla V_\varepsilon(x)|\le \frac{C}{\varepsilon}V_\varepsilon(x)$, for any $x\in\Rd$, so that
\[
\left\|\frac{\nabla V_\varepsilon*\rho^{\varepsilon}}{V_\varepsilon*\rho^{\varepsilon}}\right\|_{L^\infty_{t,x}}\le \frac{C}{\varepsilon}\left\|\frac{V_\varepsilon*\rho^{\varepsilon}}{V_\varepsilon*\rho^{\varepsilon}}\right\|_{L^\infty_{t,x}}\le\frac{C}{\varepsilon}.
\]
The estimate on the gradient $\nabla \omega^{\varepsilon}$ follows analogously.
\end{proof}
\begin{cor}[Finite expansion]
\label{cor:fin_exp_cpct_V}
In the setting of~\Cref{lem:vel_cpct_V}, there exists a unique solution $X^{\varepsilon,\sigma}(\rho):[0,T]\to \R^d$ to the characteristic equation
\begin{equation}\label{eq:ODE_eps_sigma}
\frac{d}{dt}X^{\varepsilon,\sigma}(\rho)(t) = -\omega[\rho](X^{\varepsilon,\sigma}(t)), \quad X^{\varepsilon, \sigma}(0) = X_0 \in \R^d,
\end{equation}
satisfying the estimate
\[
\langle X^{\varepsilon,\sigma}(\rho)(t)\rangle \le \langle X_0\rangle e^{C_{\varepsilon,\sigma}^1 t}, \quad \forall t\in[0,T].
\]
In the case $\sigma=0$, under the hypotheses of~\Cref{lem:bound_velocity} we have the existence and uniqueness of solutions to the ODE in $[0,T]$. Moreover, for $V_1\equiv\mn_1$ it holds $|X^\varepsilon(\rho)(t)|\le|x_0|+C_\varepsilon t$, for any $t\in[0,T]$ with $C_\varepsilon:=C/\varepsilon$.
\end{cor}
\begin{proof}
The existence and uniqueness of local-in-time solutions follow from the Cauchy-Lipschitz theory. From~\Cref{lem:vel_cpct_V}, we obtain the following differential inequality, dropping the dependence on $\rho$ for ease of presentation,
\[
\frac{d}{dt}\langle X^{\varepsilon,\sigma}(t)\rangle^2 = \frac{d}{dt}(1 + |X^{\varepsilon,\sigma}(t)|^2) = -2 X^{\varepsilon,\sigma}\cdot \omega(X^{\varepsilon,\sigma}) \le 2\langle X^{\varepsilon,\sigma}\rangle |\omega(X^{\varepsilon,\sigma})| \le 2C_{\varepsilon,\sigma}^1 \langle X^{\varepsilon,\sigma}\rangle^2.
\]
An application of Gr\"onwall's inequality gives the desired estimate. In particular, linear growth of the velocity field implies we can extend the solution to $[0,T]$. As for the second claim, boundedness of $\omega^\varepsilon$ allow for existence of a unique solution to the ODE in $[0,T]$, by means of Cauchy--Lipschitz Theorem. The result on the support is then an easy consequence of~\Cref{lem:bound_velocity}, an integration of the ODE.
\end{proof}
\begin{prop}[Compactly supported JKO solutions]
    \label{prop:compact_support}
    Let $\rho_0\in\mptrd$ with $\supp\rho_0=B_R$, for $R>0$, and fix $\varepsilon,\sigma>0$. Assume $V_1\in C_c^2(\R^d)$ satisfying~\ref{ass:v1}. The weak solution of~\eqref{eq:nonlocal_he_gauss} from~\Cref{thm:exist_nlhe_gauss} can be represented for any $t\in[0,T]$ as $\rho_t^{\varepsilon,\sigma}=(X_t^{\varepsilon,\sigma}(\rhoes))_\#\rho_0$, for $X_t^{\varepsilon,\sigma}(\rhoes)$ solution to~\eqref{eq:ODE_eps_sigma}. It is compactly supported on the time interval $[0,T]$ and $\supp \rho^{\varepsilon,\sigma}\subseteq B_{\langle R\rangle e^{C^1_{\varepsilon,\sigma}T}}$, where $C^1_{\varepsilon,\sigma}$ is the constant in~\Cref{lem:vel_cpct_V}.

    If $\sigma=0$ and $V_1\equiv\mn_1$, under the conditions in~\Cref{lem:bound_velocity} a weak solution to~\eqref{eq:nonlocal_he} constructed in~\Cref{thm:exist_nlhe_sigma=0} is such that $\rho_t^{\varepsilon}=(X_t^\varepsilon(\rhoe))_\#\rho_0$ and $\supp\rho^\varepsilon\subseteq B_{R+C_\varepsilon T}$, for $X_t^\varepsilon(\rhoe)$ solution to
    \[
    \frac{d}{dt} X_t^\varepsilon(\rhoe)=-\omega[\rho^\varepsilon](X_t^\varepsilon(\rhoe)), \quad X_0^\varepsilon(0)=x\in\Rd.
    \]
\end{prop}
\begin{proof}
    Let us start proving the case $\varepsilon,\sigma>0$. In view of~\Cref{thm:exist_nlhe_gauss}, we know $\rhoes$ is a narrowly continuous curve solution to the continuity eqution~\eqref{eq:nonlocal_he_gauss}, whose vector field satisfies
    \begin{equation}
        \int_0^T\int_\Rd|\omega_t^{\varepsilon,\sigma}(x)|^2\diff\rhotes(x)\diff t\le C^1_{\varepsilon,\sigma}\int_0^T\int_\Rd\langle x\rangle^2\diff\rhotes(x)\diff t<\infty,
    \end{equation}
where we also used that second moments are bounded uniformly in time, cf.~\Cref{lem:en-ineq-mom-bound}. According to~\Cref{prop:probabilitic_representation} with $p=2$ there exists $\bm{\eta}^{\varepsilon,\sigma}\in\mathcal{P}(\Rd\times\Xi_T)$ concentrated on the set of pairs $(x,\gamma^{\varepsilon,\sigma})\in AC^2(0,T;\Rd)$, where $\gamma^{\varepsilon,\sigma}$ is a solution to the ODE $\gamma^{\varepsilon,\sigma}(t)=-\omega^{\varepsilon,\sigma}(\gamma^{\varepsilon,\sigma}(t))$ with $\gamma^{\varepsilon,\sigma}(0)=x$ and $\rho_t^{\varepsilon,\sigma}=\rho_t^{\bm{\eta}^{\varepsilon,\sigma}}=(\mathrm{e}_t)_\#\bm{\eta}^{\varepsilon,\sigma}$. Since we have existence and uniqueness of the flow map $X^{\varepsilon,\sigma}(\rhoes)$, given $\rho^{\varepsilon,\sigma}$ from~\Cref{thm:exist_nlhe_gauss}, we have $\bm{\eta}^\varepsilon=(\mathrm{id}\times X_\cdot)_\#\rho_0$, whence $\rho_t^{\varepsilon,\sigma}=\rho_t^{\bm{\eta}^\varepsilon}=(X_t^{\varepsilon,\sigma}(\rhoes))_\#\rho_0$. Therefore, the support is controlled as in the statement. In the case $\sigma=0$, we can argue similarly by noticing 
\[
\int_0^T\int_\Rd|\omega^{\varepsilon}|^2\diff\rho_t^\varepsilon(x)\diff t\le \frac{C^2}{\varepsilon^2}T\|V_\varepsilon\|_{L^1(\Rd)}^2<\infty,
\]
by means of~\Cref{lem:bound_velocity}. Furthermore, existence and uniqueness of the flow map, solution to
\[
\frac{d}{dt} X_t^\varepsilon=-\omega[\rho^\varepsilon](X_t^\varepsilon), \quad X_0^\varepsilon(0)=x\in\Rd,
\]
for $\rho^\varepsilon$ solution from~\Cref{thm:exist_nlhe_sigma=0}, implies we we have $\bm{\eta}^{\varepsilon}=(\mathrm{id}\times X_\cdot)_\#\rho_0$, hence $\rho_t^\varepsilon=\rho_t^{\bm{\eta}^\varepsilon}=(X_t^\varepsilon)_\#\rho_0$. 
\end{proof}

Let us now look at the convexity with respect to the 2-Wasserstein metric of the regularised functional
\[
\mh_\sigma^\varepsilon[\rho] = \int_{\Rd} ((1-\sigma) V_\varepsilon * \rho + \sigma \mn_1)\log((1-\sigma)V_\varepsilon*\rho + \sigma \mn_1) \diff x,
\]
which coincides with $\mh^\varepsilon$ for $\sigma=0$, cf.~\eqref{eq:log_entropy_reg}.
Fix any $\varrho_0, \varrho_1 \in \mptrd$ with $\gamma \in \Gamma(\varrho_0,\varrho_1)$. We define the geodesic interpolant connecting $\varrho_0$ to $\varrho_1$
\begin{equation}\label{eq:geodesic_interpolaton}
\varrho_\alpha := ((1-\alpha)\pi^1 + \alpha\pi^2)_\# \gamma, \quad \alpha \in[0,1],
\end{equation}
for $\pi^1$, $\pi^2$ projection operators, whereas the one connecting $\varrho_1$ to $\varrho_0$ is
\[
\varrho_{1-\alpha} := (\alpha\pi^1 + (1-\alpha)\pi^2)_\# \gamma, \quad \alpha \in[0,1],
\]
Following the strategy outlined in~\cite{CEW_nl_to_local_24} and the references therein we want to show $\mh_\sigma^\varepsilon$ satisfies the following ``above the tangent line'' inequality
\begin{equation}
\label{eq:above_tan}
\mh_\sigma^\varepsilon[\varrho_1] - \mh_\sigma^\varepsilon[\varrho_0] - \left.\frac{d}{d\alpha}\right|_{\alpha = 0}\mh_\sigma^\varepsilon[\varrho_\alpha] \ge \frac{\lambda_{\eps,\sigma}}{2}d_W^2(\varrho_0,\varrho_1),
\end{equation}
for some $\lambda_{\varepsilon,\sigma}\in\R$ we shall specify later on. This is actually a sufficient condition for $\lambda$-convexity, which will allow us to exploit the theory in~\cite{AGS} and prove~\Cref{thm:particle_approx}.

\begin{prop}[Directional derivative of $\mh_\sigma^\varepsilon$ and $\mh^\varepsilon$]\label{prop:deriv-eps-sigma}
Assume $\varrho_0,\varrho_1\in\mptrd$ with $\supp\varrho_0=B_{R_0}$ and $\supp\varrho_1=B_{R_1}$, for $R_0,R_1>0$. Fix $\varepsilon,\sigma>0$ with $V_1\in C_c^2(\R^d)$ satisfying~\ref{ass:v1}. For any geodesic interpolant $\varrho_\alpha$ it holds:
\[
\left.\frac{d}{d\alpha}\right|_{\alpha = 0}\mh_\sigma^\varepsilon[\varrho_\alpha] = (1-\sigma)\int_\Rd\log((1-\sigma)V_\varepsilon*\varrho_0 + \sigma\mn_1)(x) \left(\iint_{\R^{2d}}(y_0-y_1)\cdot \nabla V_\varepsilon(x-y_0)\diff \gamma(y_0,y_1)\right) \diff x.
\]
In case $\sigma=0$ and $V_1\equiv\mn_1$, the directional derivative is:
\[
\left.\frac{d}{d\alpha}\right|_{\alpha = 0}\mh^\varepsilon[\varrho_\alpha] = \int_\Rd\log(V_\varepsilon*\varrho_0)(x) \left(\iint_{\R^{2d}}(y_0-y_1)\cdot \nabla V_\varepsilon(x-y_0)\diff \gamma(y_0,y_1)\right) \diff x.
\]
For a geodesic interpolant $\varrho_{1-\alpha}$ and $\sigma>0$, it holds:
\[
\left.\frac{d}{d\alpha}\right|_{\alpha = 0}\!\!\!\!\!\!\!\!\!\mh_\sigma^\varepsilon[\varrho_{1-\alpha}] = -(1-\sigma)\int_\Rd\log((1-\sigma)V_\varepsilon*\varrho_1 + \sigma\mn_1)(x) \left(\iint_{\R^{2d}}(y_0-y_1)\cdot \nabla V_\varepsilon(x-y_1)\diff \gamma(y_0,y_1)\right) \diff x.
\]
and, for $\sigma=0$, we have
\[
\left.\frac{d}{d\alpha}\right|_{\alpha = 0}\mh^\varepsilon[\varrho_{1-\alpha}] = - \int_\Rd\log(V_\varepsilon*\varrho_1)(x) \left(\iint_{\R^{2d}}(y_0-y_1)\cdot \nabla V_\varepsilon(x-y_1)\diff \gamma(y_0,y_1)\right) \diff x.
\]
\end{prop}
\begin{proof}
By using the mean value theorem,
\begin{align}\label{eq:diff_quotient_alpha}
    \frac{1}{\alpha}\left(\mh_\sigma^\varepsilon[\varrho_\alpha]-\mh_\sigma^\varepsilon[\varrho_0]\right)=\frac{1-\sigma}{\alpha}\int_\Rd\int_0^1h_\sigma^\varepsilon(s,x)V_\varepsilon*(\varrho_\alpha-\varrho_0)(x)\diff s\diff x,
\end{align}  
where $h_\sigma^\varepsilon(s,x)=\log(s(1-\sigma)V_\varepsilon*\varrho_\alpha(x)+(1-s)(1-\sigma)V_\varepsilon*\varrho_0(x)+\sigma\mn_1(x))$. Note that 
\begin{align*}
V_\varepsilon*(\varrho_\alpha-\varrho_0)(x)&=\iint_\Rdd (V_\varepsilon(x-(1-\alpha)y-\alpha z)-V_\varepsilon(x-y))\diff\gamma(y,z)\\
&=\alpha\iint_\Rdd\nabla V_\varepsilon(x-y)\cdot(y-z)\diff\gamma(y,z)+R_\alpha(x),
\end{align*}
being 
\[
\mathcal{R}_\alpha(x):=\alpha^2\iint_{\R^{2d}}\left\langle\left\{\int_0^1\int_0^\tau D^2V_\varepsilon(x-[(1-\alpha h)y+\alpha h z])\, \diff h \, \diff \tau\right\}(y-z),(y-z)\right\rangle\, \diff \gamma(y,z),
\]
where the contribution of this remainder goes to 0 as $\alpha \to 0$, which we prove later on. Disregarding higher order terms in $\alpha$ for the moment being,~\eqref{eq:diff_quotient_alpha} becomes
\begin{align}\label{eq:diff_quotient_alpha_2}
    \frac{1}{\alpha}\left(\mh_\sigma^\varepsilon[\varrho_\alpha]-\mh_\sigma^\varepsilon[\varrho_0]\right)=(1-\sigma)\int_\Rd\int_0^1 h_\sigma^\varepsilon(s,x)\iint_\Rdd\nabla V_\varepsilon(x-y)\cdot(y-z)\diff\gamma(y,z)\diff s\diff x.
\end{align}
The proof is concluded by applying the Lebesgue Dominated Convergence Theorem, using that $h_\sigma^\varepsilon(s,x)\to\log((1-\sigma)V_\varepsilon*\varrho_0(x)+\sigma\mn(x))$ for a.e. $x$ as $\alpha\to0$, and the majorant
\[
|h_\sigma^\varepsilon(s,x)\,(y-z)\cdot\nabla V_\varepsilon(x-y)|\le C \langle y\rangle|\nabla V_\varepsilon(x-y)|\langle x-y\rangle|y-z|,
\]
where we exploited \Cref{lem:log_gauss} and \Cref{lem:peetre} with $p=1$. The majorant, indeed, is integrable since $\gamma$ has marginals in $\mptrd$ and $V$ is compactly supported ($\sigma>0)$. The higher order terms in $\alpha$ in~\eqref{eq:diff_quotient_alpha} can be estimated as follows. Let us notice there exists $q\gg 1$ such that $|D^2V_1(z)|\le C_q\langle z\rangle^{-q}$. By means of~\Cref{lem:peetre} we know
\[
\langle x\rangle \lesssim \langle x-[(1-\alpha h)y+\alpha h z]\rangle\langle [(1-\alpha h)y+\alpha h z]\rangle.
\]
Using the previous information with~\Cref{lem:log_gauss} and a change of variable $x\mapsto x+[(1-\alpha h)y+\alpha hz]$, we have the following bound:
\begin{align*}
    \int_\Rd \int_0^1 &h_\sigma^\varepsilon(s,x)\mathcal{R}_\alpha(x)\diff s\diff x\\&\le C \alpha^2\int_\Rd\langle x\rangle\iint_{\R^{2d}}\int_0^1 \int_0^\tau \left|(y-z)^T D^2V_\varepsilon(x-[(1-\alpha h)y+\alpha h z])(y-z)\right| \diff h\,\diff \tau\diff \gamma(y,z)\diff x\\
    &\le C \alpha^2\iint_{\Rdd}|y-z|^2(1+|y|+|z|)\int_0^1\int_0^\tau\int_\Rd\langle x-[(1-\alpha h)y+\alpha h z]\rangle^{1-q}\diff x \diff h\,\diff \tau\diff\gamma(y,z)\\
    &=C \alpha^2\iint_{\Rdd}|y-z|^2(1+|y|+|z|)\int_\Rd\langle x\rangle^{1-q}\diff x\diff\gamma(y,z)\\
    &\le C \alpha^2(1+R_0+R_1)\left(m_2(\varrho_0)+m_2(\varrho_1)\right).
\end{align*}
In particular, the term involving $\mathcal{R}_\alpha$ in~\eqref{eq:diff_quotient_alpha} converges to $0$ as $\alpha\to0$.
In case $\sigma=0$ and $V_1\equiv\mn_1$ we follow the same procedure by exploiting~\Cref{lem:log_conv_glob_rescaling}~and~\Cref{lem:peetre} for $p=1$. The computation of the last two directional derivatives is similar, starting from the incremental ratio
\[
\frac{1}{\alpha}\left(\mh_\sigma^\varepsilon[\varrho_{1-\alpha}]-\mh_\sigma^\varepsilon[\varrho_1]\right).
\]
We omit the details so that we are not repetitive in our argument. We observe the lack of compact support for $V$ is not an issue due to the assumption on its decay, cf.~\Cref{rem:general_global_kernel}. 
\end{proof}

\begin{prop}\label{prop:convexity-energy}
    Assume $\varrho_0,\varrho_1\in\mptrd$ with $\supp\varrho_0=B_{R_0}$ and $\supp\varrho_1=B_{R_1}$, for $R_0,R_1>0$. Let $\varepsilon,\sigma>0$ be fixed, $V_1\in C^2_c(\Rd)$ satisfy~\ref{ass:v1} and $\mn_1$ as in~\eqref{eq:notation_kernel_n}. The functional $\mh_\sigma^\varepsilon$ satisfies~\eqref{eq:above_tan} and it is $\lambda_{\varepsilon,\sigma}$-convex along the geodesic connecting $\varrho_0$ with $\varrho_1$, being
    \begin{equation}\label{eq:lambda-eps-sigma}
        \lambda_{\varepsilon,\sigma}=-\varepsilon^{-2}C_{\varepsilon,\sigma}(1-\sigma)(1+R_0+R_1),
    \end{equation}
where the constant scaling like $C_{\varepsilon,\sigma} \simeq |\log \sigma| +d| \log \varepsilon| + \log\|V_1\|_{L^\infty}$, for $0 < \varepsilon,\sigma\ll 1$. If $\sigma=0$ and $V_1\equiv\mn_1$, the functional $\mh^\varepsilon$ satisfies~\eqref{eq:above_tan} and it is $\lambda_\varepsilon$-convex along the geodesic connecting $\varrho_0$ with $\varrho_1$, for
\begin{equation}\label{eq:lambda-eps}
\lambda_{\varepsilon}=-\varepsilon^{-3}\,C_{R}(1+R_0+R_1), 
\end{equation}
for a constant $C_{R}=C_{1,R}$ from~\Cref{lem:log_conv_glob_rescaling} applied to $\varrho_0$.
\end{prop}
\begin{proof}
We want to prove~\eqref{eq:above_tan}. Owing to the convexity of $t\mapsto t\log t$, the first two terms on the left-hand side of~\eqref{eq:above_tan} can be written as
\begin{align*}
    \mh_\sigma^\varepsilon[\varrho_1] - \mh_\sigma^\varepsilon[\varrho_0]
    &= \int_{\R^d} ((1-\sigma) V_\varepsilon * \varrho_1 + \sigma \mn_1)\log((1-\sigma)V_\varepsilon*\varrho_1 + \sigma \mn_1) \diff x\\
    &\quad- \int_{\R^d} ((1-\sigma) V_\varepsilon * \varrho_0 + \sigma \mn_1)\log((1-\sigma)V_\varepsilon*\varrho_0 + \sigma \mn_1) \diff x \\
    &\ge (1-\sigma)\int_{\R^d}\log ((1-\sigma)V_\varepsilon*\varrho_0 + \sigma \mn_1)(x) \left(
V_\varepsilon*\varrho_1(x) - V_\varepsilon*\varrho_0(x)
    \right) \diff x \\
    &= (1\!-\!\sigma)\!\!\int_{\Rd}\!\!\!\log ((1\!-\!\sigma)V_\varepsilon*\varrho_0 \!+\! \sigma \mn_1)(x) \!\!\left(
\iint_{\R^{2d}}\!\!\!\left(V_\varepsilon(x\!-\!y_1)\!-\!V_\varepsilon(x\!-\!y_0)\right) \diff \gamma(y_0,y_1)
    \!\right) \!\diff x.
\end{align*}
These expressions allow us to estimate the left-hand side of~\eqref{eq:above_tan} with
\begin{align*}
    \mh_\sigma^\varepsilon&[\varrho_1] - \mh_\sigma^\varepsilon[\varrho_0] - \left.\frac{d}{d\alpha}\right|_{\alpha = 0}\mh_\sigma^\varepsilon[\varrho_\alpha] \\
    &\ge (1-\sigma)\int_{\R^d}\log ((1-\sigma)V_\varepsilon*\varrho_0 + \sigma \mn_1)(x) \times \dots \\
    &\quad\times \left(
\iint_{\R^{2d}} V_\varepsilon(x-y_1)-V_\varepsilon(x-y_0) - (y_0-y_1)\cdot \nabla V_\varepsilon(x-y_0)\diff \gamma(y_0,y_1)
    \right) \diff x   \\
    &= (1-\sigma)\int_{\R^d}\log ((1-\sigma)V_\varepsilon*\varrho_0 + \sigma \mn_1)(x)\times\dots \\
    &\quad\times\left(
\iint_{\R^{2d}}\!\left\langle\left\{\int_0^1\!\!\int_0^\tau D^2V_\varepsilon(x-[(1- h)y_0+ h y_1]) \diff h  \diff \tau\right\}(y_0-y_1),(y_0-y_1)\right\rangle \diff \gamma(y_0,y_1)
    \!\right) \diff x.
\end{align*}
The last equality uses the second order Taylor expansion with the function $V_\varepsilon$. Taking absolute values, we get
\begin{align*}
     \mh_\sigma^\varepsilon[\varrho_1] &- \mh_\sigma^\varepsilon[\varrho_0] - \left.\frac{d}{d\alpha}\right|_{\alpha = 0}\mh_\sigma^\varepsilon[\varrho_\alpha] \\
    &\ge -(1-\sigma)\int_{\R^d}\left|\log ((1-\sigma)V_\varepsilon*\varrho_0 + \sigma \mn_1)(x)\right|\times\dots \\
    &\quad\times\left(
\iint_{\R^{2d}}|y_0-y_1|^2 \left\{\int_0^1\int_0^\tau\left|D^2V_\varepsilon(x-[(1-h)y_0+hy_1])\right|\diff h\diff \tau\right\}\diff \gamma(y_0,y_1)
    \right) \diff x.
\end{align*}
Using~\Cref{lem:log_gauss} with $p=1$, we get
\begin{align}
\mh_\sigma^\varepsilon&[\varrho_1] - \mh_\sigma^\varepsilon[\varrho_0] - \left.\frac{d}{d\alpha}\right|_{\alpha = 0}\mh_\sigma^\varepsilon[\varrho_\alpha] \notag\\
    &\ge \!- C(1\!-\!\sigma)\iint_{\R^{2d}}\!\!|y_0-y_1|^2 \left( \int_0^1\int_0^\tau\int_{\R^d} \langle x\rangle \left|D^2V_\varepsilon(x-[(1-h)y_0+hy_1])\right| \diff x \diff h \diff \tau \right) \diff\gamma(y_0,y_1)\notag\\
    &=\!-\frac{C(1\!-\!\sigma)}{\varepsilon^2}\iint_{\R^{2d}}\!\!|y_0-y_1|^2 \left( \int_0^1\int_0^\tau\int_{\R^d} \!\!\langle \varepsilon z+[(1-h)y_0+hy_1]\rangle\left|D^2V_1(z)\right| \diff z \diff h \diff \tau \right) \diff\gamma(y_0,y_1)\label{eq:conv_est}
\end{align}
for $C \simeq |\log \sigma| + \log \|V_\varepsilon * \varrho_0\|_{L^\infty} \le |\log \sigma| +d| \log \varepsilon| + \log\|V_1\|_{L^\infty}$, $0 < \varepsilon,\sigma\ll 1$. 
By means of~\Cref{lem:peetre}, we could write
\[
\langle \varepsilon z+[(1-h)y_0+hy_1]\rangle \lesssim \langle \varepsilon z \rangle \langle  (1-h)y_0 + hy_1 \rangle.
\]
According to~\ref{ass:v1}, since $V$ is compactly supported, we can estimate
\[
|D^2V_1(z)| \le C_q \langle z \rangle^{-q}, \quad \forall q\gg 1.
\]
Hence, taking $q\gg 1$ sufficiently large, we can estimate the term in brackets in~\eqref{eq:conv_est} by
\begin{align*}
    &\quad \int_0^1\int_0^\tau\int_\Rd \langle \varepsilon z+[(1-h)y_0+hy_1]\rangle \left|D^2V_1(z)\right| \diff z \diff h\diff\tau  \\
    &\lesssim_q \int_0^1\int_0^\tau\langle  (1-h)y_0 + hy_1 \rangle\int_{\R^d}\left\langle z \right\rangle^{1-q}\diff z \diff h \diff\tau    \\
    &\lesssim_q \int_0^1\int_0^\tau \langle  (1-h)y_0 + hy_1 \rangle \diff h\diff\tau \\
    &\lesssim_q  (1 + |y_0| + |y_1|).
\end{align*}
If we plug this back into~\eqref{eq:conv_est}, we get
\begin{align*}
    \mh_\sigma^\varepsilon[\varrho_1] - \mh_\sigma^\varepsilon[\varrho_0] - \left.\frac{d}{d\alpha}\right|_{\alpha = 0}\mh_\sigma^\varepsilon[\varrho_\alpha]\gtrsim_p-\varepsilon^{-2}C(1-\sigma)(1+R_0+R_1)\iint_{\R^{2d}}|y_0-y_1|^2\diff \gamma(y_0,y_1),
\end{align*}
for $C \simeq |\log \sigma| +d| \log \varepsilon| + \log\|V_1\|_{L^\infty}$, $0<\varepsilon,\sigma\ll 1$. In the last estimate, we used~\Cref{lem:vel_cpct_V}. The result is obtained by taking the infimum over $\gamma\in\Gamma(\varrho_0,\varrho_1)$. As for the case $\sigma=0$ and $V_1=\mn_1$, we follow the same strategy exploiting~\Cref{lem:log_conv_glob_rescaling}. With a similar computation, we have
\begin{align*}
    \mh^\varepsilon&[\varrho_1] - \mh^\varepsilon[\varrho_0] - \left.\frac{d}{d\alpha}\right|_{\alpha = 0}\mh^\varepsilon[\varrho_\alpha] \\
    &\ge \int_{\R^d}\log (V_\varepsilon*\varrho_0)(x)\left(
\iint_{\R^{2d}} V_\varepsilon(x-y_1)-V_\varepsilon(x-y_0) - (y_0-y_1)\cdot \nabla V_\varepsilon(x-y_0)\diff \gamma(y_0,y_1)
    \right) \diff x   \\
    &= \int_{\R^d}\log (V_\varepsilon*\varrho_0)(x)\times\dots \\
    &\quad\times\left(
\iint_{\R^{2d}}\!\left\langle\left\{\int_0^1\!\!\int_0^\tau D^2V_\varepsilon(x-[(1- h)y_0+ h y_1]) \diff h  \diff \tau\right\}(y_0-y_1),(y_0-y_1)\right\rangle \diff \gamma(y_0,y_1)
    \!\right) \diff x.
\end{align*}
Taking absolute values and using~\Cref{lem:log_conv_glob_rescaling} with $C_R=C_{1,R}$, we get
\begin{align*}
     \mh^\varepsilon&[\varrho_1] - \mh^\varepsilon[\varrho_0] - \left.\frac{d}{d\alpha}\right|_{\alpha = 0}\mh^\varepsilon[\varrho_\alpha] \\
    &\ge -\int_{\R^d}\left|\log (V_\varepsilon*\varrho_0)(x)\right|\times\dots \\
    &\quad\times\left(
\iint_{\R^{2d}}|y_0-y_1|^2 \left\{\int_0^1\int_0^\tau\left|D^2V_\varepsilon(x-[(1-h)y_0+hy_1])\right|\diff h\diff \tau\right\}\diff \gamma(y_0,y_1)
    \right) \diff x\\
    &\ge - \frac{C_{R}}{\varepsilon}\iint_{\R^{2d}}\!\!|y_0-y_1|^2 \left( \int_0^1\int_0^\tau\int_{\R^d} \left\langle x\right\rangle \left|D^2V_\varepsilon(x-[(1-h)y_0+hy_1])\right| \diff x \diff h \diff \tau \right) \diff\gamma(y_0,y_1)\notag\\
    &=\!-\frac{C_R}{\varepsilon^3}\iint_{\R^{2d}}\!\!|y_0-y_1|^2 \left( \int_0^1\int_0^\tau\int_{\R^d} \!\!\left\langle \varepsilon z+[(1-h)y_0+hy_1]\right\rangle\left|D^2V_1(z)\right| \diff z \diff h \diff \tau \right) \diff\gamma(y_0,y_1).
\end{align*}
For the term in brackets, we can use~\Cref{lem:peetre} together with the fact that $|D^2 V_1(z)| \le C_q\langle z\rangle^{-q}$ for all large $q\gg 1$ to estimate
\begin{align*}
    \int_0^1\int_0^\tau\int_\Rd &\left\langle \varepsilon z+[(1-h)y_0+hy_1]\right\rangle \left|D^2V_1(z)\right| \diff z \diff h\diff\tau  \\
    &\lesssim_q \int_0^1\int_0^\tau\left\langle  (1-h)y_0 + hy_1 \right\rangle\int_{\R^d}\left\langle z \right\rangle^{1-q}\diff z \diff h \diff\tau    \\
    &\lesssim_q \int_0^1\int_0^\tau \left\langle (1-h)y_0 + hy_1\right \rangle \diff h\diff\tau \\
    &\lesssim_q  \left(1 + |y_0| + |y_1|\right).
    \end{align*}
So,
\[
\mh^\varepsilon[\varrho_1] - \mh^\varepsilon[\varrho_0] - \left.\frac{d}{d\alpha}\right|_{\alpha = 0}\mh^\varepsilon[\varrho_\alpha]\gtrsim_p-\varepsilon^{-3}C_R(1+R_0+R_1)\iint_{\R^{2d}}|y_0-y_1|^2\diff \gamma(y_0,y_1),
\]
and $\lambda_\varepsilon=-\varepsilon^{-3}C_R(1+R_0+R_1)$. The decay of the Hessian matrix follows from the choice of $V_1$. 
\end{proof}
\begin{rem}
    Following the same procedure in the proof of~\Cref{prop:convexity-energy} we obtain 
    \begin{equation}
\label{eq:above_tan_reverse}
\mh_\sigma^\varepsilon[\varrho_0] - \mh_\sigma^\varepsilon[\varrho_1] + \left.\frac{d}{d\alpha}\right|_{\alpha = 0}\mh_\sigma^\varepsilon[\varrho_{1-\alpha}] \ge \frac{\lambda_{\eps,\sigma}}{2}d_W^2(\varrho_0,\varrho_1),
\end{equation}
for $\lambda_{\varepsilon,\sigma}$ as in~\Cref{prop:convexity-energy} for either $\sigma>0$ or $\sigma=0$.
\end{rem}
\Cref{prop:convexity-energy} implies that $\mh_\sigma^\varepsilon$ is $\lambda$-convex on the space of compactly supported probability measures with modulus of convexity depending on the size of the support. For $\varepsilon,\sigma>0$, or $\varepsilon>0$ and $\sigma=0$, \Cref{prop:compact_support} guarantees that weak solutions to~\eqref{eq:nonlocal_he_gauss}, or to~\eqref{eq:nonlocal_he}, remain compactly supported if they were initially compactly supported. \Cref{prop:compact_support} in particular quantifies the growth of the support. The next result provides a stability estimate in $d_W$ for solutions of the nonlocal equations considered, extending~\cite[Theorem 11.1.4]{AGS} to the case of solutions whose support changes in time.

\begin{prop}\label{prop:stability_dW}
Let $\varepsilon,\sigma>0$ be fixed, $V_1\in C^2_c(\Rd)$ satisfy~\ref{ass:v1}, and $\mn_1$ as in~\eqref{eq:notation_kernel_n}. Consider $\rho_i^{\varepsilon,\sigma}:[0,T]\to\mptrd$, $i=1,2$, weak solutions to~\eqref{eq:nonlocal_he_gauss} with initial datum $\rho_{i,0}\in\mptrd$ such that $supp\rho_{i,0}=B_{R_i}$, for $R_i>0$ and $i=1,2$. It holds:
\begin{equation}\label{eq:stability_lambda-T}
d_W(\rho_{1,t}^{\varepsilon,\sigma},\rho_{2,t}^{\varepsilon,\sigma})\le e^{-\lambda_{\varepsilon,\sigma}^Tt}d_W(\rho_{1,0},\rho_{2,0}), \quad \forall\ t\in[0,T],
\end{equation}
being
    \begin{equation}\label{eq:lambda-eps-sigma-T}
        \lambda_{\varepsilon,\sigma}^T=-\varepsilon^{-2}C_{\varepsilon,\sigma}(1-\sigma)(1+\langle R_1\rangle e^{C^1_{\varepsilon,\sigma}T}+\langle R_2\rangle e^{C^1_{\varepsilon,\sigma}T}),
    \end{equation}
for $C_{\varepsilon,\sigma}$ as in~\Cref{prop:convexity-energy} and $C^1_{\varepsilon,\sigma}$ as in~\Cref{lem:vel_cpct_V}. In case $\sigma=0$, $V_1\equiv \mn_1$, and $\rho_i^{\varepsilon}:[0,T]\to\mptrd$, $i=1,2$, are weak solutions to~\eqref{eq:nonlocal_he} with initial data $\rho_i$ as above, then~\eqref{eq:stability_lambda-T} holds for
\begin{equation}\label{eq:lambda-eps-T}
        \lambda_\varepsilon^T=-\varepsilon^{-3}C_R(1+R_1+R_2+2C/\varepsilon T),
    \end{equation}
for a constant $C_{R}$, where $R>0$ is such that $\min_{t\in [0,T]}\inf_{\varepsilon>0}\min(\rho^{\varepsilon}_{1,t}(B_R),\rho^{\varepsilon}_{2,t}(B_R))\ge 1/2$, and $C=C(V_1)$.
\end{prop}
\begin{proof}
Let us start by considering solutions of~\eqref{eq:nonlocal_he}. Let $\gamma_0\in\Gamma_0(\rho_{1,0},\rho_{2,0})$ be the optimal transportation plan for $d_W$. For two solutions $\rho^{\varepsilon}_{1,t}=(X_{1,t}^\varepsilon)_\#\rho_{1,0}$ and $\rhoe_{2,t}=(X_{2,t}^\varepsilon)_\#\rho_{2,0}$, where we use the shorthand notation $X_{i,t}^\varepsilon:=X_{i,t}^\varepsilon(\rho_{i,t}^\varepsilon)$ for $i=1,2$, we define $\gamma_t^\varepsilon=(X_{1,t}^\varepsilon\times X_{2,t}^\varepsilon)_\#\gamma_0$ which is obviously a transportation plan between $\rho_{1,t}^\varepsilon$ and $\rho_{2,t}^\varepsilon$. By definition of $d_W$ we know
\[
d_W^2(\rho_{1,t}^\varepsilon,\rho_{2,t}^\varepsilon)\le \iint_\Rdd|x-y|^2\diff \gamma_t^\varepsilon(x,y)=\iint_\Rdd|X_{1,t}^\varepsilon(x)-X_{2,t}^\varepsilon(y)|^2\diff \gamma_0(x,y).
\]
In order to prove the result we want to show
\begin{align}\label{eq:inequality_derivative_lambda}
    \frac{d}{dt}d_W^2(\rho_{1,t}^\varepsilon,\rho_{2,t}^\varepsilon)\le -2 \lambda_\varepsilon^Td_W^2(\rho_{1,t}^\varepsilon,\rho_{2,t}^\varepsilon),
\end{align}
so that to conclude by applying Gr\"onwall's inequality. The intermediate steps to achieve the inequality above consist in proving (in order):
\begin{subequations}
\begin{align}\label{eq:exchange_derivative_integral}
    \frac{d}{dt}\iint_\Rdd|X_{1,t}^\varepsilon(x)-X_{2,t}^\varepsilon(y)|^2\diff \gamma_0(x,y)=\iint_\Rdd\frac{d}{dt}|X_{1,t}^\varepsilon(x)-X_{2,t}^\varepsilon(y)|^2\diff \gamma_0(x,y);
\end{align}
\begin{align}
    \iint_\Rdd\frac{d}{dt}|X_{1,t}^\varepsilon(x)-X_{2,t}^\varepsilon(y)|^2\diff \gamma_0(x,y)&=2\left(\left.\frac{d}{d\alpha}\right|_{\alpha = 0}\mh^\varepsilon[\varrho_{t,\alpha}^\varepsilon]-\left.\frac{d}{d\alpha}\right|_{\alpha = 0}\mh^\varepsilon[\varrho_{t,1-\alpha}^\varepsilon]\right)\label{eq:inequality_with_directional_derivative}\\
    &\le-2\lambda_\varepsilon^T d_W^2(\rho_{1,t}^\varepsilon,\rho_{2,t}^\varepsilon)\label{eq:inequality_lambda},
\end{align}
\end{subequations}
where $\varrho_{t,\alpha}^\varepsilon$ is the geodesic interpolant connecting $\rho_{1,t}^\varepsilon$ to $\rho_{2,t}^\varepsilon$ at time $t\in[0,T]$, i.e.
\[
\varrho_{t,\alpha}^\varepsilon:= ((1-\alpha)\pi^1 + \alpha\pi^2)_\# \gamma^\varepsilon_t, \quad \alpha \in[0,1].
\]
Inequality~\eqref{eq:inequality_derivative_lambda} can be proven as in~\cite[Theorem 11.1.4]{AGS} using~\cite[Lemma 4.3.4]{AGS}. Exchanging the derivative with the integral in~\eqref{eq:exchange_derivative_integral} can be justified using Dominated Convergence since the flow maps $X_{i,t}^\varepsilon$ are $C^1_t$ functions, for $i=1,2$, and the following computation hold:
\begin{align*}
\frac{d}{dt}|X_{1,t}^{\varepsilon}(x)-X_{2,t}^{\varepsilon}(y)|^2&=-2\left\langle X_{1,t}^\varepsilon(x)-X_{2,t}^\varepsilon(y), \omega_t^\varepsilon[\rhoe_{1,t}](X_{1,t}^{\varepsilon}(x))-\omega_t^\varepsilon[\rhoe_{2,t}](X_{2,t}^{\varepsilon}(y))\right\rangle\\
&=-2\left\langle X_{1,t}^\varepsilon(x)-X_{2,t}^\varepsilon(y),\nabla V_\varepsilon*\log(V_\varepsilon*\rho_{1,t}^\varepsilon(X_{1,t}^{\varepsilon}(x))\right\rangle\\
&\quad+2\left\langle X_{1,t}^\varepsilon(x)-X_{2,t}^\varepsilon(y),\nabla V_\varepsilon*
\log(V_\varepsilon*\rho_{2,t}^\varepsilon(X_{2,t}^{\varepsilon}(y))\right\rangle.
\end{align*}
In particular,
\[
\left|\frac{d}{dt}|X_{1,t}^{\varepsilon}(x)-X_{2,t}^{\varepsilon}(y)|^2\right|\le C_\varepsilon \left(|x|+|y|+C_\varepsilon T\right)\in L^1(\diff\gamma_0).
\]
As for~\eqref{eq:inequality_with_directional_derivative}, note that
\begin{equation}
\begin{split}
    -2\iint_\Rdd&\left\langle X_{1,t}^\varepsilon(x)-X_{2,t}^\varepsilon(y),\nabla V_\varepsilon*\log(V_\varepsilon*\rho_{1,t}^\varepsilon(X_{1,t}^{\varepsilon}(x))\right\rangle \diff\gamma_0(x,y)\\
    &=2\int_\Rd \log(V_\varepsilon*\rho_{1,t}^\varepsilon(z)\iint_\Rd\left(X_{1,t}^\varepsilon(x)-X_{2,t}^\varepsilon(y)\right)\cdot\nabla V_\varepsilon(z-X_{1,t}^{\varepsilon}(x))\diff\gamma_0(x,y)\diff z\\
    &=2\int_\Rd \log(V_\varepsilon*\rho_{1,t}^\varepsilon)(z)\iint_\Rd(x-y)\cdot\nabla V_\varepsilon(z-x)\diff\gamma_t^\varepsilon(x,y)\diff z\\
    &=2\left.\frac{d}{d\alpha}\right|_{\alpha = 0}\mh^\varepsilon[\varrho_{t,\alpha}^\varepsilon],
\end{split}
\end{equation}
Likewise we have
\begin{equation}
\begin{split}
    2\iint_\Rdd&\left\langle X_{1,t}^\varepsilon(x)-X_{2,t}^\varepsilon(y),\nabla V_\varepsilon*\log(V_\varepsilon*\rho_{2,t}^\varepsilon(X_{2,t}^{\varepsilon}(y))\right\rangle \diff\gamma_0(x,y)\\
    &=2\int_\Rd \log(V_\varepsilon*\rho_{2,t}^\varepsilon(z))\iint_\Rd(x-y)\cdot\nabla V_\varepsilon(z-y)\diff\gamma_t^\varepsilon(x,y)\diff z\\
    &=-2\left.\frac{d}{d\alpha}\right|_{\alpha = 0}\mh^\varepsilon[\varrho_{t,1-\alpha}^\varepsilon],
\end{split}
\end{equation}
where $\varrho_{t,1-\alpha}^\varepsilon$ is the geodesic interpolant connecting $\rho_{2,t}^\varepsilon$ to $\rho_{1,t}^\varepsilon$ at time $t\in[0,T]$. At this point we infer~\eqref{eq:inequality_lambda} from~\Cref{prop:convexity-energy}, by using~\eqref{eq:above_tan}~and~\eqref{eq:above_tan_reverse}, where $\lambda_\varepsilon^T$ is as in~\eqref{eq:lambda-eps-T}, taking into account~\Cref{prop:compact_support} for the support of the solutions. As noticed in~\Cref{prop:convexity-energy}, the constant $C_R$ comes from~\Cref{lem:log_conv_glob_rescaling}.
\end{proof}
\begin{rem}
  We observe that~\Cref{prop:stability_dW} can be applied to $\lambda$-convex gradient flows with growing-in-time support, where $\lambda$ does depend on the supports of the solutions.
\end{rem}
\begin{thm}[Particle approximation to~\eqref{eq:heat_eq}]
\label{thm:particle_approx}
Let $N\in\mathbb{N}$, $T>0$, and $\sigma > 0$. Consider $V_1\in C^2_c(\Rd)$ satisfying~\ref{ass:v1} and $\mn_1$ as in~\eqref{eq:notation_kernel_n}. Assume $\rho_0\in \mpdtard$ be such that $\supp\rho_0=B_{R_0}$ for $R_0>0$, $\mh[\rho_0]<\infty$. Let 
\begin{equation}\label{eq:particle_solution}
\rho^{\varepsilon,\sigma, N}_t := \frac{1}{N}\sum_{j=1}^N \delta_{x^j_{\varepsilon,\sigma}(t)}\qquad \mbox{and} \qquad \rho^{N}_0 := \frac{1}{N}\sum_{j=1}^N \delta_{x^j_0},
\end{equation}
where $\{x^i_{\varepsilon,\sigma}\}$ satisfy the following ODE system
$$
\dot{x}^i_{\varepsilon,\sigma}(t) = - \int_{\R^d} \nabla V_\varepsilon(x^i_{\varepsilon,\sigma}(t) - y) \log\left(
\frac{1-\sigma}{N}\sum_{j=1}^NV_\varepsilon(y - x^j_{\varepsilon,\sigma}(t))+\sigma\mn_1(y)
\right) \diff y,
$$
with initial conditions $x^i_{\varepsilon,\sigma}(0) = x^i_0$ such that $d_W(\rho_0^N,\rho_0) \leq \frac{1}{N}$. Let $N = N(\eps)$ and $\sigma = \sigma(\eps)$ be
\begin{equation}\label{eq:asymptotics_N_compactly_supp_V}
N(\eps)\simeq \exp\left(\exp\left(\varepsilon^{-1/\gamma}\right)\right) \quad \mbox{ or equivalently } \quad \varepsilon\simeq(\log (\log N))^{-\gamma}, \qquad \sigma(\eps) = \eps,
\end{equation}
for some $\gamma \in (0,1)$. Then, 
\begin{equation}\label{eq:rate_conv_particles_compact}
\rho_{t}^{\varepsilon,\sigma(\eps),N(\varepsilon)}\to \rho_t \mbox{ narrowly},
\end{equation}
where $\rho$ is the weak solution to \eqref{eq:heat_eq}. The same holds with $\sigma = 0$ with \eqref{eq:asymptotics_N_compactly_supp_V} replaced by 
\begin{equation}\label{eq:asymptotics_N_global_V}
N(\eps)\simeq \exp\left(\varepsilon^{-1/\gamma}\right) \quad \mbox{ or equivalently } \quad \varepsilon\simeq  (\log N)^{-\gamma}
\end{equation}
for $\gamma \in \left(0,\frac{1}{4}\right)$.
\end{thm}
\begin{rem}
Theoretically, to find the discretization satisfying $d_W(\rho_0^N,\rho_0) \leq \frac{1}{N}$, for a given $\rho_0\in\mpdtard$, we consider $\rho_0^N$ obtained by dividing the area below the graph of $\rho_0$ in $N$ regions of equal masses. In practise, this procedure cannot be applied and much worse rates of the form $\frac{1}{N^{\beta}}$ are expected, see \cite{MR3383341}. Clearly, the formulation of Theorem \ref{thm:particle_approx} can be easily adapted to such rates. 
\end{rem}
\begin{proof}[Proof of Theorem \ref{thm:particle_approx}]
Let us first consider the case of compactly supported $V_1$. Applying~\Cref{prop:stability_dW} to $\rho^{\eps, \sigma, N}$ as in \eqref{eq:particle_solution} and $\rho^{\eps, \sigma}$ being the weak solution to \eqref{eq:nonlocal_he_gauss} with initial condition $\rho_0$, we get, for all $t \in [0,T]$,
\begin{equation}\label{eq:wasserstein_estimate_discrete_and_cont_proof_final}
   d_W(\rho^{\eps, \sigma, N}_t, \rho^{\eps, \sigma}_t)\le e^{-\lambda_{\varepsilon,\sigma}^Tt}d_W(\rho_0^{\eps, \sigma, N},\rho_0) \leq \frac{e^{-\lambda_{\varepsilon,\sigma}^T t}}{N},
\end{equation}
where $\lambda_{\varepsilon,\sigma}^T=-\varepsilon^{-2}C_{\varepsilon,\sigma}(1-\sigma)(1+2R_0e^{C^1_{\varepsilon,\sigma}T})$. We choose $\sigma = \sigma(\eps)$ as in \eqref{eq:asymptotics_N_compactly_supp_V} to deduce 
$$
\lambda^T_{\varepsilon,\sigma}\simeq-\varepsilon^{-2}(|\log\sigma|+|d\log\varepsilon|)e^{\varepsilon^{-1}T(|\log\sigma|+d|\log\varepsilon|)}
\simeq - e^{\eps^{-1}\,|\log(\eps)|}
$$ for $\eps$ small. Hence, if $N(\eps)$ is as in \eqref{eq:asymptotics_N_compactly_supp_V} we deduce that the expression in \eqref{eq:wasserstein_estimate_discrete_and_cont_proof_final} converges to 0. Combining with the narrow convergence $\rho^{\eps, \sigma(\eps)}_t \to \rho_t$ as $\eps, \sigma(\eps) \to 0$ from Theorem \ref{thm:joint_limit_eps_sigma} we arrive at \eqref{eq:rate_conv_particles_compact}.

\smallskip
The case $\sigma=0$ requires only an adaptation in the relation between $N$ and $\eps$. Indeed, \eqref{eq:wasserstein_estimate_discrete_and_cont_proof_final} is still valid with $\lambda^T_{\eps} = -\varepsilon^{-3}C_R(1+2R_0+2C/\varepsilon T) \simeq -\eps^{-4}$ and $C_R$ being the constant from Proposition~\ref{prop:stability_dW}. Hence, if $N = N(\eps)$ satisfies \eqref{eq:asymptotics_N_global_V}, the conclusion follows by following the argument above and Theorem~\ref{thm:eps_to_zero}.

\end{proof}

\section{Extension to Fast Diffusion equations}\label{sec:fast_diffusion}
In this section we consider the class of PDEs
\begin{equation}\label{eq:fast_diff}
\partial_t \rho = \Delta \rho^m = -\frac{m}{1-m}\,\mbox{div}(\rho \nabla \rho^{m-1})\tag{FDE}
\end{equation}
with $m \in \left(\frac{d}{d+2},1\right)$ and its nonlocal approximation
\begin{equation}\label{eq:fast_diff_nonlocal}
\partial_t \rho^{\eps} = -\frac{m}{1-m}\, \mbox{div}(\rho^{\eps} \nabla V_{\eps}\ast (V_{\eps}\ast\rhoe)^{m-1})\tag{NLFDE}.
\end{equation}
Following the strategy used for~\eqref{eq:nonlocal_he}, we prove existence of weak solutions to~\eqref{eq:fast_diff_nonlocal} and that they constitute a valuable approximation for solutions of~\eqref{eq:fast_diff}, allowing for a particle scheme. For ease of presentation, we only consider the case of globally supported kernels corresponding to the assumption \ref{ass:v2-g} in the case of heat equation. The form of the kernel we use in this section is specified in~\eqref{eq:kernel_fast_diffusion}. Note that adding a $\sigma$-perturbation in case of compactly supported kernels does not constitute further difficulties.
\subsection{Nonlocal equation for fast diffusion}
The energy functional we consider is $\mathcal{U}^m:\mptrd\to[-\infty,\infty)$, defined by, for any $\mu=\rho\mathcal{L}^d+\mu_s$ with $\mu_s\perp\mathcal{L}^d$,
\begin{equation}
    \mathcal{U}^m[\mu]:=
        -\frac{1}{1-m}\int_\Rd\rho(x)^m\diff x, \qquad m\in\left(\frac{d}{d+2},1\right).
\end{equation}
Note that the function $F(x)=-\frac{1}{1-m}x^m$ is proper, convex, lower semicontinuous, and it satisfies $\lim_{x\to\infty}F(x)/x=0$. The range for the exponent $m$ ensures that, for $\rho\in\mpdtard$,
\begin{equation}\label{eq:functional_finite_fast_diffusion}
\int_\Rd\rho(x)^m\diff x \le \left(\int_\Rd \rho(x)(1+|x|)^2\diff x\right)^m\left(\int_\Rd(1+|x|)^{\frac{-2m}{1-m}}\diff x\right)^{1-m}<+\infty.
\end{equation}
We also observe that with this definition $\mathcal{U}^m$ is lower semicontinuous in $\mptrd$ since its relaxed envelope $(\mathcal{U}^m)^*=\mathcal{U}^m$, cf.~\cite[Section 9.3]{AGS}. One can see this by re-writing $$
\mathcal{U}^m[\mu]=-1/(1-m)\int_\Rd\rho^m\diff x+L\int_\Rd\mu_s,
$$ where $L=\lim_{x\to\infty}F(x)/x=0$.\\

The nonlocal equation is obtained by a regularisation of the functional defined by
\begin{equation}
    \mume[\mu]:=\mum[V_\varepsilon*\mu]=-\frac{1}{1-m}\int_\Rd|V_\varepsilon*\mu|^m\diff x, \qquad m\in\left(\frac{d}{d+2},1\right),
\end{equation}
since $V_\varepsilon*\mu\ll\mathcal{L}^d$ and $(V_\varepsilon*\mu)_s\equiv0$. In particular it is not necessary to keep track of the singular part w.r.t. the Lebesgue measure for $\varepsilon>0$ fixed. Therefore, for consistency of notation we will use $\rho\in\mptrd$ as in the previous sections. For fast diffusion equations we propose the kernel
\begin{equation}\label{eq:kernel_fast_diffusion}
    V(x):=c(1+|x|^2)^{-\alpha} = c\langle x\rangle^{-2\alpha}, \qquad \alpha>\frac{d}{2}+\frac{1}{m},\tag{$\bm{\mathrm{V_F}}$}
\end{equation}
where $c=1/(\int_\Rd(1+|x|^2)^{-\alpha}\diff x)$. Note that $V$ satisfies~\ref{ass:v1} and the following moments bound
\[
\int_\Rd|x|^2 V(x)\diff x<\infty,\quad \int_\Rd|x|^{2/m} V(x)\diff x<\infty.
\]
For later purposes we point out that, since $m\in \left(\frac{d}{d+2},1\right)$, we have
\[
\frac{d}{2} + \frac{1}{m}> \frac{d}{2m} - \frac{1}{2m}.
\]
The JKO scheme for the class of nonlocal fast diffusion equations~\eqref{eq:fast_diff_nonlocal} is
    \begin{equation}
        \label{eq:jko_fast}
		\rho_{\tau}^{\varepsilon,n+1}\in\argmin_{\rho\in\mptrd}\left\{\frac{d_W^2(\rhotne,\rho)}{2\tau}+\mume[\rho]\right\}.
    \end{equation}
For $m\in\left(\frac{d}{d+2},1\right)$, since $-|\cdot|^m$ is convex, Jensen's inequality implies
\[
\mume[\rho_0]\le\|V_\varepsilon\|_{L^1}\mum[\rho_0]<+\infty,
\]
for $\rho_0\in\mpdtard$. The initial datum should be $\rho_0\in\mpdtard$ such that $\mum[\rho_0]<+\infty$ and by abuse of notation $\rho_0$ denotes the density w.r.t. Lebesgue as well.\\

\noindent
Let us observe that, following \eqref{eq:functional_finite_fast_diffusion}, we have
\[
\mume[\rho]\ge-C(d,m)\left(1+m_2(V_\varepsilon*\rho)\right)^m\ge-C(d,m)\left(1+\varepsilon^2m_2(V_1)+m_2(\rho)\right)^m,
\]
for $C(d,m)$ as in~\eqref{eq:functional_finite_fast_diffusion}. This bound is essentially the same obtained in~\cite[Lemma A.1]{CEW_nl_to_local_24}, so we can refer to~\cite[Lemma A.1 and Proposition 3.1]{CEW_nl_to_local_24} for the well-posedness of the scheme and the usual (by now) narrow compactness result for the interpolating sequence as a consequence of the energy and moment bounds. We summarise the result below for the reader's convenience.
\begin{prop}[Narrow compactness, energy, $\&$ moment bounds]
    Let $0<\varepsilon<\varepsilon_0<\infty$ be fixed. There exists an absolutely continuous curve $\rho^{\varepsilon}:[0,T]\to \mptrd$ such that, up to a subsequence, the piecewise constant interpolation $\rho_\tau^\varepsilon$ admits a subsequence $\rho^\varepsilon_{\tau_k}$ so that
\[
\rho_{\tau_k}^{\varepsilon}(t) \overset{k\to\infty}{\rightharpoonup} \rho^{\varepsilon}(t), \quad \text{uniformly in }t\in[0,T].
\]
Furthermore, for any $t\in[0,T]$, the following uniform bounds in $\tau$ and $\varepsilon$ hold
\[
\mume[\rhoe(t)]\le\sup_{\varepsilon>0}\mume[\rho_0]\le\mum[\rho_0], \qquad m_2(\rhoe(t))\le C,
\]
for a constant $C=C(T,m_2(\rho_0),\varepsilon_0^2m_2(V),\mum[\rho_0])$.
\end{prop}

The same strategy used in~\Cref{sec:nonlocal_eq} leads to the existence of a distributional solution to the nonlocal equation~\eqref{eq:fast_diff_nonlocal}, for $\varepsilon>0$ fixed. More precisely, we obtain existence of an absolutely continuous curve $\rho:[0,T]\to \mptrd$ such that, for every $\varphi \in C_c^1(\R^d)$ and any $t\in[0,T]$, it satisfies
\begin{equation}
    \label{eq:weak-form-sigma-fast}
    \begin{split}
    \int_{\R^d}\varphi(x) \diff \rho_t(x) - \int_{\R^d}\varphi(x) \diff \rho_0(x) = \frac{m}{1-m}\int_0^t\int_{\R^d}\nabla \varphi(x)\cdot \nabla V_\varepsilon*(V_\varepsilon*\rho_s)^{m-1}(x) \diff \rho_s(x) \diff s.
    \end{split}
\end{equation}
The RHS of~\eqref{eq:weak-form-sigma-fast} is well defined since tightness of $\rho_t$ and~\Cref{lem:bound_fast_polyn_growth} imply
\[
\left|\int_0^t\int_{\R^d}\nabla \varphi(x)\cdot \nabla V_\varepsilon*(V_\varepsilon*\rho_s)^{m-1}(x) \mathrm{d}\rho_s(x) \mathrm{d} s\right|\le C\int_0^t\int_\Rd|\nabla \varphi(x)|\langle x\rangle^{2\alpha(1-m)}\diff\rho_s(x)\diff s<\infty.
\]
\begin{thm}\label{thm:existence_weak_sol_jko_nonlocal_fast} Let $\varepsilon>0$ and $V$ as in~\eqref{eq:kernel_fast_diffusion}. Suppose $\rho_0\in \mptrd$ satisfies $\mum[\rho_0]<+\infty$. Then, there exists a distributional solution $\rho^\varepsilon\in C([0,T];\mptrd)$ to~\eqref{eq:fast_diff_nonlocal} as in~\eqref{eq:weak-form-sigma-fast} such that $\rho^\varepsilon(0) = \rho_0$.    
\end{thm}
\begin{proof}
The strategy of the proof follows the one of~\Cref{thm:exist_nlhe_gauss}. We will only explain how to deal with the crucial step of the consistency. First of all, we notice $V_\varepsilon*\rho>0$ as proven in~\Cref{lem:bound_fast_polyn_growth}, so that we can divide by this term. Following the proof of~\Cref{thm:exist_nlhe_gauss}, $\sigma=0$, we consider two consecutive elements of the sequence constructed in~\eqref{eq:jko_fast}, i.e. $\rhotne$ and $\rhotnne$ (suppressing the dependence on $\tau, \varepsilon$ for simplicity) and consider the following perturbation, for $\zeta\in C_c^\infty(\Rd;\Rd)$,
\[
\rho_\eta^{n+1} = P_\#^\eta\rho^{n+1}, \quad P^\eta(x) = x + \eta\, \zeta(x).
\]
We need to prove
\begin{equation}
  \frac{\mume[\rho_\eta^{n+1}]-\mume[\rho^{n+1}]}{\eta}\longrightarrow -\frac{m}{1-m}\int_\Rd \zeta(x)\cdot\nabla V_\varepsilon *[V_\varepsilon*\rho^{n+1}]^{m-1}(x)\diff\rho^{n+1}(x), \quad \mbox{as } \eta\to0.  
\end{equation}
By using the Mean-Value Theorem and the symmetry of our kernel we infer
\begin{equation}\label{eq:recovering_rhs_fast_jko}
    \frac{\mume[\rho_\eta^{n+1}]-\mume[\rho^{n+1}]}{\eta}=-\frac{m}{1-m}\iint_{\Rdd}\left(\frac{V_\varepsilon(P^\eta(x) - y) - V_\varepsilon(x-y)}{\eta}\right)M_\eta^\varepsilon(y)\, \mathrm{d}y
    \, \mathrm{d}\rho^{n+1}(x),
\end{equation}
where
\[
M^\varepsilon_\eta(y)=\int_0^1\left(sV_\varepsilon*\rho_\eta^{n+1}(y)+(1-s)V_\varepsilon*\rho^{n+1}(y)\right)^{m-1}\diff s.
\]
Note that, by neglecting a nonnegative term, a majorant in $\eta$ for $M^\varepsilon_\eta(y)$ is 
$$
(V_\varepsilon*\rho^{n+1}(y))^{m-1}(1-s)^{m-1}\in L^1(\diff s),
$$
thus $M^\varepsilon_\eta(y)\to(V_\varepsilon*\rho^{n+1}(y))^{m-1}$ as $\eta\to0$. Our kernel, for $|x|\to\infty$, is such that $\nabla V(x)\simeq|x|^{-(2\alpha+1)}$, whence, using~\Cref{lem:bound_fast_polyn_growth} to bound $|M_\eta^\varepsilon(y)|$~and~\Cref{lem:peetre},
\begin{align*}
    \left|
\frac{V_\varepsilon(P^\eta(x) - y) - V_\varepsilon(x-y)}{\eta} M_\eta^\varepsilon(y)
    \right|&\lesssim\langle y\rangle^{2\alpha(1-m)}|\zeta(x)|\int_0^1|\nabla V_\varepsilon(x+s \, \eta \, \zeta(x) - y)|\,\mathrm{d}s\\
    &\lesssim\langle y\rangle^{2\alpha(1-m)}|\zeta(x)|\int_0^1\langle x+s\,\eta\,\zeta(x) - y \rangle^{-(2\alpha+1)}\,\mathrm{d}s   \\
    &\lesssim\langle y\rangle^{2\alpha(1-m)-(2\alpha+1)}|\zeta(x)|\int_0^1\langle x + s\,\eta\,\zeta(x)\rangle^{2\alpha+1}\,\mathrm{d}s     \\
    &\lesssim\langle y\rangle^{-(2\alpha m+1)}|\zeta(x)| \left(\langle x\rangle^{2\alpha+1} + \langle \zeta(x)\rangle^{2\alpha+1}\right).
    \end{align*}
The latter function is a majorant in $\eta$ and belongs to $L^1(\diff y\diff\rho^{n+1}(x))$ since $\alpha>\frac{d}{2}+\frac{1}{m}>\frac{d-1}{2m}$, for $m>\frac{d}{d+2}$, by direct computation, and owing to the fact that $\zeta(x)$ is smooth and compactly support. Therefore, we can pass to the limit $\eta\to 0$ in~\eqref{eq:recovering_rhs_fast_jko} using the Dominated Convergence Theorem to obtain
\begin{align*}
    &\frac{\mume[\rho_\eta^{n+1}]-\mume[\rho^{n+1}]}{\eta}=-\frac{m}{1-m}\int_{\R^d}
\int_{\R^d}\left(\frac{V_\varepsilon(P^\eta(x) - y) - V_\varepsilon(x-y)}{\eta}\right)M_\eta^\varepsilon(y)\, \mathrm{d}y
   \, \mathrm{d}\rho^{n+1}(x)   \\
    &\overset{\eta\to 0}{\to} -\frac{m}{1-m}\int_{\R^d}\zeta(x)\cdot \int_{\R^d}\nabla V_\varepsilon(x-y)(V_\varepsilon * \rho^{n+1}(y))^{m-1}\, \mathrm{d}y\,\mathrm{d}\rho^{n+1}(x)  \\
    &\quad = -\frac{m}{1-m}\int_{\R^d}\zeta(x)\cdot \nabla V_\varepsilon* (V_\varepsilon * \rho^{n+1})^{m-1}(x)\, \mathrm{d}\rho^{n+1}(x).
\end{align*}
To conclude the proof we need to send the time step $\tau \to 0$, keeping in mind we used $\rho^{n+1}\equiv\rhotnne$ for ease of presentation. This is done similarly as in the proof of ~\Cref{thm:exist_nlhe_sigma=0} by noticing that, for $\zeta(x)=\nabla\varphi(x)$,
\[
|\nabla \varphi(x)\cdot\nabla V_\varepsilon* (V_\varepsilon * \rho^{n+1})^{m-1}(x)|\lesssim|\nabla\varphi(x)|\langle x\rangle ^{2\alpha(1-m)}
\]
and
\[
\int_\Rd|\nabla\varphi(x)|\langle x\rangle ^{2\alpha(1-m)}\diff\rho^\varepsilon_\tau(r,x)\le C(\alpha,m,\varphi)\in L^1([s,t];\diff r),
\]
so that we can pass to the $\tau\to0$ limit in the expression
\begin{align*}
    \int_{\R^d}\varphi(x) \, \mathrm{d}\rho_\tau^\varepsilon(t,x) &- \int_{\R^d}\varphi (x) \, \mathrm{d}\rho_\tau^\varepsilon(s,x) + \mathcal{O}(\tau^2)\\
    &=\frac{m}{1-m}\int_s^t\int_{\R^d}\nabla \varphi(x)\cdot \nabla V_\varepsilon* (V_\varepsilon*\rho_\tau^\varepsilon(r,\cdot))^{m-1}(x) \, \mathrm{d}\rho_\tau^\varepsilon(r,x) \, \mathrm{d}r,
\end{align*}
where we remind the reader
\[
\rho_\tau^{\varepsilon}(t) = \rhotne, \quad t\in((n-1)\tau,n\tau], \quad n=1,\dots,N,
\]
is the piecewise constant interpolation and approximating solution.
\end{proof}
\subsection{Nonlocal-to-local limit}
In order to deal with the $\varepsilon\to0$ limit, we exploit higher regularity of $V_{\varepsilon}*\rhote$. Indeed, by using the heat equation as auxiliary flow we can obtain a uniform $L^2_tH^1_x$-estimate on $V_\varepsilon*\rhote$.
\begin{lem}\label{lem:h1_fast diff_nonlocal}
    Let $\rho_0\in\mpdtard$ such that $\mum[\rho_0]<+\infty$ and $\mh[\rho_0]<+\infty$. There exists a constant $C=C(\rho_0,V,T,m,d)$ such that
    \[
    \sup_{\varepsilon,\tau}\|(V_\varepsilon*\rhote)^\frac{m}{2}\|_{L^2(0,T);H^1(\Rd))}\le C.
    \]
\end{lem}
\begin{proof}
First of all, we notice that for $\frac{d}{d+2}<m<1$ it holds
\begin{align*}
	\left\|(V_\varepsilon*\rhote)^\frac{m}{2} \right\|_{L^2([0,T]; L^2(\R^d))}^2 &= \int_0^T\|V_\varepsilon*\rhote\|_{L^m}^m \diff t\le C T,
\end{align*}
for $C=C(d,m,m_2(V),m_2(\rhot))$. The bound on the gradient can be obtained by exploiting the flow-interchange technique with the heat flow as $2$-Wasserstein gradient flow of the log entropy $\mh[\rho]$. Therefore, we compute
\begin{equation*}\label{eq:h_1_fast_diff_nonloc}
	   \frac{d}{dt}\mume[S_{\mh}^t\rhotnne] =-\frac{4}{m}\int_{\R^d}\left|
		\nabla (V_\varepsilon* S_\mh^t\rhotnne)^\frac{m}{2}
		\right|^2 \diff x,
\end{equation*}
for $S_\mh^t\rhotnne>0$ being the solution of the heat equation at $t$ with initial value $\rhotnne$. For further details we refer the reader to~\cite[Lemma 4.1]{CEW_nl_to_local_24}. The proof is similar up to keeping track of small differences due to $\frac{d}{d+2}<m<1$ and using that $\mh[\rho_0]<\infty$ (instead of deducing from the finiteness of the $L^m$-norm).
\end{proof}
\Cref{lem:h1_fast diff_nonlocal} suggests a stronger weak formulation of~\eqref{eq:fast_diff_nonlocal} with respect to~\eqref{eq:weak-form-sigma-fast}. 
\begin{defn}[Weak measure solution to~\eqref{eq:fast_diff_nonlocal}]
    \label{def:weak-meas-sol_fd}
    For $\varepsilon>0$, we say that an absolutely continuous curve $\rho:[0,T]\to \mptrd$ is a weak measure solution to~\eqref{eq:fast_diff_nonlocal} if the following holds:
    \begin{enumerate}[label = (NLFD-\arabic*)]
        \item (Finite initial energy) \label{a:fie_fd}$\mathcal{U}^m[\rho_0]<+\infty$.
        \item (Regularity) \label{a:l2_fd} $(V_{\eps}\ast\rho)^{m} \in L^{\infty}_t L^1_x$, $\nabla (V_{\eps}\ast\rho)^{m/2} \in L_{t,x}^2$.
        \item (Weak formulation) For every $\varphi \in C_c^1(\R^d)$ and any $t\in[0,T]$, it holds
        \begin{align}
            \label{eq:weak-form_fd}
            \int_{\R^d}\varphi(x) \diff \rho_t(x) \!-\!\! \int_{\R^d}\varphi(x) \diff \rho_0(x) \!=\!
            -2 \int_0^t\!\!\int_{\R^d}  \!\!{V_\varepsilon * (\nabla \varphi \rho_s)}\, (V_{\eps}\ast \rho_s)^{m/2-1} \cdot \nabla( V_\varepsilon * \rho_s)^{m/2} \diff x \diff s.
        \end{align}
    \end{enumerate}
\end{defn}
Let us note that the (RHS) of \eqref{eq:weak-form_fd} makes sense for the following reasons. Arguing as in~\Cref{rem:gaussian} or ~\Cref{lem:bound_fast_polyn_growth}, we know that $V_\varepsilon\ast\rho_s>0$, for any $s\in[0,T]$, since $V_{\eps}$ is a globally supported kernel and $\rho_s$ is tight. The function $(V_{\eps}\ast \rho_s)^{m/2-1}$ is, therefore, well-defined point-wise. Furthermore, by nonnegativity of $V_{\eps}$ and $\rho_s$ we have
$$
|{V_\varepsilon * (\nabla \varphi \rho_s)}\, (V_{\eps}\ast \rho_s)^{m/2-1}| \leq \|\nabla \varphi\|_{\infty}\, |(V_{\eps}\ast \rho_s)^{m/2}| \in L^2_{t,x}.
$$

\begin{lem}[A priori estimates]\label{lem:est_comp_fast_diff}
Let $\{\rho^{\eps}\}_\varepsilon$ be the sequence of solutions to \eqref{eq:fast_diff_nonlocal} constructed in~\Cref{thm:existence_weak_sol_jko_nonlocal_fast}. The following uniform-in-$\varepsilon$ bounds hold:
\begin{enumerate}
    \item\label{item:mass_conservation_fast} $\{V_{\eps}\ast\rhoe\}_\varepsilon$ in $L^{\infty}(0,T; L^1(\Rd))$;
    \item\label{item:enery_dissipation_fast} $\{(V_{\eps}\ast\rho^{\eps})^m\}_\varepsilon$ in $L^{\infty}(0,T; L^1(\Rd))$;
    \item\label{item:h1-estimate-fast} $\{\nabla (V_{\eps}\ast\rho^{\eps})^{m/2}\}_\varepsilon$ in $L^2((0,T)\times\Rd)$.
\end{enumerate}
Moreover, the sequences $\{(V_{\eps}\ast\rhoe)^m\}_\varepsilon$ and $\{(V_{\eps}\ast\rhoe)^{m/2}\}_\varepsilon$ are strongly compact in $L^1((0,T)\times\Rd)$ and $L^2((0,T)\times\Rd)$, respectively.
\end{lem}
\begin{proof}
The first bound follows from the conservation of mass and the integrability of the kernel. In~\Cref{lem:h1_fast diff_nonlocal} we showed~\eqref{item:enery_dissipation_fast}~and~\eqref{item:h1-estimate-fast} for the sequence $V_\varepsilon*\rhote$. By the usual weak lower semicontinuity argument, the same estimates are true after the limit $\tau \to 0$. We proceed to the proof of strong compactness. The additional difficulty in the fast-diffusion case is that $m<1$ so one has to justify that the gradient information on the fraction is sufficient to deduce compactness.\\

\noindent \underline{Step 1: $V_\varepsilon\ast\rho^{\eps}_t\to \rho_t$ in $L^1(\Rd)$ for a.e. $t \in (0,T)$.} 
As the second moment is controlled for a.e. $t \in (0,T)$, we only need to prove convergence in $L^1_{loc}(\Rd)$. Aiming at an application of~\Cref{thm:aulirs-meas-appendix}, we consider a functional
$$
\mathcal{F}(u) = \begin{cases}
 \int_{\Omega}\left(\left|\nabla u^{m/2}\right|^2 + \left|u^{m/2}\right|^2 + |u|\right) \diff x & \mbox{ if } u\in \mathcal{P}_1(\Omega) \mbox{ and } u^{m/2}\in H^1(\Omega);\\
 +\infty & \mbox{ otherwise };
\end{cases}
$$
on $X= L^1(\Omega)$ where $\Omega \subset \Rd$ is a bounded set. As for the pseudo-distance $g$ we can choose the $1$-Wasserstein distance. It is easy to see that $\mathcal{F}$ is lower semicontinuous on $X$. We prove that its level sets are compact in $X$. To this end, we assume that $\mathcal{F}(u_k) \leq C$ for some sequence $\{u_k\} \subset L^1(\Omega)$. Then, there exists a subsequence (not relabelled) such that $u_k^{m/2} \to g$ a.e. It follows that the sequence $u_k = ((u_k)^{m/2})^{2/m}$ also converges a.e. To conclude, by the Vitali theorem (Theorem \ref{thm:vitali}, see also Remark \ref{rem:vitali} if necessary), it is sufficient to prove equiintegrability of $\{u_k\}$ in $L^1(\Omega)$ which can be achieved by the uniform bound on $u_k$ in $L^{1+\delta}(\Omega)$ for some $\delta>0$. This follows easily by the Sobolev embedding: we have that $\{u_k^{m/2}\}$ is bounded in $L^{\frac{2d}{d-2}}(\Omega)$ which means that $\{u_k\}$ is bounded in $L^{\frac{md}{d-2}}(\Omega)$. We have $\frac{md}{d-2}>1$ when $m > \frac{d-2}{d}$ which always holds when $m \geq \frac{d}{d+2}$ because of the inequality $\frac{d}{d+2}>\frac{d-2}{d}$.\\

\noindent \underline{Step 2: $(V_\varepsilon\ast\rho^{\eps}_t)^{m/2} \to (\rho_t)^{m/2}$ in $L^2(\Rd)$ for a.e. $t \in (0,T)$.} Let $\mathcal{N} \subset (0,T)$ be the set of times such that the convergence in Step~1 is true and let $\mathcal{M} \subset (0,T)$ be the set of times such that $\nabla (V_{\eps}\ast\rho^\varepsilon_t)^{m/2}$, $(V_{\eps}\ast\rho^\varepsilon_t)^{m/2} \in L^2(\Rd)$. Clearly, $\mathcal{N} \cap \mathcal{M}$ is a set of full measure in $(0,T)$, so let $t\in \mathcal{N} \cap \mathcal{M}$. From Step~1, we know that $V_{\eps}\ast\rho^\varepsilon_t \to \rho_t$ in measure. By the Sobolev embedding, $\{(V_{\eps}\ast\rho^\varepsilon_t)^{m/2}\}$ is uniformly bounded in $L^{\frac{2d}{d-2}}_{loc}(\Rd)$. As $\frac{2d}{d-2}>2$, the sequence is (locally) uniformly integrable so we obtain convergence in $L^2_{loc}(\Rd)$. To obtain global convergence, we need to prove tightness. We estimate for $p$ to be chosen later
$$
\| (V_{\eps}\ast\rho^\varepsilon_t)^{m/2}\, (1+|x|^{p}) \|_{L^2(\R^d)} \leq 
\| (V_{\eps}\ast\rho^\varepsilon_t)^{m/2}\, (1+|x|^{p})^{m/p} \|_{L^{2/m}(\Rd)} \, \| (1+|x|^{p})^{1-m/p} \|_{L^{2/(1-m)}(\Rd)}.
$$
The first term is controlled due to the second moment while the second is finite provided 
$$
(p-m)\, \frac{2}{1-m} + d-1 < -1 \iff \frac{2}{1-m}\,p < \frac{2m}{1-m} - d.
$$
Now, we see that we can choose $p>0$ provided that $m > \frac{d}{d+2}$. \\

\noindent \underline{Step 3: $(V_{\eps}\ast\rho^\varepsilon)^{m/2} \to \rho^{m/2}$ in $L^2((0,T)\times \Rd)$, $(V_{\eps}\ast\rhoe)^{m} \to \rho^{m}$ in $L^1((0,T)\times \Rd)$.} The first claim follows by the dominated convergence theorem: we have $\|(V_{\eps}\ast\rho^\varepsilon_t)^{m/2} - \rho_t^{m/2}\|_{L^2(\Rd)} \to 0$ for a.e. $t \in (0,T)$ and
$$
\|( V_{\eps}\ast\rho_t^{\eps})^{m/2} - \rho^{m/2}_t\|_{L^2(\Rd)} \leq 2 \sup_{\eps \in (0,1)} \|(V_{\eps}\ast\rho^{\eps})^{m} \|_{L^1(\Rd)}^{1/2} \in L^{\infty}(0,T),
$$
where we used strong convergence of the square root in $L^2$ from Step 2 and lower semicontinuity.
The second claim follows easily from the first one and $(V_{\eps}\ast\rhoe)^{m} = (V_{\eps}\ast\rhoe)^{m/2}\, (V_{\eps}\ast\rhoe)^{m/2}$.
\end{proof}

\begin{lem}[commutator estimate]\label{lem:comm_estimate_fast_diff}
Let $\{\rho^{\eps}\}_\varepsilon$ be a sequence of solutions to \eqref{eq:fast_diff_nonlocal} constructed in~\Cref{thm:existence_weak_sol_jko_nonlocal_fast}. Let $V$ be as in~\eqref{eq:kernel_fast_diffusion}. Then,
$$
[(V_{\eps}\, \min(1, |\cdot|))\ast\rhoe]\, |V_{\eps}\ast\rhoe|^{m/2-1} \to 0 \mbox{ strongly in } L^2((0,T)\times\Rd). 
$$
\end{lem}
\begin{proof}
\noindent \underline{Step 1: Convergence in measure.} We introduce $f_{\eps} = f_{\eps}(t,x)$ to be chosen later, we estimate $\min(1,|x|)\leq |x|$ and we write
\[
[(V_{\eps}|\cdot|)\ast\rho^{\eps}(x)]\, |\rho^{\eps}\ast V_{\eps}|^{m/2-1}(x) = \int_{\Rd} V_\varepsilon(x-y) \frac{|x-y|}{f_{\eps}(x)} \, f_{\eps}(x) \, | V_{\eps}\ast\rhoe|^{m/2-1}(x) \diff \rhoe(y).
\]
We apply Young's inequality with exponents $p$ and $q$ such that $\frac{1}{p}+\frac{1}{q} = 1$ to be determined later 
\begin{equation}\label{eq:fast_diff_first_estimate_pointwise_comm}
[(V_{\eps}|\cdot|)\ast\rhoe]\, |V_{\eps}\ast\rhoe|^{m/2-1} \leq \frac{(V_{\eps} |\cdot|^p)\ast\rhoe}{f_{\eps}^p} + f_{\eps}^q \, |V_{\eps}\ast\rhoe|^{1+q(m/2-1)}.
\end{equation}
To make the terms on the (RHS) equal we require
$$
f_{\eps}^{q+p} = {(V_{\eps} |\cdot|^p)\ast\rhoe} \, |V_{\eps}\ast\rhoe|^{-1-q(m/2-1)},
$$
or equivalently 
$$
f_{\eps} = \left|{(V_{\eps} |\cdot|^p)\ast\rhoe}\right|^{1/(p+q)} \, |V_{\eps}\ast\rhoe|^{(-1-q(m/2-1))/(p+q)}.
$$
Plugging this into the estimate \eqref{eq:fast_diff_first_estimate_pointwise_comm} we deduce
$$
2\, \frac{(V_{\eps} |\cdot|^p)\ast\rhoe}{f_{\eps}^p} = 2\, \left|{(V_{\eps} |\cdot|^p)\ast\rhoe}\right|^{q/(p+q)} \, | V_{\eps}\ast\rhoe|^{(1+q(m/2-1))\, p/(p+q)}.
$$
We choose $q = 2/(2-m)$ and $p = 2/m$ so that $1+q(m/2-1) = 0$  and $q/(p+q) = m/2$. It follows that
$$
[(V_{\eps}|\cdot|)\ast\rhoe]\, |V_{\eps}\ast\rhoe|^{m/2-1} \leq 2\, \left|{(V_{\eps} |\cdot|^{(2/m)})\ast\rhoe}\right|^{m/2}.
$$
It follows that the sequence converges strongly in $L^{2/m}((0,T)\times\Rd)$ to 0 when $\varepsilon\to 0$, since $\int_{\Rd} V(x)\, |x|^{2/m} 
\diff x < \infty$. More precisely, we have
\[
\int_{\Rd}[(V_{\eps} |\cdot|^{(2/m)})\ast\rho^\varepsilon_t](x)\diff x=\varepsilon^{(2/m)}\int_\Rd\int_\Rd(V(z) |z|^{(2/m)})\diff\rho^\varepsilon_t(y)\diff z\,.
\]
Hence the sequence also converges in measure.
\\

\noindent \underline{Step 2: Convergence in $L^2((0,T)\times \Rd)$.} By Step 1, we know that the sequence converges to 0 in measure. Moreover, it is pointwisely estimated by
$$
0 \leq [(V_{\eps}\, \min(1, |\cdot|))\ast\rhoe]\, |V_{\eps}\ast\rhoe|^{m/2-1} \leq  |V_{\eps}\ast\rhoe|^{m/2},
$$
where the (RHS) is strongly compact in $L^2((0,T)\times \Rd)$ by Lemma \ref{lem:est_comp_fast_diff}. We conclude with a version of the Vitali theorem from Corollary \ref{cor:Vitali}.
\end{proof}

\begin{thm}[convergence $\eps \to 0$]\label{thm:eps_to_0_fast}
Let $\{\rho^{\eps}\}_\varepsilon$ be a sequence of solutions to \eqref{eq:fast_diff_nonlocal} constructed in~\Cref{thm:existence_weak_sol_jko_nonlocal_fast}. Then, $\rho^{\eps}_t \to \rho_t$ narrowly, where $\rho$ is the unique solution to~\eqref{eq:fast_diff}.
\end{thm}
\begin{proof}
We only need to explain how to pass to the limit $\eps\to 0$ in the expression
\begin{equation*}
    \begin{split}
\int_0^t&\int_{\R^d} {V_\varepsilon * (\nabla \varphi \, \rho^{\eps}_s)}  \, (V_{\eps}\ast \rho^{\eps}_s)^{m/2-1} \,  \nabla( V_\varepsilon * \rho^{\eps}_s)^{m/2} \diff x \diff s \\
&=   \int_0^t\int_{\R^d} \nabla \varphi\,  (V_{\eps}\ast \rho_s^{\eps})^{m/2}\,  \nabla( V_\varepsilon * \rho^{\eps}_s)^{m/2} \diff x \diff s  
\\ &\quad+
\int_0^t \int_{\Rd} \int_{\Rd} (\nabla \varphi(y) - \nabla \varphi(x) ) \, V_{\eps}(x-y) \, (V_{\eps}\ast \rho_s^{\eps})^{m/2-1}(x) \, \nabla(V_{\eps}\ast\rho^{\eps}_s)^{m/2}(x)  \diff \rho_s^{\eps}(y) \diff x \diff s\\
&=: A_{\eps} + B_{\eps}.
    \end{split}
\end{equation*}
We first claim that 
$
B_{\eps} \to 0.
$
Indeed, we can estimate it as
$$
|B_{\eps}| \leq 2({\|\nabla \varphi\|_{\infty}}+\|\nabla^2 \varphi\|_{\infty})  \int_0^t \int_{\Rd} (V_{\eps}\, \min(1, |\cdot|))\ast\rho^{\eps}_s \,  | V_{\eps}\ast\rho^{\eps}_s|^{m/2-1} \, |\nabla (V_{\eps}\ast\rho^{\eps}_s)^{m/2}|  \diff x \diff s.
$$
Thanks to Lemma \ref{lem:comm_estimate_fast_diff} and the uniform $L^2_{t,x}$ bound on $\nabla (V_{\eps}\ast\rho^{\eps})^{m/2}$ from Lemma \ref{lem:est_comp_fast_diff}, $B_{\eps} \to 0$ as $\eps \to 0$. Now, concerning the term $A_{\eps}$, the integrand is a product of weakly and strongly converging sequences in $L^2((0,T)\times\R^d)$ so that we deduce
$$
\lim_{\eps \to 0} A_{\eps} = \frac{1}{2}\int_0^t \int_{\Rd} \nabla \varphi \,  \rho^{m/2} \, \nabla \rho^{m/2}  \diff x \diff s=-\int_0^t\int_\Rd\Delta \varphi\,\rho^m\diff x\diff s.
$$
The proof is concluded upon observing that solutions to~\eqref{eq:fast_diff} are unique, cf.~\cite{Pierre_87,DK88,CV02}, hence the whole sequence converges to the unique solution. 
\end{proof}

\subsection{Particle approximation}
The continuity equation is~\eqref{eq:fast_diff_nonlocal}, where the velocity field is, for $\frac{d}{d+2}<m<1$ and given $\rho\in C([0,T];\mptrd)$:
\[
\omega^\varepsilon (x):=-\frac{m}{1-m}\nabla V_\varepsilon*(V_\varepsilon*\rhoe)^{m-1}(x), \qquad x\in\Rd.
\]
\begin{lem}\label{lem:growth_w_fast}
Let $\varepsilon>0$, $V$ as in~\eqref{eq:kernel_fast_diffusion}, and $\rhoe\in C([0,T];\mptrd)$. There exists a constant $C_{\alpha,m,R}>0$ depending on $\alpha$, $m$, and $R$ such that the velocity field satisfies the growth condition
\[
|\omega^\varepsilon(x)|\leq \frac{C_{\alpha,m,R}}{\varepsilon}\left( \frac{\langle x\rangle}{\varepsilon}\right)^{2\alpha(1-m)},
\]
for $R>0$ such that $\min_{t\in [0,T]}\inf_{\varepsilon>0}\rho^\varepsilon_t(B_R)\ge \frac{1}{2}$.
\end{lem}
\begin{proof}
Using~\Cref{lem:bound_fast_polyn_growth}~and~\Cref{lem:peetre} we obtain
    \begin{equation}
        \begin{split}
            |\omega^\varepsilon(x)|&=\frac{m}{1-m}\left|\int_{\R^d}\nabla V_\varepsilon(x-y)(V_\varepsilon * \rhoe(y))^{m-1}\mathrm{d}y\right|\\
            &\le C_{\alpha,m,R}\int_\Rd |\nabla V_\varepsilon(x-y)|\left(\frac{\langle y\rangle}{\varepsilon}\right)^{2\alpha(1-m)}\diff y \\
            &\le C_{\alpha,m,R} \left(\frac{\langle x\rangle}{\varepsilon}\right)^{2\alpha(1-m)}\int_\Rd |\nabla V_\varepsilon(x-y)|\langle x-y\rangle^{2\alpha(1-m)}\diff y\\
            &\lesssim \frac{C_{\alpha,m,R}}{\varepsilon}\left(\frac{\langle x\rangle}{\varepsilon}\right)^{2\alpha(1-m)}\int_\Rd\langle z\rangle^{-(2\alpha m+1)}\diff z \le \frac{C_{\alpha,m,R}}{\varepsilon}\left(\frac{\langle x\rangle}{\varepsilon}\right)^{2\alpha(1-m)},
        \end{split}    
    \end{equation}
where we also used $\nabla V(x)\simeq|x|^{-(2\alpha+1)}$.
\end{proof}

Following the discussion in~\Cref{sec:particles}, we observe~\Cref{lem:growth_w_fast} suggests to choose $\alpha=\frac{1}{2(1-m)}$ in order to have linear growth and being able to have a control on the expansion of the support for the corresponding characteristic. We observe that for this exponent $\alpha$, the regularising kernel $V$ is exactly the Barenblatt profile for fast diffusion equations. Note that this choice of $\alpha$ satisfies the inequality
    \begin{equation}\label{eq:restriction_m_d_large}
    \frac{1}{2(1-m)}>\frac{d}{2}+\frac{1}{m} \iff dm^2+(3-d)m -2 >0.
    \end{equation}
    The quadratic inequality~\eqref{eq:restriction_m_d_large} is more restrictive than $m\in \left(\frac{d}{d+2},1\right)$. Therefore, throughout this section we shall assume $\alpha=\frac{1}{2(1-m)}$ for admissible $m<1$ such that~\eqref{eq:restriction_m_d_large} holds. 
    \begin{rem}\label{rem:growth_sigma_fast}
    In the case $V$ is a smooth compactly supported kernel, we can consider $\omega^{\varepsilon,\sigma}(x)=-m/(1-m)\nabla V_\varepsilon*((1-\sigma)V_\varepsilon*\rhoe+\sigma \mn)$, for $\mn$ as in~\eqref{eq:kernel_fast_diffusion}. Following the estimate in~\Cref{lem:bound_fast_polyn_growth} we obtain $|\omega^{\varepsilon,\sigma}(x)|\lesssim \langle x \rangle/\varepsilon$.
\end{rem}
   \noindent In this setting, the following results hold true. We omit the proof as it will follow the same strategy adopted in~\Cref{sec:particles}.
\begin{prop}
\label{prop:expand_fast}
Let $\varepsilon>0$, $V$ as in~\eqref{eq:kernel_fast_diffusion} for $\alpha=1/2(1-m)$ and $m<1$ satisfying~\eqref{eq:restriction_m_d_large}. Fix $\rho\in C([0,T];\mptrd)$. There exists a unique solution $X^{\varepsilon}(\rho):[0,T]\to \R^d$ to the characteristic equation
\begin{equation}\label{eq:ODE_eps_fast}
\frac{d}{dt}X^{\varepsilon}(\rho)(t) = -\omega^\varepsilon[\rho](X^{\varepsilon}(t)), \quad X^{\varepsilon}(0) = X_0 \in \R^d,
\end{equation}
satisfying the estimate
\[
\left\langle X^{\varepsilon}(\rho)(t)\right\rangle \le \left\langle X_0\right\rangle e^{C_{\varepsilon} t}, \quad \forall t\in[0,T],
\]
for $C_\varepsilon=(C_{m,R}/\varepsilon^2)$, where $C_{m,R}$ is the constant from~\Cref{lem:growth_w_fast}.
\end{prop}
\begin{prop}[Compactly supported JKO solutions]
    \label{prop:compact_support_fast_diffusion}
    Let $\rho_0\in\mptrd$ with $\supp\rho_0=B_R$, for $R>0$, and fix $\varepsilon>0$. Let $V$ be as in~\eqref{eq:kernel_fast_diffusion} for $\alpha=1/2(1-m)$ and $m<1$ satisfying~\eqref{eq:restriction_m_d_large}. The weak solution of~\eqref{eq:fast_diff_nonlocal} from~\Cref{thm:existence_weak_sol_jko_nonlocal_fast} can be represented for any $t\in[0,T]$ as $\rho_t^{\varepsilon}=(X_t^{\varepsilon}(\rhoe))_\#\rho_0$, for $X_t^{\varepsilon}(\rhoe)$ solution to~\eqref{eq:ODE_eps_fast}. It is compactly supported on the time interval $[0,T]$ and $\supp \rho^{\varepsilon}\subseteq B_{\langle R\rangle e^{C_{\varepsilon}T}}$, where $C_{\varepsilon}$ is the constant in~\Cref{prop:expand_fast}.
\end{prop}
We now focus on the computation of the directional derivatives, denoting by $\varrho_\beta$, $\beta\in[0,1]$ the geodesic interpolant between $\varrho_0$ and $\varrho_1$ as in \eqref{eq:geodesic_interpolaton}. This is important to establish $\lambda$-converxity of $\mume$ as we will show it holds:
\begin{equation}\label{eq:above_tangent_fast}
    \mume[\varrho_1] - \mume[\varrho_0] - \left.\frac{d}{d\beta}\right|_{\beta = 0}\mume[\varrho_\beta] \ge \frac{\lambda_{\eps}}{2}d_W^2(\varrho_0,\varrho_1),
\end{equation}
for $\lambda_\varepsilon$ computed in~\Cref{prop:convexity-energy_fast_nonlocal} below. 
\begin{prop}[Directional derivative of $\mume$]\label{prop:directional_deriv_fast}
Assume $\varrho_0,\varrho_1\in\mptrd$ with $\supp\varrho_0=B_{R_0}$ and $\supp\varrho_1=B_{R_1}$, for $R_0,R_1>0$. Fix $\varepsilon>0$ and $V$ as in~\eqref{eq:kernel_fast_diffusion} for $\alpha=1/2(1-m)$ and $m<1$ satisfying~\eqref{eq:restriction_m_d_large}. For any geodesic interpolant $\varrho_\beta$ the directional derivative is:
\[
\left.\frac{d}{d\beta}\right|_{\beta = 0}\mume[\varrho_\beta] = -\frac{m}{1-m}\int_\Rd(V_\varepsilon*\varrho_0)^{m-1}(x) \left(\iint_{\R^{2d}}(y_0-y_1)\cdot \nabla V_\varepsilon(x-y_0)\diff \gamma(y_0,y_1)\right) \diff x.
\]
For a geodesic interpolant $\varrho_{1-\beta}$ it holds:
\[
\left.\frac{d}{d\beta}\right|_{\beta = 0}\mume[\varrho_{1-\beta}] = \frac{m}{1-m} \int_\Rd(V_\varepsilon*\varrho_1)^{m-1}(x) \left(\iint_{\R^{2d}}(y_0-y_1)\cdot \nabla V_\varepsilon(x-y_1)\diff \gamma(y_0,y_1)\right) \diff x.
\]
\end{prop}
\begin{proof}
The proof follows closely the one of~\Cref{prop:deriv-eps-sigma} applied to the functional $\mume$. The technical difference compared to the heat equation is essentially keeping track of the behaviour at infinity of $|D^2 V|(x)\simeq|x|^{-2(\alpha+1)}$. After applying the mean value theorem, we need to bound the reminder term. Let us set
\[
\mathcal{R}_\beta(x):=\beta^2\iint_{\R^{2d}}\left\langle\left\{\int_0^1\int_0^\tau D^2V_\varepsilon(x-[(1-\beta h)y+\beta h z]) \diff h  \diff \tau\right\}(y-z),(y-z)\right\rangle\, \diff \gamma(y,z).
\]
By means of~\Cref{lem:bound_fast_polyn_growth} applied to $V_{\eps}\ast \varrho_0$ with $\alpha=1/2(1-m)$ and the change of variable $x\mapsto x+[(1-\beta h)y+\beta hz]$, we have
\begin{align*}
    \left|\int_\Rd \int_0^1\right.&\left| (sV_\varepsilon*\varrho_\beta+(1-s)V_\varepsilon*\varrho_0)^{m-1}\mathcal{R}_\beta(x)\diff s\diff x\right|\\
    &\le \int_0^1 (1-s)^{m-1} \diff s\, \int_\Rd  (V_\varepsilon*\varrho_0)^{m-1}\left|\mathcal{R}_\beta(x)\right|\diff x\\
    &\le \frac{C_{m,R}}{\varepsilon} \int_\Rd\left\langle x\right\rangle|\mathcal{R}_\beta(x)|\diff x\\
    &\le \frac{C_{m,R}}{\varepsilon} \beta^2\!\!\iint_{\Rdd}|y-z|^2\int_0^1\int_0^\tau\int_\Rd\left\langle x\right\rangle \left|D^2V_\varepsilon(x-[(1-\beta h)y+\beta h z]\right|\diff x \diff h\diff \tau\diff\gamma\\
    &=\frac{C_{m,R}}{\varepsilon} \beta^2\!\!\iint_{\Rdd}|y-z|^2\int_0^1\int_0^\tau\int_\Rd\left\langle x+[(1-\beta h)y+\beta h z]\right\rangle \left|D^2V_\varepsilon(x)\right|\diff x \diff h\diff \tau\diff\gamma\\
    &=\frac{C_{m,R}}{\varepsilon^3} \beta^2\!\!\iint_{\Rdd}|y-z|^2\int_0^1\int_0^\tau\int_\Rd\left\langle \varepsilon x+[(1-\beta h)y+\beta h z]\right\rangle \left|D^2V(x)\right|\diff x \diff h\diff \tau\diff\gamma\\
    &\le \frac{C_{m,R}}{\varepsilon^3} \beta^2\!\!\iint_{\Rdd}|y-z|^2\int_0^1\int_0^\tau\int_\Rd\left\langle \varepsilon x+[(1-\beta h)y+\beta h z]\right\rangle \langle x 
    \rangle^{-2(\alpha+1)}\diff x \diff h\diff \tau\diff\gamma,
\end{align*}
with $C_{m,R}$ and $R$ specified in~\Cref{lem:bound_fast_polyn_growth}. 
\Cref{lem:peetre} implies
\[
\langle \varepsilon x+[(1-h)y+hz]\rangle \lesssim \langle \varepsilon x \rangle \langle  (1-h)y + hz \rangle,
\]
so that
\begin{align*}
    \int_0^1\int_0^\tau\int_\Rd &\left\langle \varepsilon x+[(1-h)y+hz]\right\rangle \langle x\rangle^{-2(\alpha+1)} \diff x \diff h\diff\tau\\
    &\lesssim \int_0^1\int_0^\tau\left\langle  (1-h)y+ hz \right\rangle\int_{\R^d}\left\langle x \right\rangle^{1-2(\alpha+1)}\diff x \diff h \diff\tau\\
    &\lesssim(1+|y|+|z|)\int_{\R^d}\left\langle x \right\rangle^{-(2\alpha+1)}\diff x \lesssim(1+|y|+|z|),
\end{align*}
since $\alpha=\frac{1}{2(1-m)}>\frac{d-1}{2}$ as $m>\frac{d}{d+2}>\frac{d-2}{d}$. Summarising, the remainder term can be bounded as follows
\begin{align*}
    \left|\int_\Rd\right.&\left| \int_0^1 (sV_\varepsilon*\varrho_\beta+(1-s)V_\varepsilon*\varrho_0)^{m-1}\mathcal{R}_\beta(x)\diff s\diff x\right|\\
    &\qquad \qquad\lesssim \frac{C_{m,R}}{\varepsilon^3} \beta^2\!\!\iint_{\Rdd}|y-z|^2\int_0^1\int_0^\tau\int_\Rd\left\langle \varepsilon x+[(1-\beta h)y+\beta h z]\right\rangle \langle x 
    \rangle^{-2(\alpha+1)}\diff x \diff h\diff \tau\diff\gamma\\
    &\qquad \qquad\lesssim\frac{C_{m,R}}{\varepsilon^3} \beta^2(1+|R_0|+|R_1|)\iint_{\Rdd}|y-z|^2\diff \gamma(y,z).
\end{align*}
This information is enough to conclude the proof by following the same argument in~\Cref{prop:deriv-eps-sigma}.
\end{proof}
\begin{prop}\label{prop:convexity-energy_fast_nonlocal}
    Assume $\varrho_0,\varrho_1\in\mptrd$ with $\supp\varrho_0=B_{R_0}$ and $\supp\varrho_1=B_{R_1}$, for $R_0,R_1>0$. Let $\varepsilon>0$ and $V$ as in~\eqref{eq:kernel_fast_diffusion} for $\alpha=1/2(1-m)$ and $m<1$ satisfying~\eqref{eq:restriction_m_d_large}. The functional $\mume$ is $\lambda_\varepsilon$-convex along the geodesic connecting $\varrho_0$ with $\varrho_1$, for
\begin{equation}\label{eq:lambda-eps_fast}
\lambda_\varepsilon=-\varepsilon^{-3}C_{m,R}(1+R_0+R_1),
\end{equation}
for a constant $C_{m,R}=C(m,R)$ and $0<\varepsilon\ll 1$ where $C_{m,R}$ is the constant from~\Cref{lem:growth_w_fast} applied to $\varrho_0$.
\end{prop}
\begin{proof}
    Convexity of $-|\cdot|^m$,~\Cref{prop:directional_deriv_fast}, a Taylor expansion of $V_\varepsilon$, and~\Cref{lem:bound_fast_polyn_growth} applied to $\varrho_0$ with $\alpha=1/2(1-m)$ imply that 
    \begin{align*}
     &\mume[\varrho_1] - \mume[\varrho_0] - \left.\frac{d}{d\beta}\right|_{\beta = 0}\mume[\varrho_\beta] \\
    &\ge -\frac{m}{1-m}\int_{\R^d}\left|V_\varepsilon*\varrho_0(x)\right|^{m-1}\times\dots \\
    &\quad\times\left(
\iint_{\R^{2d}}|y_0-y_1|^2 \left\{\int_0^1\int_0^\tau\left|D^2V_\varepsilon(x-[(1-h)y_0+hy_1])\right|\diff h\diff \tau\right\}\diff \gamma(y_0,y_1)
    \right) \diff x\\
    &\gtrsim -\frac{m}{1-m}\frac{C_{m,R}}{\varepsilon}\int_\Rd\iint_{\R^{2d}}|y_0-y_1|^2\int_0^1 \int_0^\tau \left\langle x\right\rangle\left|D^2V_\varepsilon(x-[(1-\beta h)y_0+\beta h y_1])\right| \diff h\,\diff \tau\diff \gamma\diff x.
    \end{align*}
Arguing as in~\Cref{prop:directional_deriv_fast}, we can further bound from below as follows:
\begin{align*}
    \mume[\varrho_1]& - \mume[\varrho_0] - \left.\frac{d}{d\beta}\right|_{\beta = 0}\mume[\varrho_\beta]\\
    &\gtrsim -\frac{m}{(1-m)}\frac{C_{m,R}}{\varepsilon^3}\iint_{\Rdd}|y_0-y_1|^2(1+|y_0|+|y_1|)\int_\Rd\left\langle x\right\rangle^{-(2\alpha+1)}\diff x\diff\gamma(y_0,y_1)\\
    &\gtrsim-\frac{m}{(1-m)}\frac{C_{m,R}}{\varepsilon^3} (1+R_0+R_1)\iint_{\Rdd}|y_0-y_1|^2\diff\gamma(y_0,y_1)
\end{align*}
whence we infer $\lambda_\varepsilon=-C_{m,R}\varepsilon^{-3}(1+R_0+R_1)$, for a constant $C_{m,R}=C(m,R)$.
\end{proof}
A stability estimate in $d_W$ as in~\Cref{prop:stability_dW} can be proven for solutions to~\eqref{eq:fast_diff_nonlocal}. This is crucial for our particle approximation to~\eqref{eq:fast_diff}. We state the results below with no proof as they follow the strategy detailed in~\Cref{sec:particles}.
\begin{prop}\label{prop:stability_dW_fast}
Let $\varepsilon>0$ and $V$ as in~\eqref{eq:kernel_fast_diffusion} for $\alpha=1/2(1-m)$ and $m<1$ satisfying~\eqref{eq:restriction_m_d_large}. Consider $\rho_i^{\varepsilon}:[0,T]\to\mptrd$, $i=1,2$, weak solutions to~\eqref{eq:fast_diff_nonlocal} with initial datum $\rho_{i,0}\in\mptrd$ such that $supp\rho_{i,0}=B_{R_i}$, for $R_i>0$ and $i=1,2$. It holds:
\begin{equation}\label{eq:stability_lambda-T_fast}
    d_W(\rho_{1,t}^{\varepsilon},\rho_{2,t}^{\varepsilon})\le e^{-\lambda_{\varepsilon}^Tt}d_W(\rho_{1,0},\rho_{2,0}), \quad \forall\ t\in[0,T],
\end{equation}
being
    \begin{equation}\label{eq:lambda-eps-T-fast}
        \lambda_{\varepsilon}^T=-\varepsilon^{-3}C_{m,R}(1+\langle R_1\rangle e^{C_{\varepsilon}T}+\langle R_2\rangle e^{C_{\varepsilon}T}),
    \end{equation}
for $C_\varepsilon=(C_{m,R}/\varepsilon^2)$ and $C_{m,R}=C(m,R)$ the maximum of the constants in~\Cref{lem:growth_w_fast} applied to $\rho_{1,t}^{\varepsilon}$ and $\rho_{2,t}^{\varepsilon}$.    
\end{prop}

\begin{thm}[Particle approximation to~\eqref{eq:fast_diff}]
\label{thm:particle_approx_fast}
Let $N\in\mathbb{N}$, $T>0$. Let $V$ as in~\eqref{eq:kernel_fast_diffusion} for $\alpha=1/2(1-m)$ and $m<1$ satisfying~\eqref{eq:restriction_m_d_large}. Let $\rho_0\in \mpdtard$ such that $\supp\rho_0=B_{R_0}$, for $R_0>0$, $\mum[\rho_0]<\infty$, $\mh[\rho_0]<\infty$. Let 
\begin{equation*}
\rho^{\varepsilon, N}_t := \frac{1}{N}\sum_{j=1}^N \delta_{x^j_{\varepsilon}(t)}\qquad \mbox{and} \qquad \rho^{N}_0 := \frac{1}{N}\sum_{j=1}^N \delta_{x^j_0},
\end{equation*}
where $\{x^i_{\varepsilon,\sigma}\}$ satisfy the following ODE system
\[
\dot{x}^i_{\varepsilon}(t) = -\frac{m}{1-m} \int_{\R^d} \nabla V_\varepsilon(x^i_{\varepsilon}(t) - y) \left(
\frac{1}{N}\sum_{j=1}^NV_\varepsilon(y - x^j_{\varepsilon}(t))\right)^{m-1} \diff y,
\]
with initial conditions $x^i_{\varepsilon}(0) = x^i_0$ such that $d_W(\rho_0^N, \rho_0) \leq \frac{1}{N}$. Let $N = N(\eps)$ satisfy
\begin{equation}\label{eq:dependence_N_eps_fast_diff}
N(\eps)\simeq \exp\left(\exp\left(\varepsilon^{-1/\gamma}\right)\right) \quad \mbox{ or equivalently } \quad \varepsilon\simeq(\log (\log N))^{-\gamma}
\end{equation}
for $\gamma \in \left(0, \frac{1}{2}\right)$. Then $\rho^{\varepsilon, N(\eps)}_t$ converges narrowly to $\rho_t$, the weak solution of~\eqref{eq:fast_diff}.
\end{thm}
\begin{proof}
Comparing to the proof of Theorem \ref{thm:particle_approx}, we notice that this time
$$
\lambda_\varepsilon^T \simeq -\varepsilon^{-3}\left(1+e^{-\varepsilon^{-2}T}\right) \simeq - e^{-\varepsilon^{-2}T}.
$$
Hence, choosing $N(\eps)$ as in \eqref{eq:dependence_N_eps_fast_diff} we deduce
$
\frac{1}{N(\eps)} \, e^{-\lambda_\varepsilon^T\, T} \to 0
$
and the proof is concluded.
\end{proof}
\begin{rem}
In the case of a $\sigma$ perturbation, we would need $\gamma<1$ in \eqref{eq:dependence_N_eps_fast_diff} because of the constant in the linear growth of the velocity field, cf.~\Cref{rem:growth_sigma_fast}. We also point out that for the fast diffusion case the rate we obtain in~\eqref{eq:dependence_N_eps_fast_diff} is worse than the rate for the heat equation in~\eqref{eq:asymptotics_N_global_V}, since we do not have boundedness of the velocity field as in~\Cref{lem:bound_velocity} when computing $\nabla (V_\varepsilon\ast\rho)^{m-1}$.
\end{rem}

\begin{appendices}
    \section{Auxiliary results}\label{sec:appendix}
\subsection{The Peetre's inequality}
The lemma below allows to control quantities related to the bracket $\langle \cdot \rangle$, cf.~\cite{BN73} for further details.
   \begin{lem}[Peetre's inequality]
    \label{lem:peetre}
        For any $p\in \R$ and vectors $x,y\in\R^d$, we have
        \[
        \frac{\langle x\rangle^p}{\langle y\rangle^p}\le 2^{\frac{|p|}{2}}\langle x - y\rangle^{|p|}.
        \]
    \end{lem}
    Of course, since $\langle y\rangle = \sqrt{1+|y|^2} = \sqrt{1 + |-y|^2} = \langle -y\rangle$, Peetre's inequality can also be expressed with $\langle x+y\rangle$ on the right-hand side.
    \begin{proof}
        We reproduce the proof from~\cite[Lemma 44]{CDDW20} for completeness. We fix vectors $a,b\in\R^d$ (which will be precised later on in terms of $x,y$) and begin with the case $p=2$. By Young's inequality, we get
        \begin{align*}
            1 + |a-b|^2 \le 1 + 2|a|^2 + 2|b|^2 \le 2 + 2|a|^2 + 2|a|^2|b|^2 + 2|b|^2 = 2(1+|a|^2)(1+|b|^2).
        \end{align*}
        Dividing both sides by $\langle b\rangle^2 = 1+|b|^2$ gives
        \[
        \frac{\langle a - b\rangle^2}{\langle b\rangle^2}\le 2\langle a\rangle^2.
        \]
        Setting $a = x-y$ and $b = -y$ gives
        \[
        \frac{\langle x\rangle^2}{\langle y\rangle^2}\le 2\langle x-y\rangle^2.
        \]
        Raising both sides of this inequality by non-negative powers proves Peetre's inequality for general $p\ge 0$. On the other hand, with the choice of $a = x-y$ and $b = x$, we get
        \[
        \frac{\langle y\rangle^2}{\langle x\rangle^2}\le 2\langle x-y\rangle^2.
        \]
        Raising both sides of this inequality by non-negative powers proves Peetre's inequality for general $p\le 0$.
    \end{proof}

\subsection{Growth control on nonlocal quantities} Here we state estimates that allow to control the growth of several quantities arising in the velocity fields of our nonlocal PDEs~\eqref{eq:nonlocal_he_gauss},~\eqref{eq:nonlocal_he}, and~\eqref{eq:fast_diff_nonlocal}.

    \begin{lem}
        \label{lem:log_gauss}
        Fix $\sigma>0, \rho \in L^\infty(\R^d)$ such that $\rho(x) \ge 0$, and $\mn(x) = \exp(-\langle x\rangle^p)$ for some $p\ge 1$. Then, there exists a constant $C=C(\sigma, \|\rho\|_{L^\infty})>0$ (explicitly computable as below) such that
        \[
        |\log ((1-\sigma)\rho(x) + \sigma\mn(x))|\le C\langle x\rangle^p, \quad \text{for a.e. }x\in\R^d.
        \]
        In particular, when $0<\sigma \ll 1$, the constant $C$ behaves like
        \[
        C\simeq |\log \sigma| + |\log \|\rho\|_{L^\infty}|.
        \]
    \end{lem}
    \begin{proof}
        Since $\log$ is increasing, we immediately get the following upper bound for almost every $x\in\R^d$
        \[
        \log ((1-\sigma)\rho(x) + \sigma\mn(x)) \le \log ((1-\sigma)\|\rho\|_{L^\infty} + \sigma)
        =:C_1.
        \]
        As for the lower bound, again by the monotonicity of $\log$, we get
        \[
        \log((1-\sigma)\rho(x) + \sigma\mn(x)) \ge \log (\sigma\mn(x)) =\log \sigma  - \langle x \rangle^p.
        \]
        More precisely, since $\langle x\rangle \ge 1$, we can further lower bound this by
        \[
        \log((1-\sigma)\rho + \sigma \mn) \ge - (|\log\sigma| + \langle x\rangle^p) \ge - (1 + |\log \sigma|)\langle x\rangle^p.
        \]
        Therefore, taking
        \[
        C := \max \{|C_1|, \, 1 + |\log \sigma|\}
        \]
        gives
        \[
        |\log((1-\sigma)\rho + \sigma \mn_1)| \le C\langle x\rangle^p.
        \]
        To finish the proof, notice that for $\|\rho\|_{L^\infty}\leq 1$ the convexity of $|\log(x)|$ implies $|C_1|\leq |\log (\|\rho\|_{L^\infty})|$. If $\|\rho\|_{L^\infty}\geq 1$ we can make use of the subaditivity of $\log(x+1)$ for $x\geq 0$ to obtain
$$
C_1\leq \log (\|\rho\|_{L^\infty}+\sigma)\leq  \log (\|\rho\|_{L^\infty})+ \log (\sigma+1).
$$        
        These estimates imply the behaviour of $C$ for $\sigma $ near zero as claimed.
    \end{proof}
\begin{lem}
    \label{lem:log_conv_glob_rescaling}
    Let $\mn(x) = \exp(-\langle x\rangle^p)$, for $p\in[1,\infty)$, $0<\varepsilon\leq 1$, and $\rho \in \mptrd$. Let $R>0$ be such that $\rho(B_R)\ge 1/2$ and consider $\mn_\varepsilon(x) = \varepsilon^{-d}\mn\left(\frac{x}{\varepsilon}\right)$. There exists a constant $C_{p,R}>0$ depending on $p$ and $R$ such that
\begin{equation}
    \label{eq:log_conv_glob_rescaling}
    |\log (\mn_\varepsilon * \rho(x))| \leq C_{p,R} \left(\frac{\langle x\rangle}{\varepsilon}\right)^p.
\end{equation}
\end{lem}
\begin{proof}
For the upper bound, since $\|\mn_\varepsilon\|_{L^\infty} \le \varepsilon^{-d}$, we have
    \[
    \log(\mn_\varepsilon*\rho(x)) = \log \left(\int_{\R^d} \mn_\varepsilon(x-y) \, \mathrm{d}\rho(y)\right) \le \log \int_{\R^d} \varepsilon^{-d}\,\mathrm{d}\rho(y) = -d\log \varepsilon.
    \]
For the lower bound, we use the triangle inequality to estimate
$$
\mn_\varepsilon*\rho(x) = \int_{\R^d} \mn_{\varepsilon}(x-y) \diff \rho(y) \geq  \int_{B_R} \mn_{\varepsilon}(x-y) \diff \rho(y) \geq \frac{1}{2}\, \exp\left(-C_p\,\left\langle \frac{x}{\eps} \right\rangle^p - C_p\,\frac{R^p}{\eps^p}\right)\,.
$$
Applying the logarithm on both sides, we obtain
$$
\log(\mn_\varepsilon*\rho(x)) \geq - \log2 - C_p\,\left\langle \frac{x}{\eps} \right\rangle^p  - C_p\,\frac{R^p}{\eps^p}. 
$$
Collecting both the upper and lower bounds leads to the desired claim.
\end{proof}
\begin{lem}\label{lem:bound_fast_polyn_growth}
    Let $V$ be as in~\eqref{eq:kernel_fast_diffusion}, $0<\varepsilon\le 1$, and $\rho\in\mptrd$. Let $R>0$ be such that $\rho(B_R)\ge 1/2$ and consider $V_\varepsilon(x)=\varepsilon^{-d}V\left(\frac{x}{\varepsilon}\right)$. There exists a constant $C_{\alpha,m,R}>0$ depending on $\alpha$, $m$, and $R$ such that, for any $x\in\Rd$, $V_\varepsilon*\rho(x)>0$ and
    \[
    |V_\varepsilon*\rho(x)|^{m-1}\le C_{\alpha,m,R}\left(\frac{\langle x \rangle}{\varepsilon}\right)^{2\alpha(1-m)}.
    \]
\end{lem}
\begin{proof}
    Using the triangle inequality, $\alpha>1$, and fixing $|y|\le R$, we have
    \begin{align*}
        \left\langle \frac{x-y}{\varepsilon} \right\rangle^{2\alpha} \le \left(2\left\langle \frac{x}{\varepsilon}\right\rangle^2+2\left|\frac{y}{\varepsilon}\right|^2\right)^\alpha\le C_\alpha\left(\left\langle \frac{x}{\varepsilon}\right\rangle^{2\alpha}+\left|\frac{R}{\varepsilon}\right|^{2\alpha}\right),
    \end{align*}
    whence
    \[
    V_\varepsilon*\rho(x)\ge \varepsilon^{-d}C_\alpha \left(\frac{\langle x\rangle^{2\alpha}}{\varepsilon^{2\alpha}}+\left|\frac{R}{\varepsilon}\right|^{2\alpha}\right)^{-1}\ge \varepsilon^{-d}C_{\alpha,R}\frac{\left\langle x\right\rangle^{-2\alpha}}{\varepsilon^{-2\alpha}}>0.
    \]
Since $0<m<1$, we obtain the result
\[
|V_\varepsilon*\rho|^{m-1}\le C_{\alpha,m,R}\frac{\left\langle x\right\rangle^{2\alpha(1-m)}}{\varepsilon^{2\alpha(1-m)}}.
\]
\end{proof}
\subsection{Compactness results}
We state now an extension of the Dominated Convergence Theorem~\cite[Chapter 4, Theorem 17]{R88}. This is quoted with less generality for our purposes.

\begin{thm}[Extended Dominated Convergence Theorem]
\label{thm:EDCT}
Let $\{f^\sigma\}_{\sigma>0}$ and $\{g^\sigma\}_{\sigma>0}$ be sequences of measurable functions on a (Lebesgue) measurable set $X\subset\R^n$ such that $g^\sigma\ge 0$ and suppose that there exist measurable functions $f, \, g$ satisfying the following assumptions.
\begin{enumerate}
	\item $|f^\sigma(y)|\le g^\sigma(y)$ for all $\sigma>0$ and pointwise almost every $y\in X$.
	\item $f^\sigma(y) \to f(y)$ and $g^\sigma(y) \to g(y)$ pointwise almost every $y\in X$ as $\sigma\to 0$.
	\item $\int_{X}g^\sigma(y) \mathrm{d}y \to \int_{X} g(y) \mathrm{d}y$ as $\sigma\to 0$.
\end{enumerate}
Then, we have $\int_{X}f^\sigma(y) \mathrm{d}y \to \int_{X} f(y)\mathrm{d}y$ as $\sigma\to 0$.
\end{thm}
\noindent We also recall the Vitali convergence theorem.
\begin{thm}\label{thm:vitali}
Let $(X,\mathcal{F},\mu)$ be a measure space (with possibly $\mu(X) = \infty$). Let $p \in [1,\infty)$, $\{f_n\}_{n \in \mathbb{N}} \subset L^p(X,\mathcal{F},\mu)$ and $f$ be an $\mathcal{F}$-measurable function. Then, $f_n \to f$ in $L^p(X,\mathcal{F},\mu)$ if and only if:
\begin{enumerate}[label=(W\arabic*)]
\item $f_n \to f$ in measure,
    \item\label{ass:vitali:uni_int} $\{f_n\}_{n \in \mathbb{N}}$ is uniformly integrable, i.e.
$$
\forall{\varepsilon>0} \, \exists{\delta>0}: \, \forall{A \in \mathcal{F}} \, \qquad \mu(A) < \delta \implies \sup_{n \in \mathbb{N}} \int_{A} |f_n|^p \diff \mu < \varepsilon,
$$
    \item\label{ass:vitali:tight} $\{f_n\}_{n \in \mathbb{N}}$ is tight in $L^p(X,\mathcal{F},\mu)$, i.e. 
$$
\forall{\varepsilon>0} \, \exists{\text{compact set }K_{\eps} \subset X } \, \qquad   \sup_{n \in \mathbb{N}} \int_{X\setminus K_{\eps}} |f_n|^p \diff \mu < \varepsilon.
$$   
\end{enumerate} 
\end{thm}
\begin{rem}\label{rem:vitali}
\begin{enumerate}[label = \textbf{(\Alph*)}]
    \item The condition \ref{ass:vitali:uni_int} is usually verified by proving higher moment estimates, for example, by obtaining uniform bounds on $\{f_n\}$ in $L^q(X,\mathcal{F},\mu)$ with $q>p$.
    \item The condition \ref{ass:vitali:tight} is usually verified by proving tail estimates, for example, by obtaining uniform bounds on $\{f_n |x|^r\}$ in $L^p(X,\mathcal{F},\mu)$ for some $r>0$.
    \item If $\mu(X) <\infty$, the condition \ref{ass:vitali:tight} is immediately satisfied.
\end{enumerate}
\end{rem}
\begin{cor}\label{cor:Vitali}
Under the setting of Theorem \ref{thm:vitali}, suppose that $f_n \to f$ in measure, $|f_n| \leq g_n$ where $\{g_n\}$ is a sequence convergent in $L^p(X,\mathcal{F},\mu)$. Then $f_n \to f$ in $L^p(X,\mathcal{F},\mu)$.
\end{cor}
\begin{proof}
By Theorem \ref{thm:vitali}, the sequence $\{g_n\}$ satisfies conditions \ref{ass:vitali:uni_int} and \ref{ass:vitali:tight} in Theorem \ref{thm:vitali}. By the pointwise inequality $|f_n| \leq g_n$, $f_n$ also satisfies these conditions so by Theorem \ref{thm:vitali} $f_n \to f$ in $L^p(X,\mathcal{F},\mu)$.
\end{proof}

\noindent Finally, we also recall a refined version of Aubin-Lions due to Rossi and Savaré suitable for equi-integrability in probability spaces.
\begin{thm}\cite[Theorem 2]{RS}\label{thm:aulirs-meas-appendix}
Let $X$ be a separable Banach space. Consider
\begin{itemize}
\item a lower semicontinuous functional $\mathscr{F}:X\to[0,+\infty]$ with relatively compact sublevels in $X$;
\item a pseudo-distance $g:X\times X\to[0,+\infty]$, i.e., $g$ is lower semicontinuous and such that $g(\rho,\eta)=0$ for any $\rho,\eta\in X$ with $\mathscr{F}(\rho)<\infty$, $\mathscr{F}(\eta)<\infty$ implies $\rho=\eta$.
\end{itemize}
Let $U$ be a set of measurable functions $u:(0,T)\to X$, with a fixed $T>0$. Assume further that
\begin{equation}\label{hprossav}
\sup_{u\in U}\int_{0}^T\mathscr{F}(u(t))\,dt<\infty\quad \text{and}\quad \lim_{h\downarrow0}\sup_{u\in U}\int_{0}^{T-h}g(u(t+h),u(t))\,dt=0\,.
\end{equation}
Then $U$ contains an infinite sequence $(u_n)_{n\in\mathbb{N}}$ that converges in measure, with respect to $t\in(0,T)$, to a measurable $\tilde{u}:(0,T)\to X$, i.e.
\[
\lim_{n\to\infty}|\{t\in(0,T):\|u_n(t)-u(t)\|_X\ge\sigma\}|=0, \quad \forall \sigma>0.
\]
\end{thm}
\end{appendices}

\subsection*{Acknowledgements}
The authors would like to thank Anna Korba (ENSAE Paris) for her suggestions and comments on the contents of this manuscript. They were supported by the Advanced Grant Nonlocal-CPD (Nonlocal PDEs for Complex Particle dynamics: Phase Transitions, Patterns and Synchronization) of the European Research Council Executive Agency (ERC) under the European Union’s Horizon 2020 research and innovation programme (grant agreement No. 883363). JAC was also partially supported by the “Maria de Maeztu” Excellence Unit IMAG, reference CEX2020-001105-M, funded by MCIN/AEI/10.13039/501100011033/. JAC and AE were also partially supported by the EPSRC grant numbers EP/T022132/1 and EP/V051121/1. JSHW was supported by the Mathematical Institute Award at the University of Oxford. 

\bibliography{references}
\bibliographystyle{abbrv}

\end{document}